%% file: main.tex
\documentclass[11pt,english,leqno]{amsart}

\input{header.tex}

\usepackage{stmaryrd}
\usepackage[sans]{dsfont} 
\usepackage{bm} 
\usepackage{booktabs} 
\usepackage{enumitem}
\usepackage{csquotes}
\setlist[enumerate,1]{label=(\roman*)}

\definecolor{gris25}{gray}{0.60}\definecolor{Gray85}{gray}{0.85}\definecolor{Gray65}{gray}{0.65}

\usepackage{lastpage}
\usepackage{fancyhdr}
\newcommand{\on}{\operatorname}

\newcommand{\dR}{{\textrm{\scriptsize dR}}}
\newcommand{\sing}{\textrm{\scriptsize sing}}

\newcommand{\Gr}{{\mathrm{Gr}}}

\newcommand{\Fl}{{\mathrm{Fl}}}
\newcommand{\PG}{\mathrm{PG}}
\newcommand{\OS}{\on{{OS}}}
\newcommand{\redOS}{{\on{\overline{OS}}}}

\newcommand{\cone}{\mathrm{cone}\,}
\newcommand{\eps}{\epsilon}

\usepackage{titletoc}
\usepackage[explicit]{titlesec}
\newcommand\dotafter[1]{\if\relax\detokenize{#1}\relax{}\else {#1.}\fi}

\titleformat{name=\section}[block]{\centering\large\scshape}{}{0em}{\vspace*{0em}}[\thesection.~#1]
\titleformat{name=\section,numberless}[block]{\centering\large\scshape}{}{0em}{\vspace*{.2em}#1}[] 
\titleformat{\subsection}[runin]{\bfseries}{\thesubsection.~}{0em}{}[\dotafter{#1}]
\titlespacing{\subsection}{0pt}{*2}{1ex}
\titleformat{\subsubsection}[runin]{\itshape}{\thesubsubsection.~}{0em}{}[\dotafter{#1}] 

\titlecontents{section}[0em]{\addvspace{.5pc}\scshape}{\thecontentslabel\  }{}{\titlerule*[.8pc]{.}\contentspage}[]
\titlecontents{subsection}[1.5em]{}{\thecontentslabel\ }{}{\hfill\contentspage}
\titlecontents{subsubsection}[4.5em]{}{\thecontentslabel\ -\ }{}{\hfill\contentspage}

\usepackage{standalone}
\usepackage[pdftex,%
    hyperindex=true, bookmarks=true, bookmarksopen=true, bookmarksnumbered=true,
    pdfauthor={Ruizhen Liu},
    pdftitle={Vanishing theorems on wonderful varieties},
    pdfsubject={},
    pdfkeywords={},
    pdfstartview=FitH,
    colorlinks=true,
    hidelinks=true,
    linkcolor=reference, citecolor=citation, urlcolor=e-mail
]{hyperref} 

\usepackage{caption}
\usetikzlibrary{fit,backgrounds} 

\usepackage[
    bibencoding=utf8,%
    giveninits=true,
    bibstyle=ieee-alphabetic,
    citestyle=alphabetic,%
    sorting=nyt,%
    language=english,%
    isbn=false,%
    doi=false,%
    eprint=true,
    backref=true, 
    hyperref=true,
    url=false, 
    minbibnames=4,
    maxbibnames=4,
    dashed=true,
]{biblatex}

\AtBeginBibliography{\small}
\numberwithin{equation}{subsection} 

\title{Vanishing theorems on wonderful varieties}
\author{Ruizhen Liu}
\address{Department of Mathematics, University of Toronto, Toronto, ON M5S 2E4, Canada}
\email{ruizhen.liu@mail.utoronto.ca}
\thanks{The author is partially supported by a  University of Toronto Excellence Award.}


\begin{document}
\begin{abstract}
    We study vanishing theorems of tautological bundles in the sense of Berget--Eur--Spink--Tseng restricted to wonderful varieties. As an application, we prove a characteristic-independent analogue of Brieskorn's result on cohomology of arrangement complements, in addition to a comparison theorem between Orlik--Solomon algebra and the logarithmic de Rham cohomology of wonderful varieties. In a different direction, we extend a vanishing theorem of Borel--Weil--Bott type for tautological bundles. Finally, we reduce the weak version of White's basis conjecture to a problem about cohomology vanishing of tautological bundles.
\end{abstract}
\maketitle
\vspace{-3em}
\tableofcontents
\newpage 

\section{Introduction}


Let $\kk$ be a field.  Write $E =
\{0,\dots,n\}$ and let $L \subseteq \kk^E$ be a linear subspace. The purpose of this paper is to study the positivity of certain vector bundles attached to the linear subspace $L$.

\subsection{}
Our starting point is the classical work of Brieskorn~\cite{Brieskorn} from half a century ago. We consider the linear embedding $\P L \subseteq \P^n$ with a fixed system of homogeneous coordinates $ x_i, i\in E,$ on $\P^n$. Suppose that we are in the situation
\begin{equation}\label{eq:ll}
    \text{The subspace $L\subseteq \kk^E$ is not contained in any coordinate hyperplane of $\kk^E$}.\tag{LL}
\end{equation}
Then, denote by $H_i=\{ x_i=0\}$ the restriction of the $i$\/th
coordinate hyperplane to $\P L$; the $H_i$ give rise to a hyperplane arrangement on $\P L$, thanks to Condition~\eqref{eq:ll}. Additionally, we denote by \[\P L^\circ = \P L \setminus \cup_i H_i\] the arrangement complement. Given a projective hyperplane arrangement $\{H_i\}_i$, one can construct the reduced Orlik--Solomon algebra $\redOS^\bullet(L)$; this is a graded-commutative $\ZZ$-algebra whose generators are linear forms $f=\sum f_i \be_i$ on $\ZZ^E$ satisfying $\sum f_i = 0$.
Brieskorn's result is that:
\begin{quote}
    {When $\kk=\CC$, there is a graded-commutative algebra isomorphism \[
            \redOS^\bullet(L) \to H_\sing^\bullet (\P L); \quad f \mapsto  \text{the class of }\sum f_i d \log  x_i.
        \]
    Loosely speaking, any cohomology class in $\P L^\circ$ can be represented as global differential forms on $\P L^\circ$ with logarithmic poles at infinity.}
\end{quote}
The topological statement  begs the question: can the reduced Orlik--Solomon algebra for an arrangement over \textit{any} field $\kk$ afford a geometric interpretation? The first part of this paper is to provide an answer to this question, using algebraic geometry.

First, we wish to construe, algebraically, the Betti cohomology $H^\bullet_{\sing}(\P L^\circ)$. Given a linear subspace $L \subseteq \kk^E$ satisfying Condition~\eqref{eq:ll}, one can construct its \textit{wonderful variety} $W_L$ (with maximal building set) in the sense of de
Concini--Procesi~\cite{dcp}. The variety $W_L$ is a smooth projective compactification of the arrangement complement $\P L^\circ$, with a simple normal crossings divisor $D_L = W_L \setminus \P L^\circ$ as boundary. Considering the complex $\Omega^\bullet _{W_L}(\log D_L)$ of differential forms on $W_L$ with logarithmic singularities along $D_L$, we have the Hodge--de Rham spectral sequence,
\begin{equation}\label{eq:HodgeDeligne}
    E_1^{p,q} = H^q (W_L,\Omega_{W_L}^p(\log D_L)) \Rightarrow \HH^{p+q}(W_L,\Omega^\bullet_{W_L}(\log D_L)). \tag{Hdg}
\end{equation}
When $\kk = \CC$, the right-hand side equals the Betti cohomology $H^{p+q}_{\sing} (\P L^\circ)_\CC$, and the spectral sequence always degenerates at $E_1$, thanks to Deligne~\cite[Proposition~3.1.8]{deligneHodgeII}. Moreover, Brieskorn's result asserts that the canonical inclusion $E^{p,0}_1 \hookrightarrow H^p_{\sing} (\P L^\circ)_\CC$ is an isomorphism. Hence the logarithmic de Rham cohomology $\HH^\bullet(\Omega^\bullet_{W_L}(\log D_L))$ serves as a potential algebro-geometric surrogate for the Betti cohomology of $\P L^\circ$.

However, in positive characteristic, even the degeneration of the Hodge--de Rham spectral sequence is not guaranteed for general smooth varieties; see~\cite{Raynaud} for counterexamples and Remark~\ref{rm:DeligneIllusie} for comments. Nonetheless, our first result removes the characteristic-$0$ assumption on the $E_1$-degeneration of Sequence~\eqref{eq:HodgeDeligne}.

\begin{maintheorem}\label{thm:log-vanishing}
    Let $\kk$ be a field and $L \subseteq \kk^E$ be a linear subspace satisfying Condition~\eqref{eq:ll}. Let $W_L$ be the de
    Concini--Procesi wonderful variety introduced above, and $D_L$ be its boundary divisor. Then,
    \begin{enumerate}
        \item  The sheaves of differential forms with logarithmic poles along $D_L$ have higher vanishing cohomology; that is, at $E_1$ page of Sequence~\eqref{eq:HodgeDeligne}, we have \[
            E_{1}^{p,q} = H^q(\Omega_{W_L}^p(\log D_L)) = 0,\quad\text{for all $q > 0$.}\]
        \item There are canonical isomorphisms of graded-commutative algebra \[
                \redOS^\bullet(L) \otimes \kk \to H^0(\Omega_{W_L}^\bullet (\log D_L)) \to \HH^\bullet(\Omega_{W_L}^\bullet(\log D_L)),
            \] where the first map will be defined in \S\ref{sec:orlik-solomon} and the second map is the edge morphism of Sequence~\eqref{eq:HodgeDeligne}.
    \end{enumerate}
\end{maintheorem}
An important slogan from the theorem, therefore, would be that \textit{the logarithmic de Rham cohomology of the wonderful variety of a hyperplane arrangement depends only on its combinatorial type.} The result here is inspired by a number of previous works:
\begin{enumerate}[label={\arabic*.}]
    \item Over the complex numbers, this reformulation of Brieskorn's theorem also appeared in the work of Esnault--Schechtman--Viehweg~\cite[558]{ESV92}. However, our proof is purely algebraic, and bypasses the deletion-contraction sequence for Betti cohomology of $\P L^\circ$ used by Orlik--Solomon~\cite[\S5.4]{OrlikTerao}.
    \item For any field $\kk$, setting $p=\dim L-1$ in Statement~(i) recovers \cite[Corollary~1.7]{bestCohomology}.
    \item In characteristic $\ell>0$, one can consider the arrangement $\sca$ of all $\FF_\ell$-rational hyperplanes in $\P^r_{\overline{\FF}_\ell}$. The matroid for $\sca$ is the projective geometry $\PG(r,\ell)$. Langer observed that the wonderful variety $W_\sca$ of $\sca$ agrees with a compactified Deligne--Lusztig variety~\cite[Proposition~7.1]{LangerDeligneLusztig}. In this case, Statement~(i) recovers a result of Gro{\ss}e-Kl{\"o}nne~\cite[Theorem~2.3]{GrosseKloenne05}. In \textit{Ibid.}\@\xspace, Gro{\ss}e-Kl{\"o}nne also constructed a basis for the logarithmic de Rham cohomology $\HH^\bullet(\Omega_{W_\sca}^\bullet(\log D_\sca))$ in terms of certain unipotent translates of global logarithmic forms, which in fact agrees with the known reduced \textit{nbc}-monomial basis for the reduced Orlik--Solomon algebra $\redOS^\bullet(\sca)$---this is an illustration of Statement~(ii).
\end{enumerate}
In characteristic $0$, Statements (i) and (ii) of Theorem~\ref{thm:log-vanishing} hold for any simple normal crossings compactification of $\P L^\circ$. Unfortunately, without strong resolution of singularities, we can only generalise Theorem~\ref{thm:log-vanishing} to any compactification that is related $W_L$ via blow-ups and blow-downs along smooth centres; see Corollary~\ref{cor:factorisation} for a precise statement. However, compactifications arising this way already include de Concini--Procesi wonderful varieties attached to arbitrary building sets. The latter includes, for example, the Deligne--Mumford--Knudsen moduli space $\Mdmzero$ of stable rational curves with $m$ marked points. We will deduce an integral version of Theorem~\ref{thm:log-vanishing} for $\Mdmzero$; see Corollary~\ref{cor:MdmzeroSpecZ}.

\subsection{}
The first statement of Theorem~\ref{thm:log-vanishing} is a special case of a more general result involving certain torus-equivariant vector
bundles on the permutohedral toric variety $X_E$. One can associate to a
linear space $L\subseteq \kk^E$ its \emph{tautological sub- and quotient bundles}, denoted respectively by $\scs_L$ and
$\scq_L$, introduced by Berget, Eur, Spink, and Tseng~\cite{best}. These bundles fit into a short exact sequence \[
    0\to \scs_L \to \kk^E \otimes \sco_{X_E} \to \scq_L \to 0.
\]

Consider the direct sum  $\sce_{L} \coloneqq \scs_{L} \oplus \scq_{L}
$. We prove:
\begin{maintheorem}\label{thm:SQvanishing}
    Let $\kk$ be a field and $L\subseteq \kk^E$ a linear subspace. Then all exterior powers of $\sce_{L}$ have vanishing higher cohomology: \[
        H^q(X_E, \bigwedge^d \sce_{L}) = 0,\quad\text{for all $q>0$.}
    \] Equivalently, we have that $H^{i}(X_E, \bigwedge^p \scs_{L} \otimes \bigwedge^q \scq_{L})=0$ for all $i>0$ and $p,q\ge 0$.
\end{maintheorem}
Taking $p=0$ or $q=0$ in the second assertion, we recover \cite[Theorem~1.2]{bestCohomology}, whose proof here seemed incomplete; see Remark~\ref{rm:eur-gap} for full discussions. On the other hand, combining a Koszul resolution of $\sco_{W_L}$ with the fact that the bundle $\scs_L^\vee|_{W_L}$ is an extension of the trivial bundle $\sco_{W_L}$ by the log cotangent bundle $\Omega^1_{W_L}(\log D_L)$ allows us to deduce
Theorem~\ref{thm:log-vanishing}.

Theorem~\ref{thm:SQvanishing} is obtained from a generalisation of arguments of
Eur in \cite{bestCohomology}. The strategy is to show that under a deletion map $f\colon X_E \to X_{E\setminus
n}$, the derived pushforward of $\bigwedge^d \sce_L$ along $f$ is concentrated in degree $0$ and
obeys a recursive pattern.

In another direction, we improve Statement~(i) of Theorem~\ref{thm:log-vanishing} in characteristic $0$. Our line of
argument here builds on a series of papers of Berget and Fink~\cite{BergetFinkMatrixOrbit,BergetFinkChow,BergetFinkExternalActivity}.
A key input is the \textit{geometric technique}, introduced by Kempf~\cite{Kempf} and popularised by Weyman~\cite{Weyman03}.
Our result here is

\begin{maintheorem}\label{thm:bwb-vanishing}
    Fix a linear subspace $L$ that satisfies Condition~\eqref{eq:ll}, in addition to $k$ linear subspaces $J_1,\dots,J_k \subseteq \CC^E$.
    For every $i=1,\dots,k$, let $\lambda_i$ be a partition. Then we have \[
        H^p(W_L, \schur^{\lambda_1} \scs_{J_1}^\vee \otimes \cdots \otimes \schur^{\lambda_k} \scs_{J_k}^\vee ) = 0,\quad\text{for all $p > 0$.}
    \]
\end{maintheorem}

In particular, taking $k=1$ and $J_1 = L$, one recovers Theorem~\ref{thm:log-vanishing} over the complex numbers.
If $L = \CC^E$, then $W_L=X_E$, then the assertions  recovers \cite[Theorem~5.1]{BergetFinkKtheory}; moreover, the same assertions still hold if one replaces $\scs_{L}^\vee$ by $\scq_{L}$, thanks to the standard Cremona transformation.

In general, when $W_L$ is not the permutohedral variety $X_E$, we do not expect the same vanishing to occur after swapping $\scs_{J_i}^\vee$ with $\scq_{J_i}$; for example, it can happen that $H^1(W_L,\scq_L) \neq 0$.\footnote{We thank Matt Larson for pointing this out. The idea here is to notice that $\scq_L|_{W_L}$ is the normal bundle of $W_L$ in $X_E$, and Mn{\"e}v's Universality says there exists a linear space $L\subseteq \CC^E$ that represents a singular point in the realisation space of its matroid; see the recent work of Eur, Fink, and Larson~\cite[Example~5.3]{eur2025vanishingtheoremsmatroids} for a further discussion.}
However, our next result shows that if one considers twisting by determinants of the bundles $\scq_{J_i}$, then in fact higher cohomology vanishes:

\begin{maintheorem}\label{thm:manivel-vanishing}
    Fix linear subspaces $L,J_1,\dots,J_k\subseteq \CC^E$ over the complex numbers as in Theorem~\ref{thm:bwb-vanishing}. Fix, additionally, a nonnegative number $d_i$ for every $i = 0,\dots,k$. If for every $i$, $\lambda_i$ is a partition with at most $d_i$
    parts, then,
    \begin{align*}
        H^p(  W_L, &\, \schur^{\lambda_0} \scq_L \otimes \schur^{\lambda_1} \scq_{J_1}\otimes \cdots\otimes\schur^{\lambda_k} \scq_{J_k}
            \otimes \\
        &  \det \scq_L^{\otimes d_0+1} \otimes \det\scq_{J_1}^{\otimes d_1}\otimes \cdots \otimes \det\scq_{J_k}^{\otimes d_k} ) = 0, \quad\text{for all $p> 0$.}
    \end{align*}
\end{maintheorem}

Theorem~\ref{thm:bwb-vanishing} may be interpreted as a vanishing theorem of Borel--Weil--Bott type, and Theorem~\ref{thm:manivel-vanishing} as one of Manivel type; compare with~\cite{ManivelVanishing}. It should be noted that the bundle $\scq_L$ is seldom ample on $X_E$, so one cannot deduce the statements  by directly applying vanishing theorems for ample vector bundles by Manivel and others.

Here, the Schur functors are used to simplify notations. In fact, we will mostly deal with tensor products of symmetric algebras of $\scs_{J_i}$ and then apply Pieri's rules in the end.

\subsection{}

Finally, we turn to measures of complexity of certain toric ideals \emph{via} tautological bundles.
To a rank-$r$ matroid $\rmm$ with ground set $E$,
we can associate its toric ideal $I_\rmm$, defined to be the kernel of the ring map \[
    \kk[x_B,\; \text{$B$ a basis of $\rmm$}] \to \kk[y_0,\dots,y_n];\quad x_B \mapsto \prod_{i\in B} y_i.
\]
The ideal $I_\rmm$ cuts out the affine cone of the toric variety $X_{\rmb(\rmm)}$ associated to the base polytope $\rmb(\rmm)$ of the matroid $\rmm$.
The weakest version of White's basis conjecture is the following; see~\cite{WhiteBasisMonomial}:
\begin{conjecture}[White]
    The homogeneous ideal $I_\rmm$ is generated by quadrics.
\end{conjecture}

In the case where the matroid $\rmm$ is given by a linear subspace $L\subseteq \kk^E$ of dimension $r$, the statement that toric ideal $I_{\rmm^\perp}$ of the dual matroid $\rmm^\perp$ has no cubic minimal generators inspires the following stronger prediction
\begin{conjecture}[Conjecture~\ref{conj:white-deg3}]
    Let $L$ be a linear subspace of dimension $r$ whose matroid $\rmm(L)$ is connected. Then the vector bundle
    \[
        \bigwedge^{r} \sce_L^\vee \otimes \bigwedge^{r} \sce_L^\vee \otimes \det \scq_L
    \] has vanishing $H^1$, where, as before, we set $\sce_L = \scs_L \oplus \scq_L$.
\end{conjecture}
If $\on{char} \kk \neq 2$ and the conclusion above is satisfied for the linear subspace $L$, then the $I_{\rmm^\perp}$ toric ideal has no minimal generator in degree $3$, where $\rmm$ is the matroid associated to the linear subspace $L$. Despite the similarity with the statement of Theorem~\ref{thm:SQvanishing}, the reader will see the proof strategy thereof does not readily generalise to give higher cohomology vanishing for $\bigwedge^{r} \sce_L^\vee \otimes \bigwedge^{r} \sce_L^\vee \otimes \det \scq_L$.

Concerning organisation: we start in \S\ref{sec:prelim} with generalities on vector bundles we
consider. Section~\ref{sec:SQproof} is devoted to the proof of Theorem~\ref{thm:SQvanishing}. In \S\ref{sec:orlik-solomon}, we study the Orlik--Solomon algebra associated to a linear subspace using the logarithmic Hodge--de Rham spectral sequence, proving Statements~(i) and (ii) of Theorem~\ref{thm:log-vanishing} in Theorem~\ref{thm:log-vanishing-restate} and Theorem~\ref{thm:ring-iso}, respectively. In
\S\ref{sec:kempf}, we adapt the construction of Kempf--Weyman to study collapsings of vector bundles,
proving in particular Theorem~\ref{thm:bwb-vanishing} and Theorem~\ref{thm:manivel-vanishing}. Finally, Section~\ref{sec:white} is devoted to the weak White's conjecture and its strengthening.

\subsection*{Acknowledgements}
The author thanks Hunter Spink for numerous suggestions and constant encouragement. Very many thanks to Andrew Berget and Alex Fink for many discussions: they were collaborators at many stages of this project, and conversations with them led to the material in Section~\ref{sec:kempf} and Section~\ref{sec:white}. Many thanks to Chris Eur and Matt Larson for valuable comments and suggesting the connections to Orlik--Solomon algebras, which led to Section~\ref{sec:orlik-solomon}. We are particularly indebted to Matt Larson for pointing out a gap in the argument towards Theorem~\ref{thm:SQvanishing}. This project started in Spring 2025 during the Thematic Program in Commutative Algebra and Applications at the Fields Institute. Many thanks to the institution and the organisers of the programme.

\section{Preliminaries}\label{sec:prelim}
We refer to \cite{CLS} for basics on toric varieties and \cite{Oxley} for matroid theory.
\subsection{The permutohedral variety}
Recall that $E = \{0,\dots,n\}$. For $S\subseteq E$, we denote by $\be_S$ the sum of basis vectors $\sum_{i\in S}\be_i$ in $\RR^E$. Let $N_E$ be the lattice $\ZZ^E/\be_E$, and denote by $\bar\be_S$ the image of $\be_S$ in $N_{E,\RR}$.

An \emph{ordered set partition} $\scf$ of $E$ is a sequence $(F_1\,\dots,F_\ell)$ of nonempty subsets that partition $E$.
Given an ordered partition $\scf$ of $E$, we also associate to it a flag of subsets $\{G_i(\scf)\}$ of $E$
defined by $G_i(\scf) = \bigcup_{j\le i} F_j$.

\begin{definition}
    The \emph{permutohedral fan} $\Sigma_E$ is a simplicial fan in $N_{E,\RR}$ whose $k$-dimensional cones are \[
        \sigma_\scf = \cone\{\bar\be_{G_1(\scf)},\cdots,\bar\be_{G_k(\scf)}\},
    \] for all ordered set partitions $\scf$ with $k$ parts. The \emph{permutohedral variety} $X_E$ is the toric variety associated to the fan $\Sigma_E$.
\end{definition}

We identify $\ZZ^E$ as the cocharacter lattice of the algebraic torus $T=\Gm^E$ so that $N$
is the cocharacter lattice of $\P T = T/\Gm$, the big open torus of $X_E$.

We list some properties of the toric variety $X_E$:
\begin{enumerate}
    \item All cones in $\Sigma_E$ are spanned by a subset of a basis for the lattice $N_E$, so the toric variety $X_E$ is smooth. For $t\in T$, denote by $\ol t$ its image in $\P T$.
    \item By definition, every set partition $\scf$ indexes a torus invariant stratum in $X_E$, given by $T$-orbit of a
        distinguished point, which we denote by $p_\scf$. Also denote by $Z_\scf$ the closure of the stratum. By the orbit-cone stratification, every point in $X_{E}$ is given by $t.p_{\scf}$ for some $t \in T$ and some ordered set partition $\scf$ of $E$.
    \item The variety $X_E$ affords a toric involution map $\crem\colon X_E \to X_E$, known as the \textit{standard Cremona transform} and determined by $\ol t \mapsto \ol t^{-1}$ in the big open torus $\P T$.
    \item
        Finally, the fan $\Sigma_E$ can be obtained as the first barycentric subdivision of the fan for $\P^n$. Thus one can realise $X_E$ as a sequence of blow-ups;
        \begin{equation}\label{eq:blowup-perm}
            \pi_E\colon X_E = X_{n-1}\to X_{n-2}\to \cdots \to X_1 \to X_0 = \P^n,
        \end{equation}
        where $X_{i+1} \to X_i$ is given by blowing up (the strict transform of) $i$-dimensional coordinate
        subspaces.
\end{enumerate}
The last point  lets us streamline the definition of wonderful compactifications of hyperplane arrangements (with maximal building set) in the sense of de Concini and Procesi~\cite{dcp}.
\begin{definition}
    For $L\subseteq \kk^E$ a linear space satisfying Condition~\eqref{eq:ll}, the \emph{wonderful variety} $W_L$ of $L$ is given by the strict transform $\widetilde{\P L}\subseteq X_E$.
\end{definition}
By construction, the wonderful variety $W_L$ is also an iterated blow-up of the linear subspace $\P L$, and the blow-down map $\pi_E$ restricts to an isomorphism on the intersection $\P L^\circ = \P L \cap \P T$.
The boundary $D_L = W_L \setminus \P L^\circ$ is a simple normal crossings divisor~\cite[\S3.1]{dcp}.

For an ordered set partition $\scf=\{F_1, \cdots, F_\ell\}$ of $E\setminus n$, let
$\scf(i)$ be the ordered partition of $E$ obtained by appending $n$ to $F_i$, and
$\scf^i$ be the ordered partition of $E$ obtained by inserting $\{n\}$ between $F_{i}$ and $F_{i+1}$.

Let $T'$ be the subtorus $\Gm^{E\setminus n} \subseteq \Gm^E$, and consider the projection $f^\circ\colon \Gm^E\to \Gm^{E\setminus n}$ away from the $n$th factor. This gives a rational map $\P^n \dashrightarrow \P^{n-1}$ given by $[x_0:\dots:x_n]\mapsto [x_0:\dots:x_{n-1}]$, whose base locus is clearly resolved by the blow-up $\pi_E$. In fact, the resolution factors through $X_{E\setminus n}$:
\begin{proposition.definition}\label{propdet:deletion}
    The map $f^\circ$ extends to a flat, projective toric morphism \[f\colon X_E \to X_{E \setminus n}.\]
    Moreover, given $t\in T'$ and $\scf$ an ordered set partition of $E\setminus n$ with $\ell$ parts, the fibre $C = f^{-1}(t. p_\scf)$ is a chain $\bigcup_{1\le i \le \ell} C(t,i)$ of translated rational curves $C(t,i) = (t_0,\dots,t_{n-1},1).\P^1_{\scf(i)}$, with the component $\P^1_{\scf(i)}$ given by \[
        \P^1_{\scf(i)} = p_{\scf^i} \sqcup p_{\scf^{i+1}} \sqcup
        \Gm.p_{\scf(i)} \simeq 0 \sqcup \infty \sqcup \Gm.
    \]
\end{proposition.definition}
We refer to~\cite[Corollary~2.4]{bestCohomology} for a proof. The reader will see that the flatness of $f$ plays an important role in the proof of Theorem~\ref{thm:SQvanishing}. In \cite{LosevManin}, the authors considered $X_E$ as a moduli space of weighted stable rational curve with $n+1$ points marked by $E$ with weights $(1,1,\eps,\dots,\eps)$ with infinitesimal $\eps>0$; in this case, the map $f\colon X_E \to X_{E\setminus n}$ corresponds to forgetting the weight $\eps$ point marked by $n$.

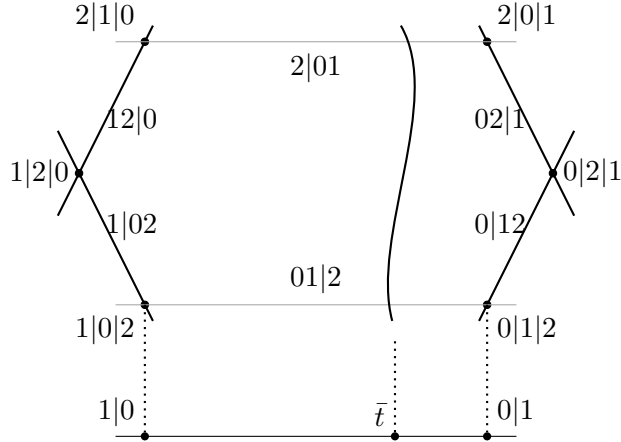
\begin{figure}
    \centering
    \begin{tikzpicture}[scale=0.7]
        \node [] (0) at (3.5, 1) {$02|1$};
        \node [] (1) at (3.5, -1) {$0|12$};
        \node [] (3) at (-3.5, -1) {$1|02$};
        \node [] (4) at (-3.5, 1) {$12|0$};
        \node [] (5) at (0, 2) {$2|01$};
        \node [] (6) at (-4, -5) {};
        \node [] (7) at (4, -5) {};
        \node [] (21) at (-4, 2.5) {};
        \filldraw[black] (-4.5, 0) circle (2pt) node[left]{$1|2|0$};
        \node [] (23) at (-4, -2.5) {};
        \node [] (24) at (4, -2.5) {};
        \filldraw[black] (4.5, 0) circle (2pt) node[right]{$0|2|1$};
        \node [] (26) at (4, 2.5) {};
        \node [] (27) at (-3, 3) {};
        \node [] (28) at (-5, -1) {};
        \node [] (29) at (-5, 1) {};
        \node [] (30) at (-5, -1) {};
        \node [] (31) at (-3, -3) {};
        \node [] (32) at (3, 3) {};
        \node [] (33) at (5, -1) {};
        \node [] (34) at (5, 1) {};
        \node [] (35) at (3, -3) {};
        \filldraw[black] (3.25, 2.5) circle (2pt) node[above right]{$2|0|1$};
        \filldraw[black] (-3.25, 2.5) circle (2pt) node[above left]{$2|1|0$};
        \filldraw[black] (3.25, -2.5) circle (2pt) node[below right]{$0|1|2$};
        \filldraw[black] (-3.25, -2.5) circle (2pt) node[below left]{$1|0|2$};
        \node [] (38) at (0, -2) {$01|2$};
        \node [] (39) at (1.5, 3) {};
        \node [] (40) at (1.5, -3) {};
        \node [circle,fill=black,scale=0.3] (41) at (1.5, -5) {};
        \node [above left] at (1.5, -5) { $\bar t$};
        \node [circle,fill=black,scale=0.3] (42) at (3.25, -5) {};
        \node [above right] at (3.25, -5) { ${0|1}$};
        \node [circle,fill=black,scale=0.3] (43) at (-3.25, -5) {};
        \node [above left] at (-3.25, -5) { ${1|0}$};

        \draw [in=180, out=0] (6) to (7);
        \draw [thick] (31) to (29);
        \draw [thick] (27) to (30);
        \draw [thick] (32) to (33);
        \draw [thick] (34) to (35);
        \draw [Gray65] (24) to (23);
        \draw [Gray65] (21) to (26);
        \draw [thick,in=105, out=-60, looseness=0.75] (39) to (40);
        \draw[thick, dotted]  (-3.25, -2.5) to (43);
        \draw[thick, dotted] (40) to (41);
        \draw[thick, dotted]  (3.25, -2.5) to (42);
    \end{tikzpicture}
    \caption{An illustration of the map $f:X_{0,1,2} \to X_{0,1}$; strata of the permutohedral varieties  labelled by set partitions.}
    \label{fig:A2toA1}
\end{figure}

\begin{example}
    Consider the deletion map $f\colon X_E \to X_{E\setminus n}$ in the case when $E = \{0,1,2\}$ and $n=2$. In Figure~\ref{fig:A2toA1}, we describe $f$ on the level of toric strata. The six straight edges forming a hexagon constitute the toric boundary of $X_{0,1,2}$. 
    The thickened edges in Figure~\ref{fig:A2toA1} are illustrations of various fibres of $f$. The curved edge is the fibre over a point $\bar t = \bar t p_{01}$ in the big open torus of $X_{0,1}$, which has two points given by intersecting with the two invariant curves indexed by $2|01$ and $01|2$, respectively. On the rightmost side of the picture, we have the ordered set partition $0|1$ of $\{0,1\}$ corresponds to a torus fixed point of $X_{0,1}$ and is indicated by the vertex $0|1$. The fibre over the point $0|1$ is the gluing of two torus invariant curves---$\P^1_{02|1}$ and $\P^1_{0|12}$---along the $T$-fixed point $p_{0|1|2}$. Similarly, on the leftmost side, the union of invariant curves $1|02$ and $12|0$ gives the fibre over the fixed point $1|0$ in $X_{0,1}$.
\end{example}

\subsection{Tautological bundles of linear spaces}\label{subsec:intro-to-best}
Recall that $\kk$ is a field of arbitrary characteristic.
Let $\kk^E_\inv$ be the dual of the standard representation of the split $\kk$-torus $T = \Gm^E$. This also induces a $T$-action on $\Gr(r,\kk^E)$ given by $t.[L] = [t^{-1}L]$ for $t\in T$.

The construction of tautological bundles rest upon the following classical result; see~\cite[Corollary~2.4]{GGMS}, \cite{WhiteBasisMonomial} and~\cite[Lemma~3.5]{best}:
\begin{proposition}\label{prop:ggms}
    Given a linear subspace $L \subseteq \kk^E$ of dimension $r$, there exists a unique $T$-equivariant map
    \begin{equation}\label{eq:pre-log-gauss}
        \phi_L\colon X_E \to \ol{T.[L]} \subseteq \Gr(r,\kk^E),
    \end{equation} mapping $\bar{t} \in \P T$ to $[t^{-1}L]$. Moreover, the image $\ol{T.[L]}$ of $\phi_L$ equals the toric variety $X_{\rmb(L)}$ associated to the base polytope $\rmb(L)$ of the matroid $\rmm(L)$.
\end{proposition}

\begin{definition}
    The \emph{tautological sub- and quotient bundle} of the linear subspace $L$ are respectively \[
        \scs_L \coloneqq \phi_L^\ast \scs_{\Gr(r,E)}, \text{ and } \scq_L \coloneqq \phi_L^\ast \scq_{\Gr(r,E)},
    \] where $\scs_{\Gr(r,E)}$ (resp. $\scq_{\Gr(r,E)}$) are the tautological sub- (resp. quotient) bundle of $\Gr(r,\kk^E)$.
\end{definition}
The equivariant Euler sequence for the Grassmannian $\Gr(r,\kk^E)$ pulls back a short exact sequence \[
    0\to \scs_L \to \kk^E_\inv \otimes \sco_{X_E} \to \scq_L \to 0.
\]
As a matter of notation, denote by $\alpha_E$ the pullback along $\pi_E$ of a general hyperplane in $\P^n$, and $\beta_E$ the Cremona pullback $\crem^\ast \alpha_E$. In the language of tautological bundles, we have $\sco_{X_E}(\alpha_E) = \scq_H$ and $\sco_{X_E}(\beta_E) = \scs_\ell^\vee$, where $\ell,H$ are respectively a general line and a general hyperplane in $\kk^E$.
\begin{remark}
    The divisors $\alpha_\rmm$ and $\beta_\rmm$ defined in \cite[Definition~5.7]{AHK}---which these authors used to deduce log-concavity of the reduced characteristic
    polynomial of a matroid---are the restrictions of $\alpha_E$ and $\beta_E$, respectively, to a subfan $\Sigma_\rmm$ of $\Sigma_E$ called the Bergman fan of the matroid $\rmm$.
\end{remark}

The next proposition relates the tautological bundles and the wonderful variety $W_L$ in the two ways; see \cite[Theorem~7.10 and Theorem~8.8]{best} for details.
\begin{proposition}\label{prop:degeneracy-locus}
    Assume $\kk$ is algebraically closed. Let $L\subseteq \kk^E$ be a linear subspace that satisfies Condition~\eqref{eq:ll}.
    \begin{enumerate}
        \item
            The wonderful variety $W_L$ is the vanishing locus of the section $s\in \Gamma (\scq_L)$ given by the image of $s'=(1,\dots,1)\in H^0(\sco_{X_E}^E)$. That is, we have a Cartesian square \[
                \begin{tikzcd}
                    {\tot \scs_L} & {\A^E_{X_E} } \\
                    {W_L} & {X_E}
                    \arrow[from=1-1, to=1-2]
                    \arrow[from=2-1, to=1-1]
                    \arrow["\lrcorner"{anchor=center, pos=0.125, rotate=90}, draw=none, from=2-1, to=1-2]
                    \arrow[from=2-1, to=2-2]
                    \arrow["s'"', from=2-2, to=1-2]
                \end{tikzcd}
            \] where the top horizontal map is given by inspecting the total space of the embedding $\scs_L \subseteq \sco^E_{X_E}$.
        \item Consider the sheaf $\Omega^1_{W_L}(\log D_L)$ of differential $1$-forms on $W_L$ with logarithmic poles along $D_L$. Contracting the section $(1,\dots,1) \in H^0(W_L,\scs_L)$ against $\scs_L^\vee|_{W_L}$ induces an Euler-like short exact sequence
            \begin{equation}\label{eq:Euler-seq}
                0\to \Omega^1_{W_L}(\log D_L) \to \scs_L^\vee|_{W_L} \to \sco_{W_L} \to 0.
            \end{equation}
    \end{enumerate}
\end{proposition}

The reader will see in \S\ref{sec:orlik-solomon} that the short exact sequence  can indeed be construed in terms of logarithmic derivations.
The first assertion in Proposition~\ref{prop:degeneracy-locus} also implies:
\begin{corollary}\label{prop:Koszul-W}
    Keeping the notation and assumptions as in Proposition~\ref{prop:degeneracy-locus}. We have a Koszul resolution
    \begin{equation}
        0 \to \det \scq^\vee_L \to \cdots \to \bigwedge^2 \scq^\vee_L \to \scq^\vee_L \to
        \sco_{X_E} \to \sco_{W_L} \to 0.
    \end{equation}
    In particular, we have canonical bundle $\omega_{W_L}$ is given by $(\omega_{X_E} \otimes \det \scq_L)|_{W_L}$.
\end{corollary}

The Koszul resolution above, combined with the following homological algebra fact, allows us to pass vanishing statements from the permutohedral variety $X_E$ to wonderful varieties $W_L$.

\begin{fact}[{\cite[Proposition~B.1.2]{LazarsfeldPositivityI}}]\label{fact:lazarsfeld}
    Let $X$ be a finite-type $\kk$-scheme and $\scg$ be a coherent sheaf on $X$ that is quasi-isomorphic to a bounded complex $\scg^\bullet$ of coherent sheaves. For an integer $k\ge 0$, then $H^k(\scg)=0$ as long as \[
    H^a(\scg^b) = 0,\quad\text{for $a,b$ such that $a+b=k$.} \]
\end{fact}
\subsubsection{Matroidal operations and tautological bundles}
During the proof of Theorem~\ref{thm:SQvanishing}, the following pieces of notation will be useful. Given a linear subspace $L\subseteq \kk^E$ and a subset $S \subseteq E$, we denote by $L / S$ the intersection $L\cap \kk^{E \setminus S} \times \{0\}^S$, understood as a subspace of $\kk^{E\setminus S}$; by $L \setminus S$ the image of $L$ under the projection $\kk^E \to \kk^{E\setminus S}$; and finally, by $L|S = L \setminus (E \setminus S)$.
In matroid theory literature, these operations are respectively contraction, deletion, and restriction.

Additionally, we denote by $L^\perp$ the linear subspace $(\kk^E/L)^\vee \subseteq (\kk^E)^\vee \simeq \kk^E$. The linear subspace $L^\perp$ gives a realisation of the dual matroid of $L$. We also note that the tautological bundles are well-behaved under duality: we have $\scq_{L^\perp} \simeq \crem^\ast \scs_L^\vee$ and similarly for $\scs_{L^\perp}$.

We say an element $e\in E$ is a \textit{loop} (resp. \textit{coloop}) of $L$ if $L$ is contained in the coordinate subspace $\kk^{E\setminus e} \times \{0\}^{\{e\}}$ (resp. $L$ contains the coordinate line $\kk^{\{e\}}$). In particular, a linear subspace $L\subseteq \kk^E$ satisfies Condition~\eqref{eq:ll} if and only if it has no loops.

Finally, we say a word about restricting tautological bundles to torus invariant subvarieties in $X_E$. These pieces of notation will only make brief appearances in \S\ref{subsec:restriction-to-fibre}. Given an ordered set partition $\scf$, we denote by
\begin{equation}\label{eq:AKdecomp}
    L_{\scf} = \bigoplus_i\, L|G_i(\scf)/G_{i-1}(\scf) \subseteq \bigoplus_i \kk^{F_i}
\end{equation}
the direct sum decomposition into minors that arises from the flag $\{G_i(\scf)\}$.
\begin{proposition}\label{prop:restriction-to-subvar}
    Under the $T$-equivariant identification, $Z_\scf\simeq X_{F_1} \times \cdots X_{F_\ell}$,
    where $T$ acts on each factor in the right-hand term by the obvious projection $T \to
    \Gm^{F_{i}}$, there is a $T$-equivariant isomorphism \[
        \scs_L|_{Z_\scf} \simeq \boxplus_{i=1}^\ell \scs_{L_{G_i/G_{i-1}}}.
    \]
\end{proposition}
We refer to \cite[Proposition~3.6 and Proposition~5.3]{best} and \cite[Lemma~2.5]{bestCohomology} for a proof.

\section{Characteristic-free vanishing results \textit{via} pushforwards}\label{sec:SQproof}

This section is mainly devoted to the proof of Theorem~\ref{thm:SQvanishing}. It will be a generalisation of arguments of Eur in \cite{bestCohomology}. At the end of the section, we also deduce a Tutte polynomial identity, and a special case of Theorem~\ref{thm:bwb-vanishing} for hook-shaped partitions.

Throughout the section, we will assume $\kk$ is algebraically closed. The case for arbitrary fields follows from passing to an algebraic closure, thanks to flat base change~\cite[\href{https://stacks.math.columbia.edu/tag/02KH}{Lemma 02KH}]{stacks-project}.

The strategy towards Theorem~\ref{thm:SQvanishing} consists of three steps:
\begin{enumerate}[label={\arabic*.}]
    \item First of all, we study a variant of tautological bundles of linear spaces, defined on $\P^1$. On a single rational curve, we know \emph{a priori}, thanks to Birkhoff--Grothendieck, that the tautological bundles $\scs_L$ and $\scq_L$ restricted to an irreducible rational curve split into direct sums of line bundles. As a key enhancement, the work of Eur~\cite{bestCohomology} determines the line bundle summands. Our plan in \S\ref{subsec:bestP1} is to work with these summands and to control cohomology groups of tautological bundles.
    \item For the second step, we recall from Proposition/Definition~\ref{propdet:deletion} that the deletion map $f$ is flat and has fibres given by chains of $\P^1$\/s. We will introduce some auxiliary vector bundles $\scl_1,\scn_1,\scl_2,\scg$ on $X_E$. The vector bundles introduced in Step~1 will turn out to coincide with the tautological bundles $\scs_L$ and $\scq_L$ on $X_E$ restricted to a fibre. Therefore, the theorem of cohomology and base change lets us kill higher direct images of the corresponding vector bundles on $X_E$; see Lemma~\ref{lem:derivedVanshings} for precise statements. This will be carried out in \S\ref{subsec:restriction-to-fibre}.
    \item Finally, the main content of \S\ref{subsec:final-computation} is an inductive procedure to derive vanishing statements for exterior power of $\sce_L$. Running a Leray spectral sequence and applying the previous step, we know that it is enough to control the higher cohomology of the pushforward $f_\ast \bigwedge^d \sce_L$. To this end, we will consider a map \[
            \tau\colon f^*(\sce_{L/n} \oplus \sco_{X_{E\setminus n}})\to \sce_L.
        \]
    We use the higher direct image vanishing results from the previous step to compute explicitly the pushforward of exterior powers of $\scs_L$ and $\scq_L^\vee$ in Lemma~\ref{lem:pushS}. This will assists us with understanding kernels and cokernels of the pushforward of various exterior powers $ \bigwedge^d \tau$; see Lemma~\ref{lem:pushforward-L1-wedgeQ} and Corollary~\ref{cor:pushforward-L1-wedgeG}. Once this is done, we can leverage the inductive hypothesis to control the cohomology of these kernels and cokernels of $f_\ast \bigwedge^d \tau$, which in turn gives control over the cohomology of $f_\ast \bigwedge^d \sce_L$.
\end{enumerate}
\begin{remark}\label{rm:eur-gap}
    To obtain the higher cohomology vanishing for $\bigwedge^d \scq_L$ for $d\ge 0$, \cite{bestCohomology} misses out Step~3 and claims the direct images $f_*\bigwedge^d\scq_L$ are given by the direct sum of exterior powers of $\scq_{L/n}$ and $\scq_{L\setminus n}$. A similar assertion is made for exterior powers of $\scs_L$ to obtain higher cohomology vanishing for $\bigwedge^d \scs_L$;  \cite[Theorem~1.2]{bestCohomology}. However, the argument given there is incomplete.
    
    We now explain the gap in their argument. In the case when $n$ is not a loop or a coloop, \cite[Theorem~1.6]{bestCohomology} claims that $f_\ast \bigwedge^d \scq_L$ is the vector bundle and isomorphic to $\bigwedge^{d-1} \scq_{L\setminus n} \oplus \bigwedge^{d}\scq_{L/n}$. The strategy there is to carry out Step~1 and Step~2 to show that the $f_\ast \bigwedge^d \scq_L$ is a vector bundle with fibre isomorphic to  $\bigwedge^{d-1} \scq_{L\setminus n} \oplus \bigwedge^{d}\scq_{L/n}$, but this only implies the former is a vector bundle of the same rank as the latter. An identical gap occurs in \cite[Theorem~1.5]{bestCohomology}, where it is announced that $f_\ast \bigwedge^d \scs_L$ is isomorphic to $\bigwedge^d \scs_{L/n}$, when $n$ is not a coloop, whereas the proof only yields that these vector bundles share the same rank. 

    Our methods are heavily inspired by that of \cite{bestCohomology}, but our argument requires some new computation of homomorphisms between various vector bundles. In particular, we will produce a map from $\bigwedge^d \scs_{L/n}$ to the direct image $f_\ast\bigwedge^d \scs_L$ and salvage \cite[Theorem~1.5]{bestCohomology} in Lemma~\ref{lem:pushS}. It turns out these exterior powers of $\scs_L$ are the only vector bundles whose direct images we need explicit descriptions for. Using this, we deduce Theorem~\ref{thm:SQvanishing}, which is sufficient to salvage the vanishing statements for exterior powers of $\scs_L$ and $\scq_L$ from \cite[Theorem~1.2]{bestCohomology}.
    Moreover, by some homological algebra manipulations, we show that any hook-shaped Schur powers of $\scq_L$ has vanishing higher cohomology in Proposition~\ref{prop:bwb-hook-charp}, salvaging in particular the vanishing statements for symmetric powers of $\scq_L$ from \cite[Theorem~1.3]{bestCohomology}. Finally, while we cannot fully salvage \cite[Theorem~1.6]{bestCohomology}, we will show in Appendix~\ref{appendix:exterior-quotient} that the space of global sections $f_\ast \bigwedge^d \scq_L$ does have the same dimension as that of $H^0(\bigwedge^{d-1} \scq_{L\setminus n} \oplus \bigwedge^{d}\scq_{L/n})$. The main obstacle of going beyond global sections is we cannot find an explicit map between the two vector bundles in question. We will also separate out a simplified argument to obtain higher vanishing for exterior powers of $\bigwedge^d\scq_L$ without an explicit description of the direct image of $\bigwedge^d\scq_L$. 
\end{remark}
A somewhat special and yet illustrative example throughout this section is the case when we take $L$ to be a general hyperplane $H\subseteq \kk^E$. In this case, the short exact sequence $0\to \scs_H \to \sco^E_{X_E} \to \scq_H \to 0$ is the pullback along the blow-up map $\pi_E\colon X_E \to \P^n$ of the Euler sequence \[
    0\to \Omega^1_{\P^n}(1) \to \sco^E_{\P^n} \to \sco(1) \to 0.
\]  The matroid for $H$ is the rank-$n$ uniform matroid $U_{n,E}$ on $E$. Then the assertion of Theorem~\ref{thm:SQvanishing} in this case is that \[
    H^p(\P^n, \Omega^{d-1}_{\P^n}(d) \oplus \Omega_{\P^n}^d(d)) =0,\quad\text{for $p>0$ and $d\ge 0$.}
\]
In this case, the result also follows from either a computation or applying Bott--Steenbrink--Danilov vanishing~\cite[Theorem~9.3.1]{CLS}.

\subsection{Tautological bundles on the projective line}\label{subsec:bestP1}
This subsection is devoted to study certain vector bundles $\P^1$ preluded in Step~1. Let $L$ be a $d$-dimensional linear subspace in $\kk^E$.
Let $\Gm$ act on $\kk^E$ by $\rho(t) = (1,\dots,1,t^{-1})$, which induces an action on $\Gr(d,\kk^E)$.
The $\Gm$-orbit closure $\overline{\Gm \cdot_\rho L}$ is isomorphic to either $\P^1$ or a
point. Either way, there is a $\Gm$-equivariant map
\[
    j_L\colon \P^1 \to \overline{\Gm \cdot_\rho L} \hookrightarrow \Gr(d,\kk^E),
\] mapping $t\in \Gm$ to $[\rho(t) L]$.

\begin{definition}
    Let $\scs'_L$ (resp. $\scq'_L$) be pullbacks \textit{via} the map $j_L$ of the tautological sub- (resp. quotient) bundle on
    $\Gr(d,\kk^E)$.
\end{definition}
Similarly we denote by $\sce'_L$ the direct sum $\scs'_L \oplus \scq'_L.$
Eur proved in \cite[Lemma~3.2, Remark~3.3]{bestCohomology}:
\begin{lemma}\label{eq:eur-lemma}
    Keeping the notation as above, we have two \emph{split} short exact sequences \[
        0\to \scs'_{L/n \oplus 0}  \to \scs_{L}' \to
        \scl_{S} \to 0,\quad\text{and}\quad 0\to \scl_{Q} \to \scq_L' \to  \scq'_{L\setminus n \oplus L \mid n} \to 0,
    \]
    where $\scs'_{L/n}$ is a trivial bundle with fibre $L/n$, $\scq'_{L\setminus n \oplus L|n}$ is a trivial bundle with fibre $\kk^{E \setminus n}/(L\setminus n)$, and finally, \[
        \scl_{S} =
        \begin{cases}
            0, & \text{if $n$ is a loop,}\\
            \sco\otimes \kk^{\{n\}}, & \text{if $n$ is a coloop,}\\
            \sco_{\P U}(-1), & \text{otherwise;}
        \end{cases} \quad \scl_{Q}=
        \begin{cases}
            0, & \text{if $n$ is a loop or coloop,}\\
            \sco_{\P U}(1) & \text{otherwise}.
        \end{cases}
    \]
    Here, in the last cases, the line bundles come from pulling back along an isomorphism $\P^1 \simeq \P U$ with $U = (L\setminus n)/(L/n)
    \oplus \kk$.
\end{lemma}
Recall that for $2$-dimensional vector space $U$, we have the identification $H^0(\P U, \sco_{\P U}(d)) = \sym^d U^\vee$ for
$d \ge 0$ and
$H^1(\P U, \sco_{\P U}(d)) = H^0(\sco_{\P U}(-2-d))^\vee = 0$ for $d \ge -1$.
\begin{lemma}\label{lem:fibreVanishing}
    Keeping the notation as above, we have
    \begin{enumerate}
        \item $H^1(\P^1, \bigwedge^d \sce'_{L}) = 0$ for all $d\ge 0$.
            and \[
                H^0(\bigwedge^d \sce'_{L}) \simeq
                \bigwedge^d [L/n \oplus \kk^{E\setminus n}/(L/n) \oplus \kk].
            \]
        \item $H^1(\P^1, \scl_S \otimes \bigwedge^d \scq_L') = 0$ for all $d\ge 0$.
    \end{enumerate}
\end{lemma}
\begin{proof}
    \underline{Statement~(i).} If $n$ is a loop or a coloop, then one can check that $\sce_L'$ is trivial with fibre identified as $L/n \oplus
    \kk^{E/n}/(L/n) \oplus \kk$ by Lemma~\ref{eq:eur-lemma}, so the result follows.
    Henceforth assume $n$ is neither a loop nor a coloop.
    Upon choosing a splitting of the two short exact sequence in Lemma~\ref{eq:eur-lemma}, we have a noncanonical isomorphism \[
        \sce'_L \simeq   [L/n \oplus \kk^{E \setminus n}/(L\setminus n)]\otimes \sco_{\P^1} \oplus \scl_S \oplus \scl_Q.
    \]
    Splitting $\bigwedge^d \sce_L'$ into line bundles, we obtain
    \begin{align*}
        \bigwedge^d \sce'_L \simeq & \bigwedge^d [L/n \oplus \kk^{E \setminus n}/(L\setminus n)]
        \otimes  \sco_{\P^1} \\ &  \oplus \bigwedge^{d-1} [L/n \oplus \kk^{E \setminus n}/(L\setminus
        n)] \otimes [\scl_Q \oplus \scl_S] \oplus \bigwedge^{d-2} [L/n \oplus \kk^{E \setminus
        n}/(L\setminus n)] \otimes \scl_Q \otimes \scl_S.
    \end{align*}

    In particular, any summand  has degree at least $-1$, which gives the first assertion.
    Next, taking the determinant of the Euler sequence for $\P U$, one obtains an isomorphism \(H^0(\scl_S \otimes \scl_Q) = \det U = (L \setminus n)/(L/n) \). Then, taking
    global section of $\bigwedge^d \sce'_L$ yields
    \begin{align*}
        H^0(\bigwedge^d \sce'_L)
        =  & \left\{\bigwedge^d [L/n \oplus \kk^{E \setminus n}/(L\setminus n)] \oplus
            \bigwedge^{d-1} [L/n \oplus \kk^{E \setminus n}/(L\setminus n)] \otimes (L \setminus
        n)/(L/n)\right\} \\
        & \oplus  \left\{\bigwedge^{d-1} [L/n \oplus \kk^{E \setminus n}/(L\setminus n)] \oplus
            \bigwedge^{d-2} [L/n \oplus \kk^{E \setminus n}/(L\setminus n)] \otimes (L \setminus
        n)/(L/n)\right\} \\
        = &  \bigwedge^d [L/n \oplus \kk^{E\setminus n} / (L /n) \oplus \kk],
    \end{align*}
    Here, the last isomorphism follows from the decomposition \[
        {\kk^{E\setminus n} / (L /n)}\oplus \kk = {\kk^{E\setminus
        n} / (L \setminus n)} \oplus {({L\setminus n})/({L/n})} \oplus \kk,
    \] which comes from taking global section of $\scq_L'$. The second assertion follows.

    \underline{Statement~(ii).} The strategy here is similar. The bundle $\scl_S\otimes \bigwedge^d \scq_L'$ is again trivial when $n$ is a loop or coloop. When $n$ is not a loop or coloop, we choose, by Lemma~\ref{eq:eur-lemma}, a splitting
    \[\scq'_L \simeq \scl_Q \oplus  \kk^{E \setminus n}/(L\setminus n) \otimes \sco_{\P^1}\]
    The line bundle direct summands in $\bigwedge^d \scq_L'$ can have degree $0$ and $1$. Therefore, the line bundles in $\scl_S \otimes \bigwedge^d \scq_L'$ have degree at least $-1$; this gives the second assertion.
\end{proof}

\subsection{Vanishing results for higher direct images}\label{subsec:restriction-to-fibre}
Consider the pullback $f^\ast \scs_{L\setminus n}$; this is the $T$-equivariant vector bundle on $X_E$ whose fibre over $\ol t \in \P T$ is $ t^{-1}(L/n \oplus 0)$, so we have an identification $\scs_{L/n\oplus 0} = f^\ast \scs_{L/ n}$ and $\scq_{L/n\oplus 0} = {f^\ast \scq_{L/n} \oplus  \sco_{X^E}^{\{n\}}}$ by \cite[Proposition~3.6]{best}. There is an injective map \[
\eps_1: f^\ast \scs_{L/ n} \to \scs_L,\] which is the induced from the inclusion $L/n \oplus 0 \hookrightarrow L$. We denote by
\begin{align*}
    \scl_1 = \ker \eps_1 = \parbox{8.5cm}{the $T$-equivariant quotient of $\scs_L$ on $X_E$ \\
    whose fibre at $\ol t\in \P T$ is $ t^{-1}L/(t^{-1}L/n\oplus 0)$.}
\end{align*}
We have the following diagram of short exact sequences of vector bundles on $X_E$,
\begin{equation}\label{eq:diagram-upstairs}
    \begin{tikzcd}
        &&& 0 \\
        & 0 && {\scl_1} \\
        0 & {f^\ast \scs_{L/ n}} & {\sco^{E\setminus n}_{X^E} \oplus \sco_{X^E}^{\{n\}}} & {f^\ast \scq_{L/n} \oplus  \sco_{X^E}^{\{n\}}} & 0 \\
        0 & {\scs_L} & {\sco_{X_E}^E} & {\scq_L} & 0 \\
        & {\scl_1} && 0 \\
        & 0
        \arrow[from=1-4, to=2-4]
        \arrow[from=2-2, to=3-2]
        \arrow[from=2-4, to=3-4]
        \arrow[from=3-1, to=3-2]
        \arrow[from=3-2, to=3-3]
        \arrow["\eps_1", from=3-2, to=4-2]
        \arrow[from=3-3, to=3-4]
        \arrow[equals, from=3-3, to=4-3]
        \arrow[from=3-4, to=3-5]
        \arrow["\eps_2",from=3-4, to=4-4]
        \arrow[from=4-1, to=4-2]
        \arrow[from=4-2, to=4-3]
        \arrow[from=4-2, to=5-2]
        \arrow[from=4-3, to=4-4]
        \arrow[from=4-4, to=4-5]
        \arrow[from=4-4, to=5-4]
        \arrow[from=5-2, to=6-2]
    \end{tikzcd}
\end{equation}
Here, the map $\eps_2$ is the natural $T$-equivariant quotient map, whose fibre at the identity of $\P T$ is the surjection $\kk^E/(L/n\oplus 0) \to \kk^E/L$.
Applying the snake lemma, we can also identify $\scl_1\simeq \ker \eps_2$.

Taking the direct sum of the first and third columns of Diagram~\eqref{eq:diagram-upstairs}, we obtain the four-term exact sequence
\begin{equation}\label{eq:four-term-upstairs}
    0\to \scl_1 \to f^\ast (\sce_{L/n}\oplus \sco_{X_{E\setminus n}}) \xrightarrow{\eps_1 \oplus \eps_2} \sce_{L} \to \scl_1 \to 0.
\end{equation}
Here, we will denote \[
    \scg = \im (\eps_1 \oplus \eps_2) \simeq \scq_L \oplus f^\ast S_{L/n},
\] and the second isomorphism here comes from inspecting Diagram~\eqref{eq:diagram-upstairs}.

Additionally, projecting $L$ to the two direct summands of $\kk^{E\setminus n}\oplus \kk^{\{{n\}}}$ gives rise to an embedding $L \hookrightarrow L \setminus n \oplus L |n$. This in turn induces a surjection $\scq_{L}\to \scq_{L\setminus n \oplus L | n}$. Similarly, denote by \[
    \scn_1 = \ker (\scq_{L}\to \scq_{L\setminus n \oplus L | n}) = \parbox{8.5cm}{the $T$-equivariant subbundle of $\scq_L$ on $X_E$ \\
        whose fibre at $\ol t\in \P T$ is \\
    $(t^{-1}L\setminus n\oplus L|n) / t^{-1}L$.}
\] Then, $\scn_1$ is either is a line bundle or zero. Dualising, we arrive at the short exact sequence
\begin{equation}\label{eq:diagram-for-Qvee}
    0\to \scq^\vee_{L\setminus n \oplus L|n} \to \scq^\vee_L\to  \scn_1^\vee \to 0.
\end{equation}

We wish to compute the (higher) direct images under $f$ of $\scl_1$, $\scn_1^\vee$, and exterior powers of $\sce_L$ and $\scg$.
The key to connecting the computation in \S\ref{subsec:bestP1} to these higher direct images
is the theorem of cohomology and base change. We only need a weaker statement
\begin{proposition}
    [Cohomology and base change]
    Let $g\colon X\to Y$ be a proper morphism of schemes, with $Y$ reduced and locally Noetherian. Assume $\scf$ is a coherent sheaf on $X$ that is $\sco_Y$-flat, and, for some integer $p$, $h^p(X_t,\scf|_{X_t})$ is a locally constant function with respect to $t\in Y$. Then the higher direct image $R^p g_*\scf$ is locally free with fibre   $t\in Y$ isomorphic to \(H^p(X_t, \scf|_{X_t})\).
\end{proposition}
A proof can be found in, for example,
\cite[Th\'eor\`eme~3.2.1]{EGAIII1}. We also need the following;
\begin{lemma}[{\cite[Lemma~2.5, Proposition~2.6]{bestCohomology}}]\label{prop:constFibre}
    Let $\scf$ be an ordered set partition of $E\setminus n$ with $\ell$ parts, and $t\in T'$. Recall from Proposition~\ref{propdet:deletion} that the fibre $C=f^{-1}(t p_\scf)$ has components $C(t,i) = (t,1).\P^1_{\scf(i)}$.
    \begin{enumerate}
        \item The
            restriction of $\scs_L$ to the component $\P^1_{\scf(i)}$ is given by $\scs'_{L_{\scf(i)}}$, where we recall $L_{\scf(i)}$ is defined in Eq.\ref{eq:AKdecomp}.
        \item Moreover, the bundles $\scs_L|_C$ and $\scq_L|_C$ are trivial if $n$ is a loop or coloop and nontrivial at one component otherwise. In the latter case, there exists exactly one $k\in [\ell]$ so that $\scs_{L}|_{C(t,i)}$ is nontrivial.
    \end{enumerate}
\end{lemma}

Combining the previous two results as well as Lemma~\ref{lem:fibreVanishing} gives the following. This argument is already given in \cite{bestCohomology}, but we will present it for completeness.
\begin{lemma}\label{lem:derivedVanshings}
    We have the following vanishing results for direct images:
    \begin{enumerate}
        \item For any $d\ge 0$, $R^{i} f_\ast \bigwedge^d \sce_L  = 0$ for all $i> 0$ and $f_\ast \bigwedge^d \sce_L$ is a vector bundle of rank \[
                \dim \bigwedge^d [L/n \oplus \kk^{E\setminus n}/(L/n) \oplus \kk];
            \]
            in particular, $R^i f_\ast (\bigwedge^a \scs_L \otimes \bigwedge^b \scq_L) =0$ for $i > 0$ and $a,b\ge 0$.
        \item $R^{i} f_\ast \scl_1 = 0$ for all $i > 0$, and $f_\ast \scl_1 = \sco_{X^{E\setminus n}}$ when $n$ is a coloop, and $0$ otherwise.
        \item $R^{i} f_\ast \scn_1^\vee = 0$ for $i\ge 0$.
        \item $R^{i} f_\ast \big(\bigwedge^d \scl_1\otimes \scq_L \big) = 0$ for all $i> 0$ and $d\ge 0$.
    \end{enumerate}
\end{lemma}

\begin{proof}
    \underline{Statement~(i).}
    Taking into account the relative dimension of $f$, it is enough to show $R^1 f_\ast \bigwedge^d \sce_L = 0$.
    Without loss of generality, we can assume the fibre has the form $C=f^{-1}(p_\scf)$ for some ordered
    partition $\scf$. By Statement~(ii) of Lemma~\ref{prop:constFibre}, the restriction $\scs_L|_C$ is nontrivial on
    at most one component. But $\scq_L$ is trivial if and only if $\scs_L$ is, so $\sce_L$
    is also nontrivial on at most one component, say, $\P^1_{\scf(k)}$.
    \cite[Theorem~2.5]{bestCohomology} says $\sce'_{L_{\scf(k)}}$ is the restriction $\sce_L$
    to $\P^1_{\scf(k)}$. As a matter of notation, we set $L' = L_{\scf(k)}$.

    Taking a normalisation of the nodal curve $C$, one obtains a short exact sequence \[
        0\to \sco_C \to \bigoplus_i \nu_*\sco_{\P^1_{\scf(i)}}  \to \bigoplus_i \sco^\nu_{p_i}/\sco_{p_i}\to 0.
    \] Twisting by $\bigwedge^d \sce_L$ and taking cohomology yields the short exact sequence
    \begin{align*}
        0&\to H^0(C, \bigwedge^d \sce_L) \to H^0(\P^1_{\scf(k)}, \bigwedge^d \sce_{L'}) \oplus
        W \to  \bigoplus_{j}
        \bigwedge^d
        \sce_L|_{p_{\scf^j}} \\
        &\to H^1(C, \bigwedge^d \sce_{L}) \to H^1(\P^1_{\scf(k)},\bigwedge^d \sce_{L'}) =0,\quad\text{where
        } W = \bigoplus_{i\ne k} H^0(\P^1_{\scf(i)}, \bigwedge^d \sce_L|_{\P^1_{\scf(i)}}).
    \end{align*} Here the last vanishing uses the Statement~(i) of
    Lemma~\ref{lem:fibreVanishing}.  Since $\sce_L$ is trivial along components $\scf(i)$
    for $i \neq k$, the third term and $W$ have the same dimension, say, $e$.
    Additivity of Euler characteristics yields $\chi(C,\bigwedge^d \sce_L)+e=
    h^0(\P^1_{\scf(k)}, \bigwedge^d \sce_L) + e$. This gives $h^1(C, \bigwedge^d \sce_L) = h^0(C,
    \bigwedge^d \sce_L) - h^0(\P^1_{\scf(k)}, \bigwedge^d \sce_L)\ge 0$. On the other hand, if a section
    $s\in H^0(C,\bigwedge^d\sce_L)$ is such that the restriction $s|_{\P^1_{\scf(k)}}$ vanishes, then by
    evaluating $s$ at points $p_{\scf^k}$ and $p_{\scf^{k+1}}$, we see that $s$ vanishes everywhere on $C$; therefore $h^0(C, \bigwedge^d \sce_L)$ is at most $h^0(\P^1_{\scf(k)}, \bigwedge^d \sce_L)$. Therefore the
    previous inequality is an equality, so that we get $h^1(C, \bigwedge^d \sce_L) = 0$, and \[
        h^0(C, \bigwedge^d \sce_L) = \dim \bigwedge^d [L'/n \oplus \kk^{E\setminus n}/(L'/n) \oplus \kk].
    \]
    Combining the second assertion with the observation $\dim L' = \dim L$, one deduces that the dimension of $H^0(f^{-1}(p),\bigwedge^d \sce_L)$ is independent of the point $p \in X_{E\setminus n}$. By cohomology and base change the two assertions follow.

    \underline{Statement~(ii).}
    We have that $\scl_1$ equals the trivial bundle $\sco_{X_E}$ when $n$ is a coloop and $0$ when $n$ is a loop. Thus the assertion follows readily in these cases. Henceforth assume $n$ is not a loop or coloop.
    As before, consider a point $p_\scf$ for an ordered set partition $\scf$ with $\ell$ parts, and $C = f^{-1}(p_\scf)$ be the corresponding fibre. The restriction $f^\ast \scs_{L/n}$ is visibly trivial along the whole fibre $C$. By Lemma~\ref{prop:constFibre}, the restriction $\scs_L|_C$ is nontrivial on exactly one component of $C$, say, $\P^1_{\scf(k)}$. Restricting to $\P^1_{\scf(k)}$ and applying Lemma~\ref{prop:constFibre} to the first and second terms of the column of Diagram~\eqref{eq:diagram-upstairs}, we get a short exact sequence \[
        0\to \scs'_{L'/n\oplus 0} \to \scs'_{L'} \to \scl_1|_{\P^1_{\scf(k)}} \to 0,\quad\text{where } L' = L_{\scf(k)}
    \]
    Looking at the first exact sequence in Lemma~\ref{eq:eur-lemma}, we see that $ \scl_1|_{\P^1_{\scf(k)}} \simeq \scl_S \simeq \sco_{\P^1_{\scf(k)}}(-1)$, where the second equality is because $n$ is not a loop or coloop. In particular,  the line bundle $\scl_1$ restricted to ${\P^1_{\scf(k)}}$ has no cohomology in any degree. For $j\neq k$, we have $\scl_1|_{\P^1_{\scf(j)}} \simeq \sco_{\P^1_{\scf(j)}}$, thanks again to Lemma~\ref{prop:constFibre}.

    Twisting the normalisation exact sequence of $C$ by $\scl_1$ and taking cohomology, we have an exact sequence
    \[
        0\to H^0(C, \scl_1) \to H^0(\P^1_{\scf(k)}, \sco_{\P^1_{\scf(k)}}(-1)) \oplus W_\scl \to \bigoplus_{1< j < \ell} \kappa(p_{\scf^j}) \to H^1(C,\scl_1) \to 0
    \]
    where we denote $W_\scl = \bigoplus_{i\ne k} H^0(\P^1_{\scf(i)}, \bigwedge^d \sce_L) \simeq \bigoplus_{i\ne k} \kk$. The map $W_\scf \to \bigoplus_{1<j<\ell} \kappa(p_{\scf^j})$ is an isomorphism for the same reason as before. Therefore, $H^0(C,\scl_1) = H^1(C,\scl_1) =0$. By cohomology and base change, we have that $f_\ast \scl_1 = 0$.

    \underline{Statement~(iii).} Turning to the direct image of $\scn_1$, we make the following adjustments to the proof of Statement~(ii). When $n$ is a loop or coloop, we have that $\scn_1 = 0$, since $L \to L\setminus n \oplus L | n$ is an isomorphism. When $n$ is not a loop or coloop, by the same argument as in Statement~(ii), we observe that Sequence~\eqref{eq:diagram-for-Qvee} restricts to the $\P^1$-component $\P^1_{\scf(k)}$ of the fibre $C$ as \[
        0\to \scq_{L'\setminus n\oplus L'|n}'^\vee \to \scq_{L'
        }'^\vee \to  \scn_1^\vee|_{\P^1_{\scf(j)}} \simeq \scl_Q^\vee \to 0,\quad\text{where } L' = L_{\scf(k)}
    \] which is dual to the the second short exact sequence in Lemma~\ref{eq:eur-lemma}. By Lemma~\ref{prop:constFibre}, $\scn_1^\vee$ restricts to the trivial bundle on all components $\P^1_{\scf(j)}$ for $j\neq k$ and $\sco(-1)$ on the component $\P^1_{\scf(k)}$. This puts us in the same situation as the final paragraph of the argument for Statement~(ii).

    \underline{Statement~(iv).} The argument for Statement~(i) can be applied directly to this case. Except that we quote Statement~(ii) of Lemma~\ref{lem:fibreVanishing} instead of Statement~(i).
\end{proof}

\subsection{Running the induction}\label{subsec:final-computation}
We first give deletion-contraction formulae for the direct images of exterior powers of $\scs_L$ and $\scq_L^\vee$. The formula involving $\scs_L$ also appeared as \cite[Theorem~1.5]{bestCohomology}, but with a different proof.

\begin{lemma}\label{lem:pushS}
    Let $L\subseteq \kk^E$ be a linear subspace. Then
    \begin{eqnarray*}
        &(i). &   f_\ast \bigwedge^d \scs_L \simeq
        \begin{cases}
            \bigwedge^d (\scs_{L/n} \oplus \sco_{X_{E\setminus n}}), & \text{if $n$ is a coloop;} \\
            \bigwedge^d \scs_{L/n}, & \text{if $n$ is not a coloop.}
        \end{cases} \\
        & (ii). & f_\ast \bigwedge^d \scq_L^\vee \simeq
        \begin{cases}
            \bigwedge^d (\scq_{L\setminus n}^\vee \oplus \sco_{X_{E\setminus n}}), & \text{if $n$ is a loop;} \\
            \bigwedge^d  \scq_{L\setminus n}^\vee, & \text{if $n$ is not a loop.}
        \end{cases}
    \end{eqnarray*}
\end{lemma}
\begin{proof}
    \underline{Statement~$(i)$.} When $n$ is a loop, we have an isomorphism $\scs_L = \scs_{L/n\oplus 0} \simeq f^\ast \scs_{L/ n}$. Therefore the assertion follows from the projection formula and the fact that $f_\ast \sco_X = \sco_{X\setminus n}$. When $n$ is a coloop, we have that $\scl_1 = \sco_{X_E}$. Moreover, that $n$ is a coloop implies the existence of a vector $v\in L \cap \kk^{\{n\}}$ gives a section $v\colon \sco_{X_E} \to \scs_L$. This gives an identification $\scs_L \simeq f^\ast \scs_{L\setminus n} \oplus \sco_{X_E}^{\{n\}} \simeq f^\ast \scs_{L/ n} \oplus \sco_{X_E}^{\{n\}}$, where we note the isomorphism $L/n\simeq L\setminus n$ when $n$ is a coloop.

    Assume from now on $n$ is not a loop or coloop. Considering the pushforward of the exterior power of the first column of Diagram~\eqref{eq:diagram-upstairs} along $f$, we obtain an exact sequence
    \begin{equation*}
        0\to \bigwedge^d S_{L/n} \xrightarrow{f_\ast \bigwedge^d \epsilon_1} \bigwedge^d f_\ast \scs_L \to f_\ast (\scl_1 \otimes \bigwedge^{d-1} f^\ast \scs_{L/n}) \simeq f_\ast \scl_1 \otimes \bigwedge^{d-1}\scs_L,
    \end{equation*}
    where we have used the projection formula in the first and third term.
    But by Statement~(ii) of Lemma~\ref{lem:derivedVanshings}, the rightmost term in the sequence above vanishes, and the claimed isomorphism is given by $f_\ast \bigwedge^d \epsilon_1$.

    \underline{Statement~$(ii)$.}
    When $n$ is a loop or coloop, the assertion is clear from the identity $\scq_L \simeq \scq_{L\setminus n \oplus L | n}$. The latter equals $f^\ast(\scq_{L\setminus n}\oplus \sco_{X_{E\setminus n}})$ when $n$ is a loop, and $f^\ast \scq_{L\setminus n}\simeq f^\ast \scq_{L/n}$ when $n$ is a coloop. Thus the assertion is clear from the projection formula.

    When $n$ is not a loop or coloop, we consider pushing forward Sequence~\eqref{eq:diagram-for-Qvee} along $f$ and obtain\[
        0\to \bigwedge^d \scq_{L\setminus n}^\vee \to f_\ast \bigwedge^d \scq_L^\vee \to f_\ast \scn_1^\vee \otimes \bigwedge^{d-1}\scq_L^\vee.
    \]
    The last factor vanishes by Statement~(iii) of Lemma~\ref{lem:derivedVanshings}. The claim follows.
\end{proof}
\begin{remark}
    Here, the case for $\scq_L^\vee$ cannot be deduced from applying the standard Cremona pullback with cohomology and base change under consideration. This is because the commuting diagram
    \[
        \begin{tikzcd}
            {X_E} & X_E \\
            {X_{E\setminus n}} & {X_{E\setminus n}}
            \arrow["\crem",from=1-1, to=1-2]
            \arrow["f",from=1-1, to=2-1]
            \arrow["f",from=1-2, to=2-2]
            \arrow["\crem",from=2-1, to=2-2]
    \end{tikzcd}\]
    is \textit{not} Cartesian. This can be observed from the big open tori.
\end{remark}

Chopping the four-term sequence~\eqref{eq:four-term-upstairs}, we arrive at two short exact sequences
\begin{align*}
    0\to \scl_1 \to f^\ast (\sce_{L/n} \oplus \sco_{X^{E\setminus n}}) \xrightarrow{\eps_1\oplus \eps_2} \scg \to 0,\quad \text{and}\\
    0 \to \scg \to
    \sce_L \to  \scl_1 \to 0.
\end{align*}
Taking exterior powers gives two short exact sequences
\begin{align}
    0\to  \bigwedge^d \scg \to    \bigwedge^d \sce_L \xrightarrow{\bigwedge^d(\eps_1\oplus \eps_2)} \scl_1 \otimes \bigwedge^{d-1} \scg \to 0,\;\text{and}\tag{*}\label{eq:*} \\
    0\to \scl_1\otimes \bigwedge^{d-1} \scg \to f^\ast \bigwedge^d (\sce_{L/n} \oplus \sco_{X^{E\setminus n}}) \to \bigwedge^d \scg \to 0.\tag{**}\label{eq:**}
\end{align}

We single out some useful pieces of computation, which will be used more than once in the proof of Theorem~\ref{thm:SQvanishing}.

\begin{lemma}\label{lem:pushforward-L1-wedgeQ}
    In the case when $n$ is not a loop, we have the following isomorphisms for $i\ge 0$, \[
        f_\ast (\scl_1\otimes \bigwedge^i \scq_L) = \det \scq_{L/n} \otimes \bigwedge^{n+1-r-i}\scq_{L\setminus n}^\vee \simeq \det \scq_{L/n} \otimes \det \scq_{L\setminus n}^\vee \otimes \bigwedge^{i-1}\scq_{L\setminus n}.
    \]
\end{lemma}
\begin{proof}
    When $n$ is not a loop, $\scl_1$ is a line bundle and we have the identity $\scl_1 \simeq \det \scs_L \otimes f^\ast \det \scs_{L/n}^\vee \simeq \det \scq_L^\vee \otimes f^\ast \det \scq_{L/ n}$ by taking determinants of the first and third columns of Diagram~\eqref{eq:diagram-upstairs}.
    For $i\ge 0$, we can write
    \begin{align*}
        f_\ast (\scl_1 \otimes \bigwedge^i \scq_L) & \simeq f_\ast ( \bigwedge^{n+1-r-i}\scq_L^\vee \otimes f^\ast \det \scq_{L/n}) \\
        & \simeq \det \scq_{L/n} \otimes f_\ast \bigwedge^{n+1-r-i}\scq_L^\vee  \simeq \det \scq_{L/n} \otimes \bigwedge^{n+1-r-i}\scq_{L\setminus n}^\vee,
    \end{align*} where the second isomorphism comes from the fact that for a rank-$e$ vector bundle $\sce$, one has the identity $\bigwedge^d \sce \otimes \det \sce^\vee \simeq \bigwedge^{e-d} \sce^\vee$; moreover, the last isomorphism comes from Statement~(ii) of Lemma~\ref{lem:pushS}.
\end{proof}
This Lemma is already enough for one to prove by induction the higher cohomology vanishing $\bigwedge^d \scq_L,\,d\ge 0$; see Appendix~\ref{appendix:exterior-quotient}. For Theorem~\ref{thm:SQvanishing}, we additionally need the following push-pull computation.
\begin{corollary}\label{cor:pushforward-L1-wedgeG}
    When $n$ is not a loop, we have \[
        f_\ast (\scl_1 \otimes \bigwedge^d \scg) = \bigoplus_{a+b=d}  \bigwedge^{n+1-r-a}\scq_{L\setminus n}^\vee \otimes  \bigwedge^{r-1-b}\scs_{L/n}^\vee
    \]
\end{corollary}
\begin{proof}
    By Lemma~\ref{lem:pushforward-L1-wedgeQ}, we have
    \begin{align*}
        f_\ast (\scl_1\otimes \bigwedge^d \scg) & \simeq \bigoplus_{a+b=d} f_\ast (\scl_1 \otimes \bigwedge^a \scq_L \otimes f^\ast \bigwedge^b \scs_{L/ n}) \\
        & \simeq \bigoplus_{a+b=d} f_\ast (\scl_1 \otimes \bigwedge^a \scq_L) \otimes \bigwedge^b \scs_{L/ n}\simeq \bigoplus_{a+b=d} \bigwedge^{n+1-r-a} \scq_{L\setminus n}^\vee \otimes \bigwedge^{r-1-b} \scs_{L/n}^\vee,
    \end{align*}
    which is the claimed identity.
\end{proof}

\begin{proof}[Proof of Theorem~\ref{thm:SQvanishing}]
    We proceed by induction on the cardinality of the ground set $E$. The base case is trivial, so we assume $n\ge 1$.
    We run the following Leray spectral sequence \[
        E_2^{p,q} = H^p(X_{E\setminus n},R^q f_* \bigwedge^d \sce_{L}) \Rightarrow
        H^{p+q}(X_E, \bigwedge^d \sce_{L}).
    \]
    By Statement~(i) of Lemma~\ref{lem:derivedVanshings}, we know that $\bigwedge^d \sce_L$ has vanishing higher direct images. This gives an isomorphism
    $E^{p,0}_2 \simeq H^{p}(X_E, \bigwedge^d \sce_{L})$. Therefore, it suffices to show $f_\ast \bigwedge^d \sce_L$ has vanishing higher cohomology.

    \underline{Case $n$ is a loop.} Here, we have $\scl_1 = 0$. Therefore, the map $\eps_1 \oplus \eps_2\colon f^\ast (\sce_{L/n}\oplus \sco_{X_{E\setminus n}})\to \sce_L$ is an isomorphism. By the projection formula, we have \[
        E_2^{p,0}  \simeq H^p(X_{E\setminus n},\bigwedge^d (\sce_{L/n}\oplus \sco_{X_{E\setminus n}})) \simeq  H^p(X_{E\setminus n},\bigwedge^d \sce_{L/n}) \oplus H^p(X_{E\setminus n},\bigwedge^{d-1} \sce_{L/n}) =0
    \] where the last equality is given by the inductive hypothesis. Therefore the vanishing statements follow.

    Henceforth, we will deal with the cases when $n$ is not a loop. Therefore, $\scl_1$ will be a line bundle.
    Pushing forward Sequences \eqref{eq:*} and \eqref{eq:**} along $f$ we obtain
    \begin{align}
        0\to  f_\ast \bigwedge^d \scg \to f_\ast \bigwedge^d \sce_L \to f_\ast (\scl_1\otimes \bigwedge^{d-1} \scg) \to 0,\;\text{and}\tag{$\dagger$}\label{eq:dag} \\
        0\to f_\ast (\scl_1\otimes \bigwedge^{d-1} \scg) \to \bigwedge^d (\sce_{L/n} \oplus \sco_{X^{E\setminus n}}) \to f_\ast \bigwedge^d \scg \to 0.\tag{$\dagger\dagger$}\label{eq:dagdag}
    \end{align}
    Here, we used the fact that $\scl_1\otimes \bigwedge^{d-1} \scg$ and $\bigwedge^d \scg$ have vanishing higher direct images, for projection formula and Statements~(i) and ~(iv) of Lemma~\ref{lem:derivedVanshings} gives for $i>0$,
    \begin{align*}
        R^i f_\ast \bigwedge^{d}\scg & \simeq \bigoplus_{a+b=d} \big(R^i f_\ast  \bigwedge^a \scq_L \big) \otimes \bigwedge^b \scs_{L/ n} = 0,\, \text{ and}\\
        R^i f_\ast (\scl_1 \otimes \bigwedge^{d-1}\scg) & \simeq \bigoplus_{a+b=d-1} R^i f_\ast (\scl_1 \otimes \bigwedge^a \scq_L) \otimes \bigwedge^b \scs_{L/ n} = 0.
    \end{align*}
    The strategy here is to control the cohomology of $f_* (\scl_1\otimes \bigwedge^{d-1} \scg)$. Once this is done, we will inspect the long exact sequence in cohomology for Sequences~\eqref{eq:dag} and \eqref{eq:dagdag} to obtain the desired vanishing statements.

    \underline{Case $n$ is a coloop.} In this situation, deletion and contraction by $n$ is the same. More precisely, the inclusion $L/n \to L\setminus n$ is an isomorphism. Therefore, applying Corollary~\ref{cor:pushforward-L1-wedgeG}, we have that for $i>0$,
    \begin{align*}
        H^i(X_{E\setminus n}, f_* (\scl_1\otimes \bigwedge^{d-1} \scg)) & \simeq \bigoplus_{a+b=d-1} H^i(X_{E\setminus n}, \bigwedge^{n+1-r-a}\scq_{L\setminus n}^\vee \otimes  \bigwedge^{r-1-b}\scs_{L/n}^\vee) \\
        & \simeq \bigoplus_{a+b=d-1} H^i(X_{E\setminus n}, \bigwedge^{n+1-r-a}\scq_{L\setminus n}^\vee \otimes  \bigwedge^{r-1-b}\scs_{L\setminus n}^\vee) \\
        &  \simeq \bigoplus_{a+b=d-1} H^i(X_{E\setminus n}, \bigwedge^{n+1-r-a}\scs_{L^\perp/n} \otimes \bigwedge^{r-1-b}\scq_{L^\perp/ n}) = 0,
    \end{align*}
    where the third isomorphism comes from pulling back by the standard Cremona transform, and the last vanishing comes from the inductive hypothesis.
    Now, taking cohomology of Sequence~\eqref{eq:dagdag} and applying the inductive hypothesis again, we see that $f_\ast \bigwedge^d \scg$ also has zero higher cohomology\[
        H^i(X_{E\setminus n},f_\ast \bigwedge^d \scg)\simeq  H^{i+1}(X_{E\setminus n}, f_\ast (\scl_1\otimes \bigwedge^{d-1}\scg)) = 0,\quad \text{for $i > 0$.}
    \]
    Turning to Sequence~\eqref{eq:dag} and taking cohomology, we get, for $i>0$, an exact sequence
    \[
        H^i(f_\ast \bigwedge^d \scg) \to H^i(f_\ast \bigwedge^d \sce_L ) \to H^i(f_\ast (\scl_1\otimes \bigwedge^{d-1}\scg )).
    \]
    But we have shown that the first and third term vanishes; therefore, the middle term must vanish, giving the assertion.

    \underline{Case $n$ is not a loop or coloop.} Consider the following diagram of short exact sequences of vector bundles in $X_{E\setminus n}$,
    \begin{equation}\label{eq:diagram-downstairs}
        \begin{tikzcd}
            &&& 0 \\
            & 0 && {\scl_2} \\
            0 & {\scs_{L/n}} & {\sco^{E\setminus n}_{X_{E\setminus n}}} & {\scq_{L/n}} & 0 \\
            0 & {\scs_{L\setminus n}} & {\sco^{E\setminus n}_{X_{E\setminus n}}} & {\scq_{L\setminus n}} & 0 \\
            & {\scl_2} && 0 \\
            & 0
            \arrow[from=1-4, to=2-4]
            \arrow[from=2-2, to=3-2]
            \arrow[from=2-4, to=3-4]
            \arrow[from=3-1, to=3-2]
            \arrow[from=3-2, to=3-3]
            \arrow["\delta_1",from=3-2, to=4-2]
            \arrow[from=3-3, to=3-4]
            \arrow[equals, from=3-3, to=4-3]
            \arrow[from=3-4, to=3-5]
            \arrow["\delta_2",from=3-4, to=4-4]
            \arrow[from=4-1, to=4-2]
            \arrow[from=4-2, to=4-3]
            \arrow[from=4-2, to=5-2]
            \arrow[from=4-3, to=4-4]
            \arrow[from=4-4, to=4-5]
            \arrow[from=4-4, to=5-4]
            \arrow[from=5-2, to=6-2]
        \end{tikzcd}
    \end{equation}
    Here, the map $\delta_1$ is induced from the inclusion $L/n\hookrightarrow L\setminus n$. Here, we take \[
        \scl_2  = \coker (\delta_1) = \parbox{9cm}{the $T'$-equivariant quotient of $\scs_{L\setminus n}$ over $X_{E\setminus n}$ \\
        whose fibre at $\ol t\in \P T'$ is $(t^{-1}{L\setminus n})/(t^{-1}L/n)$.}
    \] This is also the kernel of $\delta_2$, thanks again to the snake lemma. Moreover, since $n$ is not a loop or coloop, the inclusion $L/n\to L\setminus n$ is strict of relative dimension $1$, $\scl_2$ is a line bundle.

    Taking the direct sum of first and third row of Diagram~\eqref{eq:diagram-downstairs}, we obtain the $4$-term exact sequence
    \begin{multline*}
        0\to \scl_2 \to \sce_{L/n} \xrightarrow{\delta_1\oplus \delta_2} \sce_{L\setminus n} \to \scl_2 \to 0,\\
        \text{where we set $\sch$ to be $\im (\delta_1\oplus \delta_2) \simeq \scs_{L/ n} \oplus \scq_{L\setminus n}$.}
    \end{multline*}

    Taking exterior powers, we obtain, for any $d\ge 0$, the short exact sequences
    \begin{align}
        0 \to  \bigwedge^d \sch \to \bigwedge^d \sce_{L\setminus n} \to \scl_2\otimes \bigwedge^{d-1} \sch \to 0,\;\text{and}\tag{$\ddagger$}\label{eq:ddag} \\
        0 \to  \scl_2\otimes \bigwedge^{d-1} \sch \to \bigwedge^d \sce_{L/ n} \to  \bigwedge^d \sch \to 0.\tag{$\ddagger\ddagger$}\label{eq:ddagddag}
    \end{align}
    By the inductive hypothesis, the middle terms in Sequences~\eqref{eq:ddag} and \eqref{eq:ddagddag} have no higher cohomology. Taking cohomology of Sequence~\eqref{eq:ddag} and \eqref{eq:ddagddag} in an alternating way, we arrive at two chains of isomorphisms for any $d\ge 0$,
    \begin{align*}
        H^1(X_{E\setminus n},\scl_2\otimes \bigwedge^{d-1} \sch) \simeq H^2(X_{E\setminus n}, \bigwedge^d \sch) \simeq H^3(X_{E\setminus n},\scl_2 \otimes \bigwedge^{d-1} \sch ) \simeq \cdots,\quad\text{and}\\
        H^1(X_{E\setminus n},\bigwedge^d \sch) \simeq H^2(X_{E\setminus n}, \scl_2 \otimes \bigwedge^{d-1} \sch ) \simeq H^3(X_{E\setminus n}, \scl_2\otimes \bigwedge^{d-1} \sch) \simeq \cdots,
    \end{align*}
    both of which must end with $0$ by Grothendieck vanishing. In particular, any exterior power of $\sch$ has no higher cohomology; therefore, we have
    \begin{equation}\label{eq:deletion-contraction-vanishing}
        H^i(X_{E\setminus n}, \bigwedge^a \scs_{L/n} \otimes \bigwedge^b \scq_{L\setminus n}) = 0\quad\text{for any $i> 0$ and $a,b\ge 0$.}
    \end{equation}
    Applying Corollary~\ref{cor:pushforward-L1-wedgeG} and pulling back by the standard Cremona transform, we have that for $i>0$,
    \begin{align*}
        H^i(X_{E\setminus n}, f_* (\scl_1\otimes \bigwedge^{d-1} \scg)) & \simeq \bigoplus_{a+b=d-1} H^i(X_{E\setminus n}, \bigwedge^{n+1-r-a}\scq_{L\setminus n}^\vee \otimes  \bigwedge^{r-1-b}\scs_{L/n}^\vee) \\
        &  \simeq \bigoplus_{a+b=d-1} H^i(X_{E\setminus n}, \bigwedge^{n+1-r-a}\scs_{L^\perp/n} \otimes \bigwedge^{r-1-b}\scq_{L^\perp\setminus n}) = 0,
    \end{align*}
    thanks to Eq.\eqref{eq:deletion-contraction-vanishing}.

    Finally, we can return to Sequences~\eqref{eq:dag} and \eqref{eq:dagdag}. By the inductive hypothesis, we have that the middle terms of Eq.\eqref{eq:dagdag} have no higher cohomology, for all $d\ge 0$. Taking cohomology of Sequence~\eqref{eq:dagdag}, we obtain \[
        H^i(X_{E\setminus n}, f_\ast \bigwedge^d \scg) \simeq H^{i+1}(X_{E\setminus n}, f_\ast (\scl_1\otimes \bigwedge^{d-1} \scg)) = 0\quad\text{for all $i > 0$ and $d\ge 0$.}
    \]
    Now, turning to Sequence~\eqref{eq:dag}. Taking cohomology and noting that the first and third term in the sequence have no higher cohomology, we get that \[
        H^p(X_E,\bigwedge^d \sce_L) = H^p(X_{E\setminus n}, f_\ast \bigwedge^d \sce_L) = 0,
    \] for all $p> 0$ and $d\ge 0$.
\end{proof}

\subsection{Application to Tutte polynomials}
To a rank-$r$ matroid $\rmm$ on the finite set $E$, one can associate its \textit{Tutte polynomial} $T_{\rmm}$~\cite{CrapoTuttePolynomial}. It is defined as \[
    T_{\rmm}(x,y) = \sum_{A \subseteq E} (x-1)^{r-\rk_{\rmm}(A)}(y-1)^{\# A - \rk_{\rmm}(A)}.
\]
When $E$ has cardinality greater than $1$, the Tutte polynomial $T_\rmm$ satisfies the following deletion-contraction property, given $n\in E$,
\begin{equation}\label{eq:tutte-deletion-contraction}
    T_\rmm (x,y) =
    \begin{cases}
        x \,T_{\rmm / n} (x,y) & \text{if $n$ is a coloop in $\rmm$,}\\
        y \,T_{\rmm \setminus n}(x,y) & \text{if $n$ is a loop in $\rmm$,}\\
        T_{\rmm/ n}(x,y) + T_{\rmm \setminus n}(x,y) & \text{if $n$ is neither a loop nor a coloop in $\rmm$.}
    \end{cases}
\end{equation}

Let $\rmm(L)$ be the matroid associated to the linear subspace $L\subseteq \kk^E$ of dimension $r$. As an application of Theorem~\ref{thm:SQvanishing}, we give a formula for the Tutte polynomial for the matroid $\rmm(L)$.
\begin{proposition}\label{prop:tutte-vs-h0}
    Let $L\subseteq \kk^E$ be a linear subspace of dimension $r$ that satisfies Condition~\eqref{eq:ll}. Consider the polynomial \[
        h_L(u,v) = \sum_{p,q} h^0(\bigwedge^p \scs_L \otimes \bigwedge^q \scq_L)\,u^p v^q.
    \] We have \[
        h_L(u,v) = v^{\# E -r} T_{\rmm(L)}(u+1,v^{-1}+1) = \sum_{A \subseteq E} u^{r-\rk_{\rmm}(A)} v^{\# E \setminus A - (r-\rk_{\rmm}(A))}.
    \]
\end{proposition}
What is essentially the same formula, but with equivariant Euler characteristic in place of $h^0$, is given in \cite[Theorem~5.2]{FinkSpeyerKClasses}. The formula above also generalises \cite[Theorem~1.2]{bestCohomology}. More precisely, granting the Proposition, one can obtain the following formulae when one of the variables is set to $0$:
\begin{align*}
    h_L(u,0) = (u+1)^{\#\mathrm{coloops}(\rmm)}, \quad\text{and}\quad
    h_L(0,v) = \sum_{\substack{A \subseteq E \\
    \text{$A$ contains a basis}}} v^{\# E - \# A}
\end{align*}
Here, the first identity comes from considering the deletion-contraction formula~\eqref{eq:tutte-deletion-contraction} for $T_{\rmm (L)}$.

\begin{proof}[Proof of Proposition~\ref{prop:tutte-vs-h0}]
    By Theorem~\ref{thm:SQvanishing}, we can write \[
        h_L (u,v) = \sum_{p,q} \chi(\bigwedge^p \scs_L \otimes \bigwedge^q \scq_L)u^p v^p.
    \] Since $\scs^\vee_L$ and $\scq^\vee_L$ have simple Chern roots in the sense of~\cite[Section~10]{best}, the right-hand term is also
    equal to
    \begin{align*}
        & (u+1)^r (v+1)^{\# E -r} \int_{X_E}c(\scs^\vee_L, \tfrac{1}{u+1}) c(\scq_L,
        -\tfrac{1}{v+1})(1+\alpha+\alpha^2+\dots)\\
        & = v^{\# E -r} T_M(u+1,v^{-1}+1),
    \end{align*} thanks to \cite[Theorem~A and Theorem~10.5]{best}, giving the first equality. The second equality follows from the definition of  $T_{\rmm (L)}$ above.
\end{proof}
\begin{remark}
    In~\cite[Theorem~3.1]{bauer2025equivarianttuttepolynomial}, the authors considered the polynomial $F_\rmm(x,y,z,w)$, which is defined by taking the reciprocal polynomial in variables $x,y$ of the multidegree \[
        \int_{X_E}^T c^T(\scs_\rmm^\vee, z)c^T(\scq_\rmm, w) \frac{1}{1- x\alpha} \frac{1}{1-y\beta} \in \ZZ[t_e : e\in E][x,y,z,w],
    \]
    where we have taken into account the identification $H_T^*(\on{pt}) \simeq \sym N_E^\vee \simeq \ZZ[t_e : e\in E]$.
    Then, as polynomials over $\ZZ[t_e: e\in E][u,v]$, one has
    \begin{align*}
        & \sum_{p,q} \on{ch}|_{ 1+t_0,\dots,1+t_n}H^0(\bigwedge^p \scs_L \otimes \bigwedge^q \scq_L)\,u^p v^p = \\
        &  \text{The coefficient at $y$-degree $0$ of}\quad y^{\# E-1}F_{\rmm}\left(1,\frac{1}{y},\frac{1}{u+1},\frac{-1}{v+1}\right),
    \end{align*} where $\on{ch}|_{u_0,\dots,u_n}(V)$ is the character of the $T$-representation of $V$ evaluated at $u_0,\dots,u_n$.
\end{remark}

The \textit{characteristic polynomial} of a rank-$r$ matroid $\rmm$ is \[
    \chi_{\rmm} (u) = (-1)^r T_{\rmm} (-u+1,0).
\]
Specialising the identity in Proposition~\ref{prop:tutte-vs-h0} to $v=-1$, negating the variable $u$, and applying the Koszul resolution in Corollary~\ref{prop:Koszul-W}, we obtain:
\begin{corollary}\label{cor:char-poly}
    Keeping the notation and assumptions , we have \[
        \chi_{\rmm(L)} (u) = \sum_{p=0}^r (-u)^{r-p} h^0(W_L, \bigwedge^p \scs_L^\vee),
    \] where $\chi_{\rmm(L)}$ is the characteristic polynomial of the matroid $\rmm(L)$.
\end{corollary}

\subsection{Vanishing theorems for hook-shaped Schur functors}\label{subsec:hook-vanishing}
This subsection is logically independent from the rest of the paper. Here, Schur powers of vector bundles will be considered. To fix convention, we refer to Chapter~2 of \cite{Weyman03} for a construction.

One application of Theorem~\ref{thm:SQvanishing} is removing the characteristic-$0$ assumption in the special case of Theorem~\ref{thm:bwb-vanishing} where one take $W_L = X_E$ and $k=1$ and $\lambda_1$ is a hook. It will be a consequence of the following generalisation of a trick due to~\cite{Broer,BrionTohoku}.
\begin{lemma}\label{lem:broer-trick}
    Let $X$ be a finite-type $\kk$-scheme, and
    \begin{equation}\label{eq:ses-template}
        0\to\sce \to V\otimes \sco_X \to \scf \to 0
    \end{equation} a short exact sequence of vector bundles on $X$, where $V$ is a
    finite-dimensional $\kk$-vector space. Given an integer $t\ge 0$ and a flat coherent sheaf $\scg$ on $X$, the following are equivalent:
    \begin{enumerate}
        \item $H^p(X,\scg\otimes \sym^q \scf ) = 0$ for $p>t$ and $q\ge 0$.
        \item $H^p(X,\scg\otimes \bigwedge^q \sce) = 0$ for $p-q> t$.
    \end{enumerate}
\end{lemma}

\begin{proof}
    Let $e = \rk \sce$ and $f = \rk \scf$. By \cite[Proposition~5.1.1]{Weyman03}, we have a Koszul resolution of  $\sco_{\tot \scf^\vee}$ as an $\sco_{V\times X}$-module, \[
        0\to p_2^\ast \det \sce \to p_2^\ast \bigwedge^{e-1}\sce \to \cdots \to p_2^\ast \sce \to \sco_{X\times V} \to \sco_{\tot \scf^\vee} \to 0,
    \] where the map $p_2$ is the projection $X\times V \to X$.
    Pushing forward to $X$ and applying the projection formula with the identity $p_\ast \sco_{X\times V} \simeq \sym V$ under
    consideration, we get an exact sequence \[
        \sck\colon   0 \to \det \sce \otimes \sym V \to \bigwedge^{e-1} \sce \otimes \sym V \to \cdots \to \sce \otimes
        \sym V \to \sym V \to \sym \scf \to 0.
    \]  The natural
    $\Gm$-actions on $\sce$ and $\scf$ makes this a graded resolution; let \[
        \sck(q,\scg) \colon 0 \to \scg \otimes \bigwedge^{q} \sce \to \scg \otimes \bigwedge^{q-1} \sce \otimes V \to
        \cdots \to \scg \otimes \sym^q V \to \scg \otimes \sym^q \scf \to 0
    \] be the $q$th graded piece twisted by $\scg$.

    Assuming (i), we prove (ii) by induction on $q$.
    When $q = 0$, the conclusion is clear, for $H^p(X,\scg) = 0$ for $p>t$ by assumption. Now assume
    $q> 0 $, the
    inductive hypothesis implies that \[
        H^{p-q+j+1}(X,\scg \otimes \bigwedge^{j} \sce \otimes \sym^{q-j} V) = 0,\quad \text{for
    $0\le j< q$ and $p>q+t$.}\] As a result, chopping $\sck(q,\scg)$ into short exact sequence and taking
    cohomology,  we get that for $p > q+t$, the map \[
        H^{p-q}(X,\scg \otimes \sym^q \scf) \to H^p(X,\scg \otimes \bigwedge^q \sce)
    \] is surjective. But the assumption is that the left-hand term vanishes, so the desired
    vanishing follows. That (ii) implies (i) is a direct consequence of Fact~\ref{fact:lazarsfeld}.
\end{proof}
We recall the following; see \textit{e.g.}~\cite[Exercise~2.2]{Weyman03}:
\begin{fact}\label{fact:weyman-hook-exercise}
    Let $M$ be a finite free module over a $\kk$-algebra $A$. Fix $q>0$. The Koszul complex \[
        0 \to  \bigwedge^{q} M \to  \bigwedge^{q-1} M \otimes M \to
        \cdots \to M \otimes \sym^{q-1} M \to  \sym^q M \to 0
    \] is split exact, with the cycles in $\bigwedge^{q-i} M \otimes \sym^i M$ given by $\schur^{q-i+1,1^{i-1}}M$.
\end{fact}
We now give a case where the conclusion of Theorem~\ref{thm:bwb-vanishing} holds in positive
characteristic.
\begin{proposition}\label{prop:bwb-hook-charp}
    Keeping the assumptions and notation as in Theorem~\ref{thm:SQvanishing}, let $\lambda$ be a hook. We have \(H^i(X_E, \schur^{\lambda} \scq_L) = 0\) for $i>0$.
\end{proposition}
Taking $\lambda = (1^d)$, we have $\schur^\lambda \scq_L =\sym^d \scq_L$, so this Proposition gives higher vanishing for symmetric powers of $\scq_L$ and recovers \cite[Theorem~1.3]{bestCohomology} as promised in Remark~\ref{rm:eur-gap}.
\begin{proof}
    Write $\lambda = (q-i+1,1^{i-1})$ for some $q\ge i>0$. Applying Theorem~\ref{thm:SQvanishing} and Lemma~\ref{lem:broer-trick} with $\scg = \bigwedge^k \scq_L$, we have that $H^{p}(X_E,\bigwedge^{q-i} \scq_L \otimes \sym^i \scq_L)=0$ for $p >0$. The complex in Fact~\ref{fact:weyman-hook-exercise} globalises readily to the case where $M$ is replaced by the vector bundle $\scq_L$. Chopping the complex into short exact sequences and taking cohomology yield the desired vanishing.
\end{proof}
\begin{remark}[The Carter--Lusztig resolution]
    Let $\kk$ be a field of characteristic $p>0$ and $\sce$ a vector bundle on a finite-type $\kk$-scheme $X$ with absolute Frobenius $F_{\on{abs}}\colon X\to X$. Carter and Lusztig~\cite[235]{CarterLusztig}\footnote{Their construction of Weyl module $\ol{V}_\lambda$ is in our language $\schur^{\lambda'}\sce^\vee$ with $\lambda'$ the conjugate partition. Our presentation here follows~\cite{Arapura}, which agrees with Weyman's convention.} constructed a resolution of the Frobenius pullback $F_{\on{abs}}^\ast \sce$ by Schur powers associated with hooks: \[
        0 \to F_{\on{abs}}^\ast\sce \to \schur^{(p-1,1)}\sce \to \schur^{(p-2,1^2)} \sce \to \cdots \to \schur^{\lambda}\sce \to 0,
    \] where $\lambda = (p-N,1^N)$ with $N = \min\{p-1,\rk \sce-1\}$. Taking a linear subspace $L \subseteq \kk^E$ and applying the Carter--Lusztig resolution to $\sce = \scq_L$ and $X = X_E$, this resolution becomes acyclic. It follows that \[
        H^i(X_E,F_{\on{abs}}^\ast\scq_L) = 0,\quad\text{for $i>\min\{p-2,n-r-1\}$.}
    \]
\end{remark}

\section{Logarithmic geometry and Orlik--Solomon algebras}\label{sec:orlik-solomon}

The section is devoted to studying the Orlik--Solomon algebra via the logarithmic geometry of the pair $(W_L,D_L)$ for a linear subspace $L\subseteq \kk^E$ over an arbitrary field $\kk$.

\subsection{Degeneration of the Hodge--de Rham sequence}
Before saying more about the Orlik--Solomon algebra, we first give a quick proof of Statement~(i) of Theorem~\ref{thm:log-vanishing}, which we recall below:
\begin{theorem}\label{thm:log-vanishing-restate}
    Let $\kk$ be a field and $L \subseteq \kk^E$ be a linear subspace satisfying Condition~\eqref{eq:ll}. Then at the $E_1$ page of the Hodge--de Rham spectral sequence \eqref{eq:HodgeDeligne} we have \[
    H^q(W_L,\Omega^p_{W_L}(\log D_L)) = 0,\quad\text{for all $q> 0$.}\]
\end{theorem}
\begin{proof}
    By flat base change~\cite[\href{https://stacks.math.columbia.edu/tag/02KH}{Lemma 02KH}]{stacks-project}, we can take $\kk$ to be algebraically closed, so that the assumptions of Proposition~\ref{prop:degeneracy-locus} and Corollary~\ref{prop:Koszul-W} are satisfied. 
    Tensoring the Koszul resolution in Corollary~\ref{prop:Koszul-W} by $\bigwedge^q \scs_L^\vee$, we obtain an exact sequence
    \[
        0\to \bigwedge^q \scs_L^\vee \otimes \det \scq^\vee_L \to \cdots \to \bigwedge^q \scs_L^\vee \otimes
        \scq^\vee_L \to \bigwedge^q \scs_L^\vee\to \bigwedge^q \scs_L^\vee|_{W_L} \to 0.
    \]
    Combining Fact~\ref{fact:lazarsfeld} with the Cremona transform applied to the statement of Theorem~\ref{thm:SQvanishing}, we obtain the vanishing  \[
        H^{i}(W_L, \bigwedge^q \scs_L^\vee) = 0,\quad\text{for all $i>0$}.
    \]
    Let $p>1$ be given. Taking cohomology of exterior powers of Sequence~\eqref{eq:Euler-seq} from Statement~(ii) of Proposition~\ref{prop:degeneracy-locus}, and applying the vanishing of higher
    cohomology of $\bigwedge^q \scs_L^\vee|_{W_L}$, we obtain a chain of isomorphisms \[
        H^p(\Omega^{q}_{W_L}(\log D_L)) \xrightarrow{\simeq} H^{p+1}(\Omega^{q+1}_{W_L}(\log D_L))\xrightarrow{\simeq}
        H^{p+2}(\Omega^{q+2}_{W_L}(\log D_L)) \xrightarrow{\simeq} \cdots,
    \] but this chain has to end with $0$, thanks to Grothendieck vanishing.
\end{proof}

\begin{remark}\label{rm:DeligneIllusie}
    When $\kk$ has positive characteristic, and one replaces $(W_L, D_L)$ with any log smooth pair $(X,D)$, the $E_1$-degeneration of the spectral Sequence
    \[
        E_1^{p,q} = H^q(X, \Omega^p_X(\log D)) \implies \HH^{p+q}(X,\Omega^\bullet_X(\log D))
    \]
    is not always guaranteed. One sufficient condition, due to Deligne and Illusie~\cite{DeligneIllusie}, is that $\kk$ is a perfect field with $\on{char} \kk \ge \dim X$, and the pair $(X,D)$ admits a lift to the second Witt vectors $W_2(\kk)$. However, the assumption does not always hold even for pairs $(W_L, D_L)$ of our interest. For example, it breaks when we take $L$ to be a representation of $\PG(n,2)$ with $n = \dim W_{L} > 2$.  One might want to interpret Theorem~\ref{thm:log-vanishing} as failure of the Hodge--de Rham sequence~\eqref{eq:HodgeDeligne} to be an obstruction to lifting arrangements to characteristic $0$.
\end{remark}
\begin{remark}
    When $\kk = \CC$, Deligne's argument for the isomorphism $\HH^k( \Omega^\bullet_{W_L}(\log D_L))\to H^k_{\sing}(\P L^\circ)_\CC$ involves showing the map of complexes $\Omega^\bullet_{W_L}(\log D_L) \to j_* \Omega^\bullet_{\P L^\circ}$ is a quasi-isomorphism; see \cite[\S3.1.7]{deligneHodgeII}. One might ask in Theorem~\ref{thm:log-vanishing} if the hypercohomology $\HH^i(\Omega^\bullet (\log D_L))$ can be replaced by the algebraic de Rham cohomology $H_{\dR}^i(\P L^\circ) \coloneqq\HH^i(W_L,j_* \Omega^\bullet_{\P L^\circ})$ of the arrangement complement. In positive characteristic, the answer is an emphatic \emph{no}: with the Cartier operator, it is not too hard to see $H^\bullet_\dR ({\P L^\circ})$ is mostly infinite-dimensional as a $\kk$-vector space. For instance, if one takes $\kk = \bar{\FF}_p$ and $L = \kk^2$, then $\P L^\circ \simeq \spec \kk[x,x^{-1}]$ is the complement of the projectivised rank-$2$ Boolean arrangement, and $H_{\dR}^1(\P L^\circ)$ is a free $\kk[x^p,x^{-p}]$-module generated by $t^{p-1}dt$.
\end{remark}

\subsection{Comparison theorems}
We recall some notation from the introduction. Let $H_i$ be the $i$\/th coordinate hyperplane of $\P^n$ restricted to $\P L$, cut out by a linear form $x_i \in L^\vee$. 

We denote by $\be_i$ the basis vectors of $\ZZ^{E}$, and $\be_S = \be_{s_1} \wedge \dots \wedge \be_{s_k} \in \bigwedge^k \ZZ^E$, for an ordered subset $S = \{s_1< \dots  < s_k\}\subseteq E$.  Given a linear space $L\subseteq \kk^E$ without loops, one can associate to it the Orlik--Solomon algebra \[
    \OS^\bullet(L) \coloneqq \frac{\bigwedge^\bullet \ZZ^{E}}{\langle\partial \be_S\colon \text{$S\subseteq E$ is dependent}\rangle},
\] where $\partial$ denotes the usual Koszul differential on $\bigwedge^\bullet \ZZ^{E}$, defined by \[
\partial \be_S = \sum_{i=0}^{k-1} (-1)^{i} \be_{S \setminus s_{k-i}},\quad\text{for an ordered subset $S = \{s_1<\dots<s_k\}\subseteq E$.}
\]
 The reduced Orlik--Solomon algebra $\redOS^\bullet(L)$ is the differential subalgebra generated by the submodule $\{\sum a_i \be_i : \sum a_i = 0\} \subseteq \OS^1(L)$.

The rest of the section is devoted to proving Statement~(ii) of Theorem~\ref{thm:log-vanishing}, which is a comparison theorem across the reduced Orlik--Solomon algebra, the global logarithmic forms, and the logarithmic de Rham cohomology of $(W_L,D_L)$.

\begin{theorem}\label{thm:ring-iso}
    The following three graded-commutative $\kk$-algebras are isomorphic
    \begin{enumerate}
        \item The reduced Orlik--Solomon algebra $\redOS^\bullet(L)_{\kk}$ of $L$.
        \item The exterior algebra $H^0(\Omega^\bullet_{W_L}(\log D_L))$ graded by the na{\"i}ve filtration.
        \item The logarithmic de Rham cohomology $\bigoplus_i \HH^i(\Omega^\bullet_{W_L}(\log D_L))$ with cup product multiplication.
    \end{enumerate}
\end{theorem}

For starters, note that the linear forms $x_i$ are unique up to homothety. We shall construct a map from the Orlik--Solomon algebra to the space of global logarithmic forms that is invariant under scalings of $x_i$.

Since $L$ intersects the torus $T \subseteq \kk^E$ nontrivially, denote by $L^\circ$ the intersection $L\cap T$ and $\Xi = (x_1,\dots,x_n)$ the map $L^\circ\to T$. Consider the assignment \[
    \ell \mapsto (d\log \Xi(\ell) \colon T_\ell L^\circ \xrightarrow{D_\ell\Xi} T_{\Xi(\ell)}T \xrightarrow{\cdot \Xi(\ell)^{-1}} T_{1}T = \kk^E).
\] where $\ell \in L$.
This defines the \emph{logarithmic Gauss map}: \[
    \Gamma^\circ_L\colon \P L^\circ \to \Gr(r,\kk^E),\quad [\ell]\mapsto \im d\log \Xi(\ell).
\]
See \cite[Proposition~3.2.3]{KapranovChowQuotient} and \cite[\S2]{HackingKeelTevelev} for details. It is not hard to check that the log Gauss map $\Gamma^\circ_L$ is given by $[\ell]\mapsto [\ell^{-1}L]$. Therefore $\Gamma^\circ_L$ extends to the wonderful variety $W_L$ as the composite \[
    \Gamma_L\colon W_L\subseteq X_E\xrightarrow{\phi_L} \Gr(r,\kk^E),
\] where $\phi_L$ is given in Eq.\eqref{eq:pre-log-gauss}.

The system of coordinates $\Xi$ on $T$ induces a trivialisation of the cotangent bundle of $T$ with coframe\[
    d \log x_1,\dots,d\log x_n \in T^\vee_1T = H^0(\Gr(r,T_1T),\scs^\vee),\quad\text{where }d\log x_i \coloneqq \frac{d x_i}{x_i}.
\]
For convenience, also denote by $d \log x_i$ the pullback of $d\log x_i$ along $\Gamma_L$. We consider the map
\begin{equation}\label{eq:os-to-SL}
    \nu\colon \OS^\bullet (L)_\kk \to H^0(W_L,\bigwedge^\bullet \scs_L^\vee);\quad \be_i \mapsto d\log {x}_i.
\end{equation}
\begin{lemma}
    The map $\nu$ is a well-defined homomorphism of commutative differential graded algebra.
\end{lemma}
\begin{proof}
    The standard argument, see for example~\cite{OrlikTerao,levineOSalgebra}, carries over \textit{mutatis mutandis}. We give a sketch for completeness.  If $S = \{s_1,\dots,s_k\}$ is a dependent set. Then there is a nontrivial linear relation $x_{s_k} = \sum_{1\le i\le k-1} c_i x_{s_i}$ so that one can write \[
        d \log x_{s_k} = \sum_{1\le i\le k-1}\frac{c_i x_{s_i}}{x_{s_k}}d\log x_{s_i}.
    \] Then one computes
    \begin{align*}
        \nu(\partial \be_S) & = \sum_{1\le j\le k} (-1)^{j-1} (d\log x_{s_1}\wedge \dots\wedge d\log\widehat{x}_{s_j}\wedge \dots\wedge d\log x_{s_k})\\
        & = \Big((-1)^{k-2}\sum_{1\le j\le k-1} \frac{c_j  x_{s_j}}{ x_{s_k}} + (-1)^{k-1}\Big)\; d\log x_{s_1}\wedge\dots \wedge d\log x_{s_{k-1}}= 0.
    \end{align*}
    Therefore $\nu$ descends to an graded-commutative algebra homomorphism. It is also clear that the  map $\nu$ respects the differentials.
\end{proof}
\begin{proposition}
    The map $\nu$ is an isomorphism.
\end{proposition}
\begin{proof}
    We have the following commuting diagram:
    \[
        \begin{tikzcd}
            {\bigwedge^\bullet (\kk^E)^\vee} & {\OS^\bullet(L)_\kk } \\
            {H^0(X_E,\bigwedge^\bullet \scs_L^\vee)} & {H^0(W_L,\bigwedge^\bullet \scs_L^\vee)}
            \arrow[two heads, from=1-1, to=1-2]
            \arrow[two heads, from=1-1, to=2-1]
            \arrow["\nu",from=1-2, to=2-2]
            \arrow[two heads, from=2-1, to=2-2]
    \end{tikzcd}\]
    where the left vertical arrow comes from applying the exterior algebra functor to $(\kk^E)^\vee \otimes \sco_{X_E}\to \scs_L^\vee$ and taking global section. The bottom arrow comes from restriction of global sections.

    We say a word about surjectivities in the diagram. First, there is a filtration on the kernel $\bigwedge^k (\kk^E)^\vee \to \bigwedge^k \scs_L^\vee$, where the associated graded is $\bigwedge^a \scs_L^\vee \otimes \bigwedge^{b} \scq_L^\vee$ for $a+b = k$ and $0 \le a \le k-1$. By Theorem~\ref{thm:SQvanishing}, all the associated graded have vanishing $H^1$, so the kernel has no $H^1$ either. This gives surjectivity of the left vertical arrow. Next, considering twisting the Koszul resolution in Proposition~\ref{prop:Koszul-W} by $\bigwedge^k \scs_L^\vee$, one deduces from Fact~\ref{fact:lazarsfeld} and Theorem~\ref{thm:SQvanishing} that the bottom horizontal map has zero cokernel. Therefore, the map $\nu$ has to be surjective.

    By Corollary~\ref{cor:char-poly}, we have that $h^0(W_L,\bigwedge^k \scs_L^\vee)$ computes the characteristic polynomial of the matroid of $L$. 
    However, the coefficient of the characteristic polynomial is also the dimension of $\OS^k(L)_\kk$~\cite[Corollary~2.11,Corollary~4.6]{YuzvinskiiBook}. Hence $\nu$ is an isomorphism.
\end{proof}
\begin{proof}[Proof of Theorem~\ref{thm:ring-iso}]
    By Theorem~\ref{thm:log-vanishing-restate} and the Hodge--de Rham spectral sequence, we immediately deduce that Item~(ii) is isomorphic to Item~(iii). It suffices to show Item~(i) is isomorphic to Item~(ii).
    The generator $d\log  x_i$ is dual to the derivation $ x_i \frac{\partial}{\partial  x_i}$.

    Combining Corollary~\ref{cor:char-poly} and the exterior powers applied to Sequence~\eqref{eq:Euler-seq}, we get that $\sum_k h^0(\Omega^k_{W_L}(\log D_L)) u^k $ computes the reduced characteristic polynomial of the matroid of $L$. But this is also the dimension of $\redOS^k(L)_\kk$. Therefore it is enough to show $\nu$ restricts to an injection $\nu'\colon \redOS^\bullet(L)_\kk \to H^0(\Omega^\bullet_{W_L}(\log D_L))$.

    Taking cohomology of Sequence~\eqref{eq:Euler-seq} \[
        0\to H^0(\Omega^1_{W_L}(\log D_L)) \to H^0(W_L,\scs_L^\vee) \to H^0(\sco_{W_L}) \to 0,
    \] where the third map map is now given by contraction against the Euler vector field $\sum_i  x_i \frac{\partial}{\partial  x_i}$. Therefore the kernel of contraction is subspace \[
        H^0(\Omega_{W_L}^1(\log D_L))=\Big\{\sum_i a_i d\log x_i: \sum_i a_i = 0 \Big\},
    \] but this is also the image of $\redOS^1(L)_\kk$ under $\nu$ by definition. Since $\redOS^\bullet(L)_\kk$ is generated in degree $1$, we have the desired inclusion $\nu ( \redOS^\bullet(L)_\kk ) \subseteq H^0(\Omega^\bullet_{W_L}(\log D_L))$.
\end{proof}
\begin{example}[Boolean arrangements]
    Let $\kk$ be a field and $L = \kk^E$. By definition, the arrangement complement is $\P T$ and the wonderful compactification $W_L$ is the permutohedral variety $X_E$. The boundary divisor of the compactification is the toric boundary $D_E$ of $X_E$. Recall that $N_E$ is the cocharacter lattice of $\P T$, and denote by $N_{E,\kk} = N_E\otimes \kk$. Since $X_E$ is a smooth projective toric variety, the bundle of log $1$-forms $\Omega_{X_E}^1(\log D_E)$ is trivial, with a canonical identification $\Omega^1_{X_E}(\log D_E) \simeq N_{E,\kk} \otimes \sco_X$; see~\cite[\S8.1]{CLS}. Taking global section, one deduces that $H^0(\Omega_{X_E}^\bullet(\log D_E))\simeq \bigwedge^\bullet N_{E,\kk}$.

    Moreover, the logarithmic de Rham complex is given by \[
        0\to \sco_{X_E} \to N_{E,\kk} \otimes \sco_{X_E} \to \bigwedge^2 N_{E,\kk} \otimes \sco_{X_E} \to  \bigwedge^3 N_{E,\kk} \otimes \sco_{X_E} \to \cdots
    \] with trivial differential; therefore, the log algebraic de Rham cohomology $\HH^\bullet(\Omega_{X_E}^\bullet(\log D_E))$ is also given by $\bigwedge^\bullet N_{E,\kk}$.

    Finally, the matroid associated to $L$ is the Boolean matroid $U_{n+1,E}$ on the set $E$. There are no dependent sets in the matroid $U_{n+1,E}$, so the reduced Orlik--Solomon algebra $\redOS(L)$ is given by $\bigwedge^\bullet \be_E^\perp \simeq \bigwedge^\bullet N_E$, where we identify $N_E$ and $\be_E^\perp$ using the standard dot product on $\ZZ^E$. Reducing to $\kk$ gives back the same algebra as in the previous two paragraphs.
\end{example}
\begin{example}[Rational points on the projective line]
    Let $q$ be a prime power and $k = \bar{\FF}_q$ be an algebraic closure of $\FF_q$. Let $\P L$ be the row span of the $2$-by-$(q+1)$ matrix \[
        \begin{pmatrix}
            1 & 1 &  \cdots & 1 & 0\\
            0 & 1 & \cdots & q-1 & 1
        \end{pmatrix}
    \quad\text{inside $\P^{q}_k$.}\]  Then $\P L$ is a projective line in $\P^{q}_k$ and the intersection $\P L \cap \sum H_i$ is the $0$-dimensional subscheme of $\FF_q$-points in $\P L$. In the strict transform, one still sees $W_L \simeq \P^1$ and $D_L \simeq \P^1(\FF_q)$. Let $\omega_W \simeq \sco(-2)$ be the canonical bundle of $W_L$, so that $\Omega^1_{W_L}(\log D_L) = \omega_W(D_L) = \sco_{\P^1}(q-1)$.

    Let $[x:y]$ be a standard system of coordinates on $\P^1$. The log canonical linear series $H^0(\omega_W(D))$ contains linearly independent differential forms
    \begin{equation}\label{eq:basis-PG1q}
        d \log x - d\log(cx+y),\quad \text{for all }c \in \FF_q.
    \end{equation} These differential forms give the degree-$1$ piece of the reduced Orlik--Solomon algebra $\redOS(\PG(1,q))$.

    Since $q>1$, the line bundle $\Omega^1_{W_L}(\log D_L) = \sco(q-1)$ has no $H^1$, so $H^0(\Omega^\bullet(\log D_L)) \simeq \HH^\bullet (\Omega_{W_L}^\bullet(\log D_L))$. By Riemann-Roch, the space $H^0(\omega_W(D))$ is $q$-dimensional, so the forms in Eq.\eqref{eq:basis-PG1q} also constitute a basis for $H^0(\omega_W(D))$. The computation in this example generalises readily to the case where the linear subspace $L$ is two-dimensional.
\end{example}

Before giving the next example, we first introduce a combinatorial gadget. Fix a matroid $\rmm$ on ground set $E$ with ordering $<$, we say a subset $S\subset E$ is a \textit{broken circuit} if it is of the form $C - \min C$ for some circuit $C$ of $\rmm$. An \textit{nbc-set} of $\rmm$ is a basis of $\rmm$ that contains no broken circuits of $\rmm$. Fix the ordering $0<1<\dots<n$. It is well-known that a monomial basis for the reduced Orlik--Solomon algebra is given by the set of \textit{reduced nbc-monomials},
\begin{equation}\label{eq:reduced-nbc}
    \big\{(\be_{s_1}-\be_{0})\wedge\dots\wedge(\be_{s_k}-\be_0)\mid \text{$S = \{s_1,\dots,s_k\}\subseteq [n]$ such that $S\cup 0$ is an \textit{nbc}-set.}\big\}
\end{equation}

See \cite[Section~3.2]{OrlikTerao} for details.

\begin{example}[Linear series of cubics passing through the Fano plane]\label{ex:serre-example}
    We work in characteristic $2$. This example is due to Serre and popularised by Hartshorne in~\cite[Exercise~III.10.7]{Har}. Consider the linear series $|U|\subseteq |\sco(3)|$ of plane cubics in $\P^2$ containing $Z=\P^2(\FF_2)$. Then $|V|$ is $2$-dimensional with base locus $Z$. Let  $W\subset \P^2 \times \P V^\vee$ be the graph of the evaluation map $\phi_{|V|}\colon \P^2\dashrightarrow \P V^\vee$. Then the projection to the first factor $p_1\colon W\to \P^2$ is given by blowing up $\P^2$ at $Z$, so $W$ is the wonderful variety associated to arrangement of the lines defined over $\FF_2$ in $\P^2$; let $\Delta$ be the divisor given by sum of these lines. The matroid associated to $W$ is the Fano matroid $\PG(2,2)$. Denote by $\alpha$ the pullback of a general hyperplane, and $E$ the exceptional divisor of the blow up $p_1$. The boundary divisor of the compactification $W$ is $D=p_1^{-1}\Delta +E$.

    We can directly deduce that $\Omega^i_W(\log D)$ has no higher cohomology, for $i=0,1,2$. For $i=0$, the assertion is trivial. By \cite[Proposition~11.1]{LangerDeligneLusztig}, there is a short exact sequence
    \begin{equation}\label{eq:langer-ses}
        0\to \sco_W(3\alpha-E)\to \Omega_{W}^1(\log D ) \to \sco_W(\alpha) \to 0.
    \end{equation}
    Denote by $I_Z$ the ideal sheaf $\bigoplus_{p\in Z} \mathfrak{m}_P$ of $Z$. By \cite[Theorem~4]{Artin66}, we have that $R^i p_{1,\ast} \sco(-E)$ equals to $I_Z$ when $i=0$ and vanishes for $i>0$, so by the Leray spectral sequence and projection formula, we get $H^i(W, \sco(3\alpha-E)) \simeq H^i(\P^2, I_Z(3))$. The linear series $|V|$ is given $V = H^0(I_Z(3))$, which is $3$-dimensional. Furthermore, we have $h^0(\sco(3))=10$ and $h^0(\sco_Z(3))=7$. Combining the previous two sentences and counting dimensions gives $H^{>0}(I_Z(3))=0$. Then, taking cohomology of Sequence~\eqref{eq:langer-ses}, one deduces the claim for $i=1$. The argument for the case $i=2$ is the similar, which we will omit. Applying the claim to the Hodge--de Rham spectral sequence~\eqref{eq:HodgeDeligne}, one get the isomorphism $H^0(\Omega_W^\bullet(\log D)) \simeq \HH^\bullet (\Omega_W^\bullet(\log D))$.

    Next, we compare the Betti numbers of $H^0(\Omega_W^\bullet(\log D))$ with that of $\redOS(\PG(2,2))$. Combining Sequence~\eqref{eq:langer-ses} and the vanishing statements from the previous paragraph, one deduces that
    \begin{align*}
        h^0(\Omega^1_W(\log D)) & = h^0(I_Z(3)) + h^0(\sco(1)) = 3+3 = 6\text{, and }\\
        h^0(\Omega^2_W(\log D)) & = h^0(I_Z(4)) = 8.
    \end{align*}

    \begin{figure}
        \centering
        \begin{tikzpicture}[scale=2]
            \node[shape=circle,draw=black,scale=.75] (A) at (0,0) {6};
            \node[shape=circle,draw=black,scale=.75] (B) at (90:1) {1};
            \node[shape=circle,draw=black,scale=.75] (C) at (210:1) {0};
            \node[shape=circle,draw=black,scale=.75] (D) at (-30:1) {2};
            \node[shape=circle,draw=black,scale=.75] (E) at (30:0.5) {3};
            \node[shape=circle,draw=black,scale=.75] (F) at (-210:0.5) {5};
            \node[shape=circle,draw=black,scale=.75] (G) at (-90:0.5) {4};
            \draw (A)--(B)--(F)--(C)--(G)--(D)--(E)--(B);
            \draw (A)--(E);
            \draw (A)--(F);
            \draw (A)--(G);
            \draw (A)--(D);
            \draw (A)--(C);
            \draw (E) edge[bend right=60] (F);
            \draw (F) edge[bend right=60] (G);
            \draw (G) edge[bend right=60] (E);
        \end{tikzpicture}
        \caption{Fano plane $\check{\P}^2(\FF_2)$}
        \label{fig:fano-plane}
    \end{figure}
    In Figure~\ref{fig:fano-plane} , the labelled nodes represent points in $\check{\P}^2(\FF_2)$, \textit{i.e.}, lines in $\P^2$ defined over $\FF_2$. The edges represent hyperplanes in $\check{\P}^2(\FF_2)$; two nodes share an edge if the corresponding lines in $\P^2$ meet at a point. For example, Line $3,4,$ and $5$ meet at a point. Under the ordering $0<1<\dots<6$, the \textit{nbc}-sets are
    \begin{align*}
        0,1,2,3,4,5,6, & \\
        0\ul 1,0\ul 2,0\ul 3,0\ul 4,0\ul 5,0\ul 6, & 12,13,14,34,25,35,16,26,\\
        &   0\ul{12},0\ul{13},0\ul{14},0\ul{16},0\ul{25},0\ul{26},0\ul{34},0\ul{35}.
    \end{align*}
    Here, the underlined subsets of the \textit{nbc}-sets are subsets $S$ that give rise to reduced \textit{nbc}-monomials in Eq.\eqref{eq:reduced-nbc}. Therefore, the vector space $\redOS^i(\PG(2,2))=H^0(\Omega^i_W(\log D))$ for $i=1,2$ are respectively generated by logarithmic forms
    \begin{align*}
        & d\log \frac{ x_1}{ x_0},\; d\log \frac{ x_2}{ x_0},\;\dots,\; d\log \frac{ x_6}{ x_0},\quad\text{and,}\\
        & d\log \frac{ x_1}{ x_0}\wedge d\log \frac{ x_2}{ x_0}, \;d\log \frac{ x_1}{ x_0}\wedge d\log \frac{ x_3}{ x_0},\;\dots,\;d\log \frac{ x_3}{ x_0}\wedge d\log \frac{ x_5}{ x_0},
    \end{align*}
    where we write $d\log  x_i/ x_j = d\log  x_i - d\log  x_j$.
\end{example}
\begin{remark}[Works of Langer and Gro{\ss}e-Kl{\"o}nne]
    The space $W$ in Example~\ref{ex:serre-example} is the compactification of a Deligne--Lusztig variety inside the variety $\Fl_3$ of complete flags in $\kk^3=\FF_2^3$~\cite[Proposition~7.1]{LangerDeligneLusztig}. The set of logarithmic forms we have produced  coincides with the basis of the logarithmic de Rham cohomology that Gro{\ss}e-Kl{\"o}nne gave in~\cite{GrosseKloenne05}.
\end{remark}
\subsection{Other compactifications}\label{subsec:other-compactification}
Recall that over the complex numbers, the right-hand of sequence~\eqref{eq:HodgeDeligne} computes the Betti cohomology $H^k_\sing(\P L^\circ)_\CC$ of the arrangement complement $\P L^\circ$, and the mixed Hodge structure on $H^k_\sing(\P L^\circ)$ is independent of the choice of the compactification~\cite[Th{\'e}or{\`e}me~3.2.5.(ii)]{deligneHodgeII}. Furthermore, by \cite[Proposition~3.1.8]{deligneHodgeII}, given any smooth projective compactification $\ol{\P L^\circ}$ with simple normal crossings boundary $D$, the logarithmic de Rham cohomology $\HH^\bullet(\ol{\P L^\circ},\Omega_{\ol{\P L^\circ}}^\bullet(\log D))$ computes the reduced Orlik--Solomon algebra $\redOS(L)$ of the linear space $L$. 

When the ground field is no longer $\CC$, this is the most general form of an analogous statement we find:

\begin{corollary}\label{cor:factorisation}
    Keeping the assumptions and notation as in Theorem~\ref{thm:log-vanishing-restate}.
    For $i=1,\dots,e$, let $Y_i$ be a nonsingular projective variety over $\kk$ and
    $D_i\subset Y_i$ a simple normal crossings divisor. Assume there exists a commutative diagram of blow-ups and blow-downs along centres which are strata of the pairs $(Y_i, D_i)$: \[
        \begin{tikzcd}
            & {Y_1} && {Y_3} && {Y_{e-1}} \\
            {Y_0 = W_L} && {Y_2} && \cdots && {Y_e}
            \arrow["{b_1}", from=1-2, to=2-1]
            \arrow["{b_2}"', from=1-2, to=2-3]
            \arrow["{b_3}", from=1-4, to=2-3]
            \arrow[from=1-4, to=2-5]
            \arrow[from=1-6, to=2-5]
            \arrow["b_e"',from=1-6, to=2-7]
    \end{tikzcd}\]
    so that $D_{2i+1} = b_{2i+1}^{-1} D_{2i} = b_{2i+2}^{-1} D_{2i+2}$ as reduced subschemes for all $i$. Then for all $i=0,\dots,e$, we have $H^{>0}(Y_i,\Omega_{Y_i}^\bullet(\log D_i))=0$ and \[
    \redOS^k(L)_\kk\simeq H^0(Y_i,\Omega_{Y_i}^k(\log D_i)) \simeq \HH^k(Y_i,\Omega_{Y_i}^\bullet(\log D_i)).\]
\end{corollary}
\begin{proof}
    By the proof of~\cite[Lemma~2.1]{BrionTohoku}, we have for all $i$, \[
        \Omega_{Y_{2i+1}}^1(\log D_{2i+1}) \simeq b_{2i+1}^\ast\Omega_{Y_{2i}}^1(\log D_{2i} ) \simeq b_{2i+2}^\ast\Omega_{Y_{2i+2}}^1(\log D_{2i+2}).
    \]
    One can proceed by induction on $i$. The case $i=0$ is done. Henceforth assume $i>0$. If $i$ is odd, we apply the identity $R b_{i,\ast} \sco_{Y_i} \simeq \sco_{Y_{i-1}}$ and the projection formula to get the equalities \(H^p(\Omega^q_{Y_{i-1}}(\log D_{i-1})) = H^p(\Omega^q_{Y_{i-1}}(\log D_{i-1}))\) for all $p,q\ge 0$. If $i$ is odd, we apply the identity $R^i f_\ast \sco_{Y_i} \simeq \sco_{Y_{i+1}}$ instead to get the same equalities. In both cases, the assertion will follow from the inductive hypothesis.
\end{proof}

An interesting class of varieties that arises this way include wonderful compactifications attached to arbitrary building sets in the sense of de Concini and Procesi~\cite{dcp}. Consider the linear subspace $L_m=\kk^{m-1}/(1,\dots,1)\subseteq \kk^{\binom{m-1}{2}}$, where for $1\le i<j\le m-1$, the $ij$-factor of the embedding is $v \mapsto v_i-v_j$. This gives the braid arrangement $\mathcal{B}_{m-1}$ on $L_m$. By \cite[Section~4.3]{dcp},
there exists a building set for which the wonderful compactification is the Deligne--Mumford--Knudsen moduli space $\Mdmzero$ of stable rational curves with $m$ marked points. Moreover, the wonderful variety $W_{L_m}$ is a blow-up of $\Mdmzero$ so that the inverse image of the boundary divisor $\partial \Mdmzero = \Mdmzero - M_{0,m}$ to $W_{L_m}$ equals the boundary divisor $D_{L_m}$ of $W_{L_m}$~\cite[Lemma~4.2]{bestCohomology}. The space $\Mdmzero$ is a smooth projective scheme over $\spec \ZZ$, and $\partial \Mdmzero $ is a simple normal crossings divisor relative to $\spec \ZZ$. As a matter of notation, write $\Omega^\bullet_{\log} = \Omega^\bullet_{\Mdmzero}(\log \partial \Mdmzero)$. 

\begin{corollary}\label{cor:MdmzeroSpecZ}
    Over $\spec \ZZ$, we have ring isomorphisms \[
    \redOS^k(\mathcal{B}_{m-1})\simeq H^0(\Mdmzero,\Omega^k_{\log}) \simeq \HH^k(\Mdmzero,\Omega^\bullet_{\log}).\]
\end{corollary}
\begin{proof}
    Since $\Mdmzero$ is smooth and projective over $\spec \ZZ$ and $\partial \Mdmzero$ is relative simple normal crossings over $\spec \ZZ$, the formation of the de Rham complex $\Omega_{\log}^\bullet$ and the Hodge filtration commutes with arbitrary base change. We refer to \cite[\href{https://stacks.math.columbia.edu/tag/0FM0}{Lemma 0FM0}]{stacks-project} for the nonlogarithmic case, and the argument carries over \textit{mutatis mutandis} in the logarithmic case. In particular, we have $\HH^k(\Mdmzero,\Omega^\bullet_{\log})\otimes \kk \simeq \HH^k(\Mdmzero \times \kk,\Omega^\bullet_{\log})$ for a field $\kk$.

The same can be said for $W_{L_m}$ and $\Omega_{W_{L_m}}^\bullet(\log D_{L_m}))$.
First, note that the map $\nu'$ from the proof of Theorem~\ref{thm:ring-iso} \[
\nu'\colon \redOS^\bullet(\mathcal{B}_{m-1})\to  H^0(W,\Omega_{W_{L_m}}^\bullet(\log D_{L_m}))\simeq H^0(\Mdmzero, \Omega^\bullet_{\log})\] can be constructed over $\ZZ$ and is stable under base extension. Applying Corollary~\ref{cor:factorisation}, we get that $\nu'$ is surjective upon base changing to any field. By structure theorem of finitely generated abelian groups, one deduces that $\nu'$ is surjective;  tensoring by $\QQ$ and noting the freeness of $\redOS(\mathcal{B}_{m-1})$, one also deduces that $\nu'$ has zero kernel. This uses the first isomorphism. For the second isomorphism, we similarly have natural homomorphism $H^0(\Mdmzero, \Omega^k_{\log}) \to \HH^k(\Mdmzero, \Omega^\bullet_{\log})$ arising from the Hodge--de Rham spectral sequence is also an isomorphism once base changing to a field. Now we can conclude by the exact same argument.
\end{proof}

\section{Vanishing theorems \textit{via} Kempf collapsing}\label{sec:kempf}

In this section, we adapt the approach of Kempf and Weyman to study collapsings of vector bundles. The idea of those authors is to combine syzygy calculation with considerations of positivity to analyse certain affine varieties. By contrast, our plan towards Theorem~\ref{thm:bwb-vanishing} is to exploit the fact that these affine varieties have rational singularities to obtain vanishing statements. Our convention for Schur powers follows Chapter~2 of \cite{Weyman03}.

\subsection{Recollections on the Kempf--Weyman geometric technique}\label{subsec:recollection-kempf-weyman}
For our applications the base of vector bundles will always be wonderful varieties, but we introduce the setup agnostically. Let $X$ be a projective variety, and $\sce$ a vector bundle over $X$ whose dual $\sce^\vee$ is globally generated. Denote by $V$ the affine space $H^0(X,\sce^\vee)^\vee$. The \emph{Kempf collapsing} of the vector bundle $\sce$ is the image $Y$ of the total space $\EE = \tot \sce$ under the projection $q\colon X \times V \to V$:
\[
\begin{tikzcd}
    {\EE} & { Y\subseteq V} \\
    X.
    \arrow["q",from=1-1, to=1-2]
    \arrow["p",from=1-1, to=2-1]
\end{tikzcd}
\]
Loosely speaking, $Y$ is the affine variety swept out by $\sce$ along $X$. When $X$ is smooth, the map $q$ gives a resolution of singularities of $Y$ as long as it is birational.
\begin{example}\label{ex:collapsing-of-line-bundle}
Let $\scl$ be a basepoint-free line bundle on a projective variety $X$. Put $V = H^0(X,\scl)^\vee$. Then the image of the Kempf collapsing \[
    q\colon \tot \scl^\vee \to Y\subseteq V
\] is the affine cone over the image of the evaluation map $\phi_{|\scl|}\colon X\to \P V$. Even in this basic case, the map $q$ can detect positivity properties of $\scl$.
For example, $\scl$ is ample if and only if $q$ is finite away from the origin $\{0\}$ of $V$. The forward implication follows from the finiteness of $\phi_{|\scl|}$ and a factorisation of Kempf collapsings: \[
    \begin{tikzcd}
        {\tot \scl^\vee} & {\tot \sco_{\P V}(-1) = \on{Bl}_0 V } & V \\
        X & {\P V.}
        \arrow["{\phi'}", from=1-1, to=1-2]
        \arrow[from=1-1, to=2-1]
        \arrow["\lrcorner"{anchor=center, pos=0.125}, draw=none, from=1-1, to=2-2]
        \arrow["{q_\P}", from=1-2, to=1-3]
        \arrow[from=1-2, to=2-2]
        \arrow["{\phi_{|\scl|}}", from=2-1, to=2-2]
\end{tikzcd}\]
Similarly, the backward implication can be deduced from fppf descending the finiteness of $\tot\scl^\vee \setminus X \to V\setminus \{0\}$ along the $\Gm$-torsor $V\setminus \{0\} \to \P V$.
\end{example}
\subsection{Kempf collapsings of tautological bundles}
\begin{warn}
For the remainder of the section, we work exclusively over the complex numbers.
\end{warn}
Fix a linear subspace $L\subseteq \CC^E$, in addition to a $k$-tuple of linear subspaces $\uJ = (J_1,\dots,J_k)$ of $\CC^E$. We consider the vector bundle given by $\mathscr{E}_\uJ = \scs_{J_1} \oplus \cdots
\oplus \scs_{J_k}$. It naturally embeds as a subbundle of the $T$-linearized trivial bundle
$\sco_{X_E} \otimes \mat_{k,E}$, where $\mat_{k,E}$ is the vector space of $k$-by-$E$
matrices with $T$ acting by diagonal matrices from the right by $t.M = M\cdot\on{diag}(t^{-1})$. Denote by $\EE_{L,\uJ}$ the total
space of the pullback $\mathscr{E}_\uJ |_{W_L}$ and $\hat Y_{L,\uJ}$ the Kempf collapsing of $\mathscr{E}_\uJ|_{W_L}$. When $L = \CC^E$, we simply write $\EE_\uJ = \EE_{L,\uJ}$ and $\hat Y_\uJ = \hat Y_{L,\uJ}$.

To the pair $(L,\uJ)$, we can also associate the $(k+1)$-tuple $L\sqcup\uJ$, given by inserting $L$ into the beginning of $\uJ$. We will be concerned with the following two Kempf collapsings,
\begin{equation}\label{eq:Kempf}
\begin{tikzcd}
    {\EE_{L,\uJ}} & { \hat Y_{L,\uJ}\subseteq \mat_{k,E}} \\
    W_L,
    \arrow["q_W",from=1-1, to=1-2]
    \arrow["p_W",from=1-1, to=2-1]
\end{tikzcd}\qquad
\begin{tikzcd}
    {\EE_{L\sqcup\uJ}} & { \hat Y_{L\sqcup\uJ}\subseteq \mat_{k+1,E}} \\
    X_E.
    \arrow["q_X",from=1-1, to=1-2]
    \arrow["p_X",from=1-1, to=2-1]
\end{tikzcd}
\end{equation}

On the right-hand side the image of the collapsing is $T$-invariant. More precisely, $\hat Y_{L\sqcup \uJ}$ is the
$T$-orbit closures of $(k+1)$-by-$E$ matrices of the form $(\ell,\ell_1,\dots,\ell_k)$, where $\ell \in L$ and
$\ell_i \in J_i$. Under the quotient $\mat_{k+1,E} \ratto (\P^k)^E$, the affine variety $\hat Y_{L\sqcup \uJ}$ is realised as the
multicone of the subvariety $Y_{L\sqcup\uJ}\subseteq (\P^k)^E$ given by \[
Y_{L\sqcup\uJ} = \on{cl.}\left\{ \prod_{j \in E} [\ell^j:\ell_1^j:\cdots:\ell_k^j] \mid \ell \in
L^\circ,\ell_i \in J_i^\circ \right\},\quad \text{where $L^\circ = L \cap (\CC^\times)^E$.}
\]
Recall that for a normal variety $Y$ over $\CC$, a \textit{rational resolution} is a resolution of singularities $j\colon Y'\to Y$ such that $R^i j_\ast\sco_{Y'} = 0$ for all $i > 0$. The variety $Y$ is said to have \textit{rational singularities} if it admits a rational resolution. Once $Y$ has rational singularities, then any resolution of singularities of $Y$ is rational.

\begin{lemma}\label{lem:YLEratsing}
The variety $Y_{L\sqcup\uJ}$ is normal, Cohen-Macaulay, and has rational singularities.
\end{lemma}
\begin{proof}
By \cite[Theorem~1.1]{BinglinLi2018}, the projective variety $Y_{L\sqcup\uJ}$ is multiplicity-free (the exact formula is not necessary for us, except that the weights are $0$-$1$). The multiplicity criterion of Brion~\cite[Theorem~1]{Brion2003}, see also \cite[Theorem~4.3]{BergetFinkKtheory}, yields the assertion.
\end{proof}

We will use this to show:
\begin{lemma}\label{lem:YLuJratsing}
Assume that $L\subseteq \CC^E$ is a linear subspace satisfying Condition~\eqref{eq:ll}. Then the variety $\hat Y_{L,\uJ}$ is normal and has rational singularities.
\end{lemma}
\begin{proof}
Denote by $S = \tot \scs_{L}$. By assumption and Proposition~\ref{prop:degeneracy-locus}.(i), the wonderful variety $W_L$ is the degeneracy locus of the section $1_T \in \Gamma(\scq_L)$. Base changing the Cartesian diagram in Proposition~\ref{prop:degeneracy-locus}.(i) by $\EE_{\uJ} \to X_E$, we have the following diagram where all squares
are Cartesian:
\[
    \begin{tikzcd}
        {\EE_{L\sqcup\uJ} = S \times_{X_E} \EE_\uJ } & {\A^E_{X_E} \times_{X_E} \EE_\uJ} & {\mat_{k+1,E}} & {\A^E} \\
        {\EE_{L,\uJ} = W_L \times_{X_E} \EE_\uJ} & {\EE_\uJ} & {\mat_{k,E}} & 1
        \arrow["{h_2}"', from=1-1, to=1-2]
        \arrow["{q_X}"', bend left=20, from=1-1, to=1-3]
        \arrow["{h_1}"', from=1-2, to=1-3]
        \arrow["{\on{pr}_1}"', from=1-3, to=1-4]
        \arrow[from=2-1, to=1-1]
        \arrow["{h_2'}"', from=2-1, to=2-2]
        \arrow["{q_W}"', bend right=20, from=2-1, to=2-3]
        \arrow["{1\times_{X_E}\EE_\uJ}"', from=2-2, to=1-2]
        \arrow["{h_1'}"', from=2-2, to=2-3]
        \arrow[from=2-3, to=1-3]
        \arrow[from=2-3, to=2-4]
        \arrow[from=2-4, to=1-4]
\end{tikzcd}\]
Here, $\on{pr}_1$ projects a $(k+1)$-by-$E$ matrix to its first row. The map $h_1$ is the Kempf collapsing map of the bundle $\sco_{X_E}^E \oplus \mathscr{E}_{\uJ}$, and $h_2$ is induced by the inclusion $\mathscr{E}_{L\sqcup\uJ} \subseteq
\sco_{X_E}^E \oplus \mathscr{E}_{\uJ}$. The image of $q_X$ is $\hat Y_{L\sqcup\uJ}$, and the image of $q_W$ is $\hat Y_{L,\uJ}$.

Now, we coordinatise $\mat_{k+1,E} = \spec A$ with $A=\CC[x_0,\dots,x_n]\otimes
\CC[y_{ij},i \in [k],j\in E]$ so that the first row of $\mat_{k+1,E}$ has the entries $x_i$ and the remaining
$\mat_{k,E}$ factor has entries $y_{ij}$. Let $0$ be the apex of $\hat Y_{L,\uJ}$. There is a quotient map
$\pi\colon\hat Y_{L\sqcup\uJ} \setminus 0 \to Y_{L\sqcup\uJ}$. We claim that $\pi$ restricts to an isomorphism
$\pi'\colon \hat Y_{L,\uJ} \to U \cap Y_{L,\uJ}$, where $U$ is the basic open chart of $(\P^k)^E$ given by demanding $x_i \ne
0$ for all $i$. To see this, we first note the Cartesian diagram shows the ideal $\on{I}(\hat Y_{L,\uJ})\subseteq \CC[y_{ij},i \in [k],j\in E]$ is given dehomogenisation $\on{I}(\hat Y_{L\sqcup \uJ})^{\on{dehom}}$ of the ideal $\on{I}(\hat Y_{L\sqcup \uJ})\subseteq A$ with respect to all $x_i$. Moreover, $\pi'$ is induced by the isomorphism that identifies
dehomogenisation and graded localisation; \[
    k[U\cap Y_{L\sqcup\uJ}]=(A/\on{I}(\hat Y_{L\sqcup \uJ}))_{(x_0\cdots x_n)} \xrightarrow{\simeq}
    \CC[y_{ij},i \in [k],j\in E]/\on{I}(\hat Y_{L\sqcup \uJ})^{\on{dehom}}; \quad y_{ij}/x_j \mapsto y_{ij},
\] where the $(-)_{(x_0\cdots x_n)}$ is the $\ZZ^E$-multidegree-$(0,\dots,0)$ piece of the localization away from $(x_0\cdots x_n)$.

Now, by Lemma~\ref{lem:YLEratsing}, we know that $Y_{L\sqcup\uJ}$ is normal and has rational singularities, both of which are local properties. Hence the open subvariety
$U\simeq \hat Y_{L,\uJ}$ is normal and has rational singularities.
\end{proof}

Set $r = \dim L$ and $r_i = \dim J_i$ for $i\in [k]$. A consequence of $\hat Y_{L,\uJ}$ having rational singularities is that we get cohomology vanishing as soon as $\hat Y_{L,\uJ}$ has expected dimension.

\begin{lemma}\label{lem:expected-dim-vanishing}
Assume $L$ satisfies Condition~\eqref{eq:ll} and $\hat Y_{L,\uJ}$ has dimension $r-1+ \sum  r_i$. Then we have
\begin{equation}
    H^i(W_L, \sym^j \mathscr{E}_\uJ^\vee) = 0,\quad\text{for all $i>0, j\ge 0$},
\end{equation} and $H^0(W_L,\sym \mathscr{E}_\uJ^\vee)$ is the coordinate ring of $\hat Y_{L,\uJ}$.
\end{lemma}
\begin{proof}
Denote $\EE = \EE_{L,\uJ}$ and $\hat Y = \hat Y_{L,\uJ}$ for simplicity. When $\hat Y$ has the expected dimension, by Chevalley's theorem $q_W$ is generically finite. But the fibre $\EE_m$ of $q_W$ over the point $m \in \mat_{k,E}$ is given by the linear space closure \[
    \EE_{m} = \on{cl.} \{ ([x],y) \in \P L^\circ \times \prod_i J_i \mid mx = y \} \subseteq W_L
    \times \mat_{k,E}.
\] Since $L$ is not contained in any coordinate hyperplanes, a fibre $\EE_{m}$ for a general $m\in \mat_{k,E}$ is connected and nonempty, so, by finiteness, a single point. Therefore, $q_W$
is birational. Since $q_W$ is a closed embedding followed by projection away from $W_L$, it is proper and
therefore a desingularisation. Since $\hat Y$ has rational singularities by Lemma~\ref{lem:YLuJratsing}, the map $q_W$
is necessarily a rational resolution, meaning $R^{>0} q_{W *}\sco_{\EE} = 0$. But now the first assertion follows from
\cite[Theorem~5.1.2(b)]{Weyman03}, which gives an identification \[
    H^i(W_L, \sym \mathscr{E}_\uJ^\vee) = R^i q_{W *}\sco_{\EE},
\] for any $i$.  By \cite[Theorem~5.1.2(a)]{Weyman03}\@\xspace, we know that the graded algebra $H^0(W_L, \sym \mathscr{E}_\uJ^\vee)$ is the normalization of the coordinate ring of $k[\hat Y]$. Since $\hat Y$ is normal, the last assertion follows.
\end{proof}
Next, we need a general result that helps removing the assumption on dimension of $\hat Y_{L,\uJ}$.
\begin{lemma}\label{lem:beta-trick}
Let $\mathscr{E}$ be a vector bundle on $W_L$ whose dual is globally generated. Let $\EE^\flat$ be the total space of $\mathscr{E}\oplus\sco(-\beta)$. Then, the Kempf collapsing \[
    q_{W}\colon \EE^\flat \to Y^\flat \subseteq V \times \A^E
\] is birational into its image $Y^\flat$, where $V$ is the affine space $H^0(W_L,\mathscr{E}^\vee)^\vee$.
\end{lemma}
\begin{proof}
We simply produce a rational inverse with the map $(\ell,e)\mapsto ([\ell],(\ell,e))$ for $\ell\in \P L^\circ$ and $e \in \mathscr{E}|_{[\ell]}$.
\end{proof}
\begin{proof}[Proof of Theorem~\ref{thm:bwb-vanishing}]
Let $d_i$ by the number of parts of the partition $\lambda_i$. We consider the tuple
\begin{equation}\label{eq:big-tuple}
    \uJ = (J_1,\dots,J_1,\dots,\underbrace{J_i,\dots,J_i}_{d_i},\dots,J_k,\dots,J_k)
\end{equation}
In other words, we take on $d_i$ copies of $\scs_{J_i}$ for each $i\in [k]$. Since we are working in characteristic $0$, Cauchy's formula gives a direct sum decomposition;
see \cite[Corollary~2.3.3]{Weyman03}, \[
    \sym (\scs^\vee_{J})^{\oplus d} = \bigoplus_{\lambda} \schur^{\lambda} \scs^\vee_{L}
    \otimes \schur^{\lambda} \CC^d
\] where $\lambda$ runs over all partitions with at most $d$ parts.  Applying this to
$\sym (\scs_{J_i}^\vee)^{\oplus d_i}$, one obtains \[
    \sym \mathscr{E}^\vee_\uJ = \bigoplus_{\mu_1,\dots,\mu_k}
    C^{d_1,\dots,d_k}_{\mu_1,\dots,\mu_k}\otimes \schur^{\mu_1} \scs_{J_1}^\vee \otimes
    \cdots \schur^{\mu_k} \scs_{J_k}^\vee
\] where $C^{d_1,\dots,d_k}_{\mu_1,\dots,\mu_k} = \bigotimes_{i} \schur^{\mu_i} \CC^{d_i}$, which is nonzero when $\mu_i$ has at most $d_i$ parts.

Consider $(k+1)$-tuple $\uJ^\flat = \uJ\sqcup\CC.1_T$ by append the span of the all-one vector $1_T$. By construction, we have $\mathscr{E}_{\uJ^\flat} = \mathscr{E}_{\uJ}\oplus \sco(-\beta)$. Therefore, it is enough to show $\sym \mathscr{E}_\uJ^\vee|_{W_L}$ has vanishing higher
cohomology. But Lemma~\ref{lem:beta-trick} asserts that the Kempf collapsing map is birational. The result then follows from the assumption on $L$ and the first assertion in Lemma~\ref{lem:expected-dim-vanishing}.
\end{proof}
\subsection{}
The proof of our second vanishing result here (Theorem~\ref{thm:manivel-vanishing}) is surprisingly quick. We start with a counterpart of Lemma~\ref{lem:expected-dim-vanishing}:
\begin{lemma}\label{lem:griffiths-vanishing}
Consider a smooth projective variety $X$ over $\CC$ and a vector bundle $\mathscr{E}$ over $X$ whose dual $\mathscr{E}^\vee$ is globally generated. Let $V = H^0(X,\sce^\vee)^\vee$ and  $Y \subseteq V$ be the image of the Kempf collapsing $q_X\colon \EE \to V$, where $\EE=\tot\sce$. Suppose $q_X$ is generically finite, then \[
    H^i(\omega_{X} \otimes \det \sce^\vee \otimes \sym \sce^\vee) = 0
\] for $i>0$.
Moreover, if $Y$ has rational singularity and $q_X$ is birational, the relative trace map $\tr_{\EE/Y}$ gives an isomorphism \[
    \tr_{\EE/Y}\colon H^0(\omega_{X} \otimes \det \sce^\vee \otimes \sym \sce^\vee) \to \omega_{Y},
\] as $k[Y]$-modules, where $\omega_Y$ is the dualising sheaf of $Y$.
\end{lemma}
\begin{proof}
Denote by $p_X\colon\EE \to X$ the structure map. The canonical bundle of $\EE$ is given by the pullback $\omega_{\EE} =
p_X^\ast (\omega_X \otimes \det \sce^\vee)$. By \cite[Theorem~5.1.2(b)]{Weyman03} we have identifications
\[
    R^i q_{X \ast}\omega_\EE = H^{i}(\omega_{X}\otimes \det \sce^\vee
    \otimes \sym \sce^\vee), \quad \text{for all $i$.}
\] When $q_X$ is generically finite, by Grauert--Riemenschneider~\cite[Theorem~4.3.9]{LazarsfeldPositivityI}, we have the vanishing
\(
    R^i q_{X \ast} \omega_{\EE} = 0,
\) for $i> 0$, giving the first assertion.

If moreover $Y$ has rational singularity and $q_X$ is birational. Since $\EE$ is smooth, $q_W$ is a rational resolution. Thus Kempf's criterion \cite[50]{KKMS} tells us relative trace map $\tr_{\EE / Y}\colon q_{X \ast} \omega_\EE \to \omega_Y$ is an isomorphism. This gives the second assertion.
\end{proof}
\begin{remark}
The first assertion in the lemma is true when $\sce^\vee$ is replaced by a nef and big vector bundle---this is the vanishing theorem of Griffiths; see~\cite[Theorem~7.3.1 and Example~7.3.3]{LazarsfeldPositivityII}. 
\end{remark}
\begin{example}
Consider $\sce  = \mathscr{E}_{L,L} = \scs_L|_{W_L}$ with $X=W_L$ in the lemma above. The Kempf collapsing is an open subvariety of the Schubert variety $Y_{L,L}$ of the pair $(L,L)$ in the sense of~\cite{BergetFinkExternalActivity}. Applying Serre duality and looking at the
degree-$0$ piece of $\sym \mathscr{E}$, one recovers the known case of strengthening of Speyer's $f$-vector conjecture, proposed in \cite[Remark~1.8]{bestCohomology}, \[
    H^i(-K_{W_L}-D_L) = 0,\quad\text{for $i < r-1$.}
\]  The nonstrengthened version is solved in full generality by~\cite{BergetFinkExternalActivity}. The quantity $h^{r-1}(-K_{W_L}-D_L)$ is, up to sign twists, the $\omega$-invariant of the matroid $\rmm(L)$ of $L$ in the sense of~\cite{fink2024omegainvariantmatroid}; see~\cite[Proposition~5.2]{eur2024ktheoreticpositivitymatroids}.
\end{example}

\begin{proof}[Proof of Theorem~\ref{thm:manivel-vanishing}]
As before, we consider the tuple $\uJ$ given by \[
    \uJ = (\underbrace{L,\dots,L}_{d},J_1,\dots,J_1,\dots,\underbrace{J_i,\dots,J_i}_{d_i},\dots,J_k,\dots,J_k).
\]  The bundle $\scs_\uJ$ is a subbundle of the trivial bundle $\mathrm{Mat}_{N,E}\otimes \sco_{X_E}$ where $N = d+\sum_i d_i $. The quotient $\mathscr{F}_\uJ = (\mathrm{Mat}_{N,E}\otimes \sco_{X_E})/\mathscr{E}_\uJ$ is globally generated by $\mathrm{Mat}_{N,E}$. Consider the Kempf collapsing \[
    q_W\colon \tot (\mathscr{F}_\uJ^\vee \oplus \sco(-\alpha) \oplus \sco(-\beta)) |_{W_L} \to Z \subseteq  \mathrm{Mat}_{N,E}^\vee \times \A^E \times \A^E.
\]
Thanks to Lemma~\ref{lem:beta-trick}, the map $q_W$ is birational. Since $L$ satisfies Condition~\eqref{eq:ll}, we recall that the canonical bundle of $X_E$ is given by $\omega_{X_E} = \sco(-\alpha-\beta)$. By Proposition~\ref{prop:degeneracy-locus} that $\omega_{W_L} = (\omega_{X_E} \otimes \det \scq_L)|_{W_L} = \det\scq_L(-\alpha-\beta)|_{W_L}$. By Lemma~\ref{lem:griffiths-vanishing}, we have \[
    H^i(W_L, \sym [\mathscr{F}_\uJ \oplus \alpha \oplus \beta] \otimes \det \mathscr{F}_\uJ \otimes \det \scq_L)=0
\]
Finally, applying Cauchy's formula to each $\sym \scq_{J_i}^{\oplus d_i}$ yields the desired vanishing.
\end{proof}
Note that the argument for this vanishing result makes no assumption about the Kempf collapsing $Z$ beyond its dimension.

\begin{remark}[Vanishing theorem for divisor classes associated to generalised permutohedra]\label{rm:EurFinkLarson}

Let $L$ be a linear subspace of $\kk^E$ satisfying Condition~\eqref{eq:ll}, in addition to a nef divisor $D$ on $X_E$. Following \cite{eur2024ktheoreticpositivitymatroids}, we can define the section ring \[
    R(L,D) = \bigoplus_{\ell \ge 0} H^0(W_L, \ell D )
\] In recent work~\cite{eur2025vanishingtheoremsmatroids}, Christopher Eur, Alex Fink, and Matt Larson will show that
\begin{enumerate}
    \item The ring $R(L,D)$ is graded Cohen--Macaulay and generated in degree $1$.
    \item The image $Z_L$ of the evaluation map $f_{D_P}\colon W_L\to \P H^0(W_L,D)$ satisfies $R f_{D\,\ast}\sco_{W_L} = \sco_{Z_L}$.
\end{enumerate}
These authors obtained the statements via Gr{\"o}bner degenerations and Frobenius splitting.
Keeping the assumptions as in Theorem~\ref{thm:bwb-vanishing}, consider a generalised permutohedron $P$ defined by the Minkowski sum \[
    P = \rmb(J_1^\perp) + \cdots + \rmb(J_k^\perp)
\] of the base polytopes of dual matroids of $J_i$. The polytope $P$ defines a basepoint-free divisor $D_P$ on the permutohedral variety $X_E$. We have the following positivity results:
\begin{enumerate}[label=\emph{\alph*})]
\item Theorem~\ref{thm:bwb-vanishing} implies that \[
        H^i(W_L,\ell D_P) = 0,\quad \text{for $i>0$ and $\ell\ge 0$.}
    \]
\item If one of the $J_i$ equals $L$, then a combination of Theorem~\ref{thm:bwb-vanishing} and Fact~\ref{fact:lazarsfeld} shows the restriction map $H^0(X_E,\ell D_P)\to H^0(W_L,\ell D_P)$ is surjective for all $k> 0$.
\item Let $K_{W_L}$ be the canonical divisor of $W_L$. If the Kempf collapsing $q_W\colon \EE_{L,\uJ} \to Y_{L,\uJ}$ is birational, then, applying Lemma~\ref{lem:griffiths-vanishing} to $\mathscr{E}_\uJ$, we have \[
        H^i(W_L,K_{W_L}+\ell D_P)=0,\quad \text{for $i>0$ and $\ell > 0$.}
    \]
\end{enumerate}
The Kempf--Weyman technique seems less powerful in this case. We can only deduce Statement~(i) above assuming Statement~(ii) and all the conditions in Items b) and c). In this case, the assertion can be obtained from Serre duality and \cite[Proposition~4.3]{eur2024ktheoreticpositivitymatroids}.

\end{remark}

\section{Relationship to White's conjecture}
\label{sec:white}

In the present section, we show how cohomology groups of tautological bundles govern degrees of generators of toric ideals of matroid base polytopes.
\subsection{Constructing syzygies for a Kempf collapsing}
We start by giving a bit more generalities on Kempf collapsings.
Keeping the notation and assumptions as in \S\ref{subsec:recollection-kempf-weyman}, we have a short exact sequence
\begin{equation}
0\to \sce \to V\otimes \sco_X \to \scf \to 0,
\end{equation}
where $\scf$ is dual to the kernel of the evaluation map $V^\vee\otimes \sco_X \twoheadrightarrow \sce^\vee$. Denote by $S$ the coordinate ring $\sym V^\vee$ of $V$.
Weyman constructed a minimal resolution for the Kempf collapsing $Y$ of $\sce$~\cite[Theorem~5.1.2]{Weyman03}:
\begin{proposition}\label{prop:weyman-complex}
There exists a minimal complex $F_\bullet$ of $\ZZ$-graded $S$-modules, with the $i$\/th term given by \[
F_i = \bigoplus_{j\ge 0}H^j(X,\bigwedge^{i+j}\scf^\vee ) \otimes S(-i-j).
\]
If the collapsing map $q\colon \tot \sce \to Y$ is birational, then $H^0(\sym\sce^\vee)$ is a normalisation of the coordinate ring $k[Y]$ of $Y$. If moreover $\sym\sce^\vee$ has vanishing higher cohomology and $F_0 = S$, then $Y$ is normal, and the complex $F_\bullet$ gives a free resolution of $k[Y]$.
\end{proposition}

\subsection{Normal presentation of the toric ideal of a matroid}

Let $\rmm$ be a matroid on the ground set $E$ with rank $r$. We will define the \textit{toric ideal} $I_\rmm$ of $\rmm$ to be the kernel of the ring map \[
\kk[x_B : \text{$B$ is a basis of $\rmm$}] \to \kk[x_0,\dots,x_n];\quad x_B\mapsto \prod_{e\in B} y_e.
\]
It is evident that this is the homogeneous ideal cutting out the toric variety $X_{\rmb(\rmm)}$ associated to the base polytope $\rmb(\rmm)$. In the case when $\rmm$ is a realisable matroid,  the projective variety $X_{\rmb(\rmm^\perp)}$ equals, Proposition~\ref{prop:ggms}, the $T$-orbit closure of a linear subspace $L\subseteq \kk^E$ that realises the matroid $\rmm$.

Denote by $\sco_\Gr(1)$ the Pl{\"u}cker line bundle on the Grassmannian $\Gr(r,\kk^E)$, and $\sco_{\rmb(\rmm)}(1)$ the restriction of $\sco_\Gr (1)$ to $X_{\rmb(\rmm)}$.
Let $K = H^0(\Gr(r,\kk^E),\sco_\Gr(1))$ and $\scm_{\rmb(\rmm)}$ be the kernel of the composite \[
K \otimes \sco_{X_{\rmb(\rmm)}} \to H^0(X_{\rmb(\rmm)},\sco_{\rmb(\rmm)}(1)) \otimes \sco_{X_{\rmb(\rmm)}}\to \sco_{\rmb(\rmm)}(1),
\]
where the first map is given by restriction of a section of $\sco_\Gr(1)$ on $\Gr(r,\kk^E)$ to $X_{\rmb(\rmm)}$, and the second map is given by evaluating the section.
\begin{proposition}\label{thm:weak-white}
Let $\rmm$ be a connected matroid of rank or $r$. The following are equivalent for an integer $j>1$:
\begin{enumerate}
\item The toric ideal $I_{\rmb(\rmm)}$ has no degree-$(j+1)$ minimal generators.
\item $H^j(X_{\rmb(\rmm)},\bigwedge^{j+1} \scm_{\rmb(\rmm)}) = 0$.
\item $H^1(X_{\rmb(\rmm)}, \bigwedge^{2} \scm_{\rmb (\rmm)} \otimes  \sco_{\rmb(\rmm)}({j-1}) ) = 0.$
\end{enumerate}
\end{proposition}
Therefore, weak White's conjecture is true if and only if Statement~(ii) holds for all $j>1$. The equivalence between (i) and (ii) can be construed in terms of Koszul cohomology, but we will first give a proof using the Kempf--Weyman technique, which is more akin to the viewpoint of \S\ref{sec:kempf}.

\begin{proof}
For the equivalence between Item~(i) and Item~(ii), we consider Example~\ref{ex:collapsing-of-line-bundle} with the following diagram
\[
\begin{tikzcd}
{\tot \sco_{\rmb(\rmm)}(-1)|_{X_{\rmb(\rmm)}}} & {\hat X_{\rmb(\rmm)}\subseteq K} \\
{X_{\rmb(\rmm)}}
\arrow["{q}", from=1-1, to=1-2]
\arrow[from=1-1, to=2-1]
\end{tikzcd}\]
where $q$ is the Kempf collapsing of the line bundle $\sco_{\rmb(\rmm)}(-1)$. Since $\sco_{\rmb(\rmm)}(1)$ is very ample, the map $q$ is an isomorphism away from the zero section. Therefore, by Proposition~\ref{prop:weyman-complex}, the complex $F_\bullet$ gives a minimal free resolution of the normalisation of the coordinate ring $k[\hat X_{\rmb(\rmm)}]$, and hence $k[\hat X_{\rmb(\rmm)}]$ itself, since White showed $\rmb(\rmm)$ is a normal polytope~\cite[Theorem~1]{WhiteBasisMonomial}. Now, looking at $F_1$, one deduces that $I_L$ is has no minimal generators in degree $j+1$ if and only if $H^j(X,\bigwedge^{j+1}\scm_L) = 0$ for $j>1$.

For the equivalence between Item~(ii) and Item~(iii), we look at the short exact sequences \[
0\to \sco_{\rmb(\rmm)}(1)\otimes\bigwedge^{b-1} \scm_{\rmb(\rmm)} \to \bigwedge^{b} K \otimes \sco_{X_{\rmb(\rmm)}} \to  \bigwedge^{b} \scm_{\rmb(\rmm)} \to 0
\] for various $b\ge 0$. Twisting by tensor powers of $\sco_{\rmb(\rmm)}(1)$ and noting that the middle term always has vanishing higher cohomology, we obtain a sequence of isomorphisms
\begin{align*}
H^1(X_{\rmb(\rmm)}, \bigwedge^{2} \scm_{\rmb (\rmm)} \otimes  \sco_{\rmb(\rmm)}({j-1}) ) & \to \cdots \\
\to H^{j-1}(\bigwedge^{j}\scm_{\rmb(\rmm)} \otimes \sco_{\rmb(\rmm)}(1)) & \to H^j(\bigwedge^{j+1}\scm_{\rmb(\rmm)}).
\end{align*}
This gives the equivalence between Item~(ii) and Item~(iii).
\end{proof}
\begin{remark}[The equivalence between Item~(i) and Item~(iii) \textit{via} Koszul cohomology]
Let $\scm'$ be the kernel of the evaluation map \[
H^0(X_{\rmb(\rmm)},\sco_{\rmb(\rmm)}(1)) \otimes \sco_{X_{\rmb(\rmm)}}\to \sco_{\rmb(\rmm)}(1).
\]
Recall that White~\cite[Theorem~1]{WhiteBasisMonomial} proved that the toric variety $X_{\rmb(\rmm)}$ is projectively normal. Therefore, we have
\begin{equation}\label{eq:N0-for-BM}
H^1(X_{\rmb(\rmm)}, \scm' \otimes  \sco_{\rmb(\rmm)}(b))= 0,\quad\text{for $b\ge 1$.}
\end{equation}
By \cite[Lemma~1.4]{EinLazarsfeldKoszulCohomologyArbitraryDimension}, Item~(i) for $j\in \ZZ_{>0}$ is equivalent to the vanishing
\begin{equation}\label{eq:N1-for-BM}
H^1(X_{\rmb(\rmm)}, \bigwedge^{2} \scm' \otimes  \sco_{\rmb(\rmm)}({j-1}))= 0.
\end{equation} Computing extension classes, one can observe that $\scm_{\rmb(\rmm)}$ splits as the direct sum $\scm'$ with a trivial bundle,
\[
\scm_{\rmb(\rmm)}\simeq \scm'\oplus \ker[H^0(\sco_\Gr(1))\to H^0(\sco_{\rmb(\rmm)}(1))]  \otimes \sco_{X_{\rmb(\rmm)}}.
\]
Therefore, considering Eq.\eqref{eq:N0-for-BM}, we see that the condition in Eq.\eqref{eq:N1-for-BM} is equivalent to Item~(iii).
\end{remark}
\begin{remark}
    It is also true that the toric variety $X_P$ associated to any integral polymatroid base polytope $P$ is projectively normal; see~\cite[Chapter~18.6, Theorem~3]{Welsh} or \cite[Proposition~3.10]{EHLstella}. 
\end{remark}
Turning to the realisable case, we start with the following observation, which is a consequence of Proposition~\ref{prop:ggms}; see also \cite[Appendix~III]{best}.
\begin{observation}\label{obs:realisable-matroid}
When $\rmm$ is realised by a linear subspace $L\subseteq \kk^E$, the toric variety $X_{\rmb(\rmm)}$ is the image of the map $\phi_L\colon X_E \to X_{\rmb(\rmm)}$ introduced in Proposition~\ref{prop:ggms}. In particular, we have \begin{enumerate}
    \item $R\phi_{L\,\ast}\sco_{X_E} \simeq \sco_{X_{\rmb(\rmm)}}$; see~\cite[Theorem~9.2.5]{CLS}.
    \item $\det \scq_L \simeq \det \scs_{L}^\vee \simeq \phi_L^\ast \sco_{\rmb(\rmm^\perp)}(1)$.
\end{enumerate}
\end{observation}

\begin{remark}
In the case when the matroid $\rmm$ is realised by a linear subspace $L\subseteq\CC^E$, the projective normality of $X_{\rmb(\rmm)}$ can also be obtained as a consequence of Theorem~\ref{thm:bwb-vanishing}. Running a Leray spectral sequence with Observation~\ref{obs:realisable-matroid}, we obtain \[H^i(X_{\rmb(\rmm)},\sco_{\rmb(\rmm)}(\ell)) = H^i(X_{E},\det \scq_{L^\perp}^{\otimes\ell})= 0,\quad\text{for $i >0$ and $\ell\ge 0$,}
\] so the normality of $\hat X_{\rmb(\rmm)}$ and Proposition~\ref{prop:weyman-complex}.
\end{remark}
\begin{remark}[Degree bound on minimal generators of $I_\rmm$]
Keeping the notation and assumption as in Proposition~\ref{thm:weak-white}, we claim that the ideal $I_{\rmm}$ has generators in degree at most $n$. This can be obtained by combining the Batyrev--Borisov vanishing theorem~\cite[Theorem~9.2.7]{CLS} and the observation that the the polytope $\rmb$ has no interior lattice points.

In the case when the matroid is realised by a linear subspace $L\subseteq \kk^E$.
Twisting the sequence at the end of the proof of Proposition~\ref{thm:weak-white} by $\det\scs_{L^\perp} = \sco_{\rmb(\rmm)}(-1)$ yields \[
0\to \bigwedge^{j+2}\phi_L^*\scm_{\rmb(\rmm)} \otimes \det \scs_{L^\perp} \to \bigwedge^{j+2} K \otimes \det \scs_{L^\perp} \to \bigwedge^{j+1} \phi_L^\ast\scm_{\rmb(\rmm)} \to 0
\]
By Theorem~\ref{thm:SQvanishing}, the middle term has no higher cohomology. Considering Observation~\ref{obs:realisable-matroid} again, we have \[ 
H^j(X_{\rmb(\rmm)},\bigwedge^{j+1} \scm_{\rmb(\rmm)}) \simeq H^{j+1}(X_{\rmb(\rmm)},\bigwedge^{j+2}\scm_{\rmb(\rmm)}) \simeq H^{j+1}(X_E, \phi_L^*\bigwedge^{j+2}\scm_{\rmb(\rmm)})=0,\]
for $j\ge n$, since $X_E$ has dimension $n.$
\end{remark}

Let $\rmm$ be a matroid realised by a linear subspace $L\subseteq \kk^E$. Thanks to Item~(ii) of Observation~\ref{obs:realisable-matroid}, we can put a filtration on the vector bundle $\phi_L^\ast \scm_{\rmb(\rmm^\perp)}$ with associated graded \[
    \bigwedge^{r-i} \scs_L^\vee \otimes \bigwedge^{i} \scq_L \quad\text{for $0< i\le r$.}
\] 
These graded pieces also show up as the direct summands of $\bigwedge^r \sce_L^\vee$. At the same time, we also obtain an induced filtration on $\phi_L^\ast\scm_{\rmb(\rmm^\perp)}^{\otimes 2} \otimes \det \scq_L$, with associated graded \[
\bigwedge^{r-i}\scs_L^\vee \otimes \bigwedge^{i}\scq_L^\vee \otimes \bigwedge^{r-j} \scs_L^\vee \otimes \bigwedge^{\# E - r - j} \scq_L\quad\text{for all $0<i,j\le r$.}
\] It is natural to ask about higher cohomology vanishing of the associated graded of this filtration; to this end, we propose a stronger prediction regarding minimal generators of $I_{\rmm^\perp}$ in degree $3$.
\begin{conjecture}\label{conj:white-deg3}
Let $L\subseteq \CC^n$ be a linear subspace of dimension $r$, then the vector bundle \[
\bigwedge^r \sce_L^\vee \otimes \bigwedge^r \sce_L^\vee \otimes \det \scq_L
\simeq \bigoplus_{0\le i,j\le r} \bigwedge^{r-i}\scs_L^\vee \otimes \bigwedge^{i}\scq_L^\vee \otimes \bigwedge^{r-j} \scs_L^\vee \otimes \bigwedge^{\# E - r - j} \scq_L
\] has vanishing $H^1$. Here, we put $\sce_L = \scs_L \oplus \scq_L$, following the notation as in \S\ref{sec:SQproof}.
\end{conjecture}
When $\on{char}\kk\neq 2$, a positive answer to Conjecture~\ref{conj:white-deg3} would imply that $\phi_L^\ast \bigwedge^2 \scm_{\rmb(\rmm^\perp)}\otimes \det \scq_L$ has vanishing $H^1$, which implies that $I_{M^\perp}$ has no minimal generators in degree $3$. 

One may hope to attack Conjecture~\ref{conj:white-deg3} using the same strategy as we did for Theorem~\ref{thm:SQvanishing}. Unfortunately, as we shall explain in below, the strategy of working with exterior powers of a direct sum of the form \[
\Psi_{L} = \scs_L^\vee \oplus \scq_L^\vee \oplus \scs_L^\vee \oplus \scq_L.
\] is likely to fall short.
By \cite[Example~5.3]{eur2025vanishingtheoremsmatroids}, there exists a $3$-dimensional linear subspace $L\subseteq \CC^E$ for $E=\{0,1,\dots,11\}$ such that $H^1(W_L,\scq_L)\neq 0$. Considering Fact~\ref{fact:lazarsfeld}, we can find a natural number $0<a<\# E - r$ such that \[
H^{a+1}(\bigwedge^a \scq_L^\vee \otimes\scq_L) \neq 0.
\]
Here, we exclude the cases $a=0$ and $a=\# E -r = n+1-r$, since it does occur that $H^1(\scq_L)= 0$ and, by the standard Cremona transform, that \[
H^{n-r+2}(\bigwedge^{n+1-r} \scq_L^\vee \otimes \scq_L) \simeq H^{n-r+2}(\bigwedge^{n-r} \scq_L^\vee) =0.
\]
While this does not falsify Conjecture~\ref{conj:white-deg3}, it does imply there exists a natural number $d\ge 2$ such that the exterior power $\bigwedge^d \Psi_L$ has nonvanishing higher cohomology.

\appendix

\section{Vanishing result on tautological quotient bundles}\label{appendix:exterior-quotient}

We give a simplified proof of higher cohomology vanishing for exterior powers of the tautological quotient bundle $\scq_L$ as well as a deletion-contraction formula for the dimension of global sections of pushing forward exterior powers of $\scq_L$. The first assertion also can also be deduced from Theorem~\ref{thm:SQvanishing}; we view the second assertion below as a verification of deletion-contraction identity in \cite[Theorem~1.5]{bestCohomology} up to taking global sections.

\begin{proposition}
Let $L\subseteq \kk^E$ be a linear subspace. For any $d\ge 0$, we have the following
\begin{eqnarray*}
& (i). & H^{i}(X_E, \bigwedge^d \scq_L) = 0\quad \text{for $i>0$}.\\
& (ii).& h^0(X_{E\setminus n}, f_\ast \bigwedge^d \scq_L) = \\ & &
\begin{cases}
h^0(X_{E\setminus n},\bigwedge^d (\scq_{L\setminus n} \oplus \sco_{X_{E\setminus n}})), & \text{if $n$ is a loop;}\\
h^0(X_{E\setminus n}, \bigwedge^d \scq_{L / n}), & \text{if $n$ is a coloop;}\\
h^0(X_{E\setminus n}, \bigwedge^d \scq_{L/n}) + h^0(X_{E\setminus n}, \bigwedge^{d-1} \scq_{L\setminus n}), & \text{if $n$ is not a loop or coloop.}
\end{cases}
\end{eqnarray*}
\end{proposition}
\begin{proof}
Considering the Leray spectral sequence \[
E^{p,q}_2 = H^p(X_{E\setminus n}, R^q f_\ast \bigwedge^d \scq_L) \Rightarrow H^{p+q}(X_E, \bigwedge^d \scq_L).
\]
Since $\bigwedge^d \scq_L$ has vanishing cohomology, it is enough to show $E^{p,0}_2 = 0$ for $p> 0$.
We proceed by induction on $n$. In the case when $n$ is a loop or coloop, the second assertion is clear from the identity $\scq_L \simeq \scq_{L\setminus n \oplus L | n}$ and the projection formula. The first assertion also follows immediately from the inductive hypothesis.

It remains to handle the case when $n$ is neither a loop nor coloop. Taking exterior power of the short exact sequence from the third column of Diagram~\eqref{eq:diagram-upstairs} yields
\[
0 \to \scl_1 \otimes \bigwedge^{d-1} \scq_L \to f^\ast \bigwedge^d \scq_{L/n}\oplus \sco_{X_{E\setminus n}} \to \bigwedge^d \scq_{L} \to 0.
\]
Pushing forward along $f$, we have, thanks to the projection formula and Lemma~\ref{lem:derivedVanshings}, a short exact sequence
\begin{equation}\label{eq:appendix-exterior-upstairs}
0 \to f_\ast (\scl_1 \otimes \bigwedge^{d-1} \scq_L) \to \bigwedge^d (\scq_{L/n}\oplus \sco_{X_{E\setminus n}}) \to f_\ast \bigwedge^d \scq_{L} \to 0,
\end{equation}
where the first term is by Lemma~\ref{lem:pushforward-L1-wedgeQ}, \[
f_\ast (\scl_1 \otimes \bigwedge^d \scq_L) = \det \scq_{L/n} \otimes \scq_{L\setminus n}^\vee \otimes \bigwedge^{d-2} \scq_{L\setminus n}.
\] Therefore, this term fits into the $(d-1)$\/th exterior power of the third column of Diagram~\eqref{eq:diagram-downstairs} in the following way,
\begin{equation}
\label{eq:appendix-exterior-downstairs}
0\to f_\ast (\scl_1 \otimes \bigwedge^d \scq_L)  \to \bigwedge^{d-1} \scq_{L/n} \to \bigwedge^{d-1} \scq_{L\setminus n} \to 0.
\end{equation}
By the inductive hypothesis, the second and third items above have vanishing higher cohomology. Thus, taking cohomology of the sequence above, we obtain \[
H^i(X_{E\setminus n}, f_\ast (\scl_1 \otimes \bigwedge^d \scq_L)) = 0,\quad \text{for all $i> 1$ and $d\ge 0$.}
\]
Taking cohomology of Sequence~\eqref{eq:appendix-exterior-upstairs} and invoking the inductive hypothesis again, we obtain \[
H^p(X_{E\setminus n},f_\ast \bigwedge^d \scq_L) \simeq H^{p+1}(X_{E\setminus n}, f_\ast (\scl_1 \otimes \bigwedge^d \scq_L)), \quad\text{for $p>0$.}
\]
But we have shown that the right-hand side vanishes. The left-hand side equals $E^{p,0}_2$; therefore, the first assertion follows. To obtain the second assertion, we take cohomology of Sequences~\eqref{eq:appendix-exterior-upstairs} and \eqref{eq:appendix-exterior-downstairs}, obtaining two four-term exact sequences
\begin{align*}
0 & \to H^0(f_\ast (\scl_1 \otimes \bigwedge^{d-1} \scq_L)) \to H^0(\bigwedge^d (\scq_{L/n}\oplus \sco_{X_{E\setminus n}})) \to \\
& \to H^0(f_\ast \bigwedge^d \scq_{L}) \to  H^1(f_\ast (\scl_1 \otimes \bigwedge^{d-1} \scq_L))\to 0,\quad\text{and}\\
0 & \to H^0(f_\ast (\scl_1 \otimes \bigwedge^d \scq_L))  \to H^0(\bigwedge^{d-1} \scq_{L/n}) \to H^0( \bigwedge^{d-1} \scq_{L\setminus n}) \to H^1(f_\ast (\scl_1 \otimes \bigwedge^d \scq_L)) \to 0.
\end{align*}
Here, all cohomologies are taken in the variety $X_{E\setminus n}$.
Counting dimensions, we obtain
\begin{align*}
h^0(\bigwedge^d \scq_L) & = h^0(X_{E\setminus n}, f_\ast \bigwedge^d \scq_L) \\
& = h^0(\bigwedge^d (\scq_{L/n}\oplus \sco_{X_{E\setminus n}})) - h^0(f_\ast (\scl_1 \otimes \bigwedge^{d-1} \scq_L)) + h^1(f_\ast (\scl_1 \otimes \bigwedge^{d-1} \scq_L))\\
& =  h^0(\bigwedge^d (\scq_{L/n}\oplus \sco_{X_{E\setminus n}})) - h^0( \bigwedge^{d-1} \scq_{L/ n}) + h^0( \bigwedge^{d-1} \scq_{L\setminus n})\\
& =  h^0(\bigwedge^d \scq_{L/n})  + h^0(\bigwedge^{d-1} \scq_{L\setminus n}).
\end{align*}
This gives the second assertion.
\end{proof}

\printbibliography

\end{document}

%% file: header.tex
\usepackage{amsfonts, amsmath, amssymb}
\usepackage[fixamsmath]{mathtools} 
\usepackage{stmaryrd} 
\usepackage[all]{xy}
\usepackage{graphicx}

\usepackage[utf8]{inputenc}
\usepackage[T1]{fontenc}
\usepackage{lmodern}
\usepackage{mathrsfs}
\usepackage{eucal}

\usepackage[
    textwidth=15cm,
    heightrounded,
    hcentering,
    foot=1cm
]{geometry}

\usepackage{wrapfig}
\usepackage{calc}
\usepackage{enumerate}
\usepackage[dvipsnames,svgnames,table]{xcolor}
\usepackage{color}
\definecolor{e-mail}{rgb}{0,.40,.80}
\definecolor{reference}{rgb}{.20,.60,.22}
\definecolor{mrnumber}{rgb}{.80,.40,0}
\definecolor{citation}{rgb}{0,.40,.80}
\definecolor{gris25}{gray}{0.60}

\usepackage{xspace}
\usepackage[protrusion=true,expansion=true,kerning,spacing]{microtype}
\microtypecontext{spacing=nonfrench}
\usepackage[english]{babel}

\theoremstyle{plain}
\newtheorem{lemma}{Lemma}[section]
\newtheorem{theorem}[lemma]{Theorem}

\newtheorem{corollary}[lemma]{Corollary}

\newtheorem{maintheorem}{Theorem}

\newtheorem{proposition}[lemma]{Proposition}

\theoremstyle{definition}

\newtheorem{definition}[lemma]{Definition}
\newtheorem{remark}[lemma]{Remark}
\newtheorem{example}[lemma]{Example}

\newtheorem{conjecture}[lemma]{Conjecture}

\newtheorem{fact}[lemma]{Fact}
\newtheorem{observation}[lemma]{Observation}

\newtheorem{proposition.definition}[lemma]{Proposition/Definition}
\theoremstyle{remark}
\newtheorem{warn}[subsection]{Warning}

\usepackage{xparse}
\DeclareDocumentCommand{\Hurw}{O{G} O{g}} {\mathscr{M}_{#2}[#1]}
\DeclareDocumentCommand{\Mgm}{ O{g} O{m}}{\mathscr{M}_{#1,#2}}
\DeclareDocumentCommand{\Mzero}{ O{m}}{\Mgm[0][#1]}
\DeclareDocumentCommand{\Mdmgm}{ O{g} O{m}}{\overline{{M}}_{#1,#2}}
\DeclareDocumentCommand{\Mdmzero}{ O{m}}{\Mdmgm[0][#1]}

\DeclareDocumentCommand{\res}{O{G} O{H}} {\mathrm{res}^{#1}_{#2}}
\DeclareDocumentCommand{\cores}{O{G} O{P}} {\mathrm{cores}^{#1}_{#2}}

\DeclareMathOperator{\im}{im}
\DeclareMathOperator{\coker}{coker}
\DeclareMathOperator{\rk}{rk}
\DeclareMathOperator{\tot}{Tot}
\DeclareMathOperator{\sym}{Sym}
\DeclareMathOperator\rspec{\mathscr{S}\kern-.5pt \textit{pec }}
\DeclareMathOperator\rproj{\mathscr{P}\kern-.5pt \textit{roj }}
\newcommand{\ol}{\overline}
\newcommand{\ul}{\underline}
\newcommand*{\shom}{\mathscr{H}\kern -.5pt om}
\newcommand*{\send}{\mathscr{E}\kern -.5pt nd}
\newcommand*{\sder}{\mathscr{D}\kern -.5pt er}
\newcommand*{\stor}{\mathscr{T}\kern -.5pt or}
\newcommand*{\sext}{\mathscr{E}\kern -.5pt xt}
\newcommand{\QQ}{\mathbb Q}
\newcommand{\CC}{\mathbb C}
\newcommand{\ZZ}{\mathbb Z}
\newcommand{\RR}{\mathbb R}
\newcommand{\HH}{\mathbb H}
\newcommand{\kk}{\Bbbk}
\newcommand{\FF}{{\mathbb F}}
\renewcommand{\P}{\mathbb P}
\newcommand{\A}{\mathbb A}
\newcommand{\schur}{\mathbb{S}}
\newcommand{\EE}{\mathbb E}

\newcommand{\Gm}{{\mathbb G}_m}

\newcommand{\spec}{{\mathrm{Spec}\,}}

\newcommand{\inv}{{\mathrm{inv}}}
\newcommand{\crem}{{\mathrm{Crem}}}

\newcommand{\tr}{\mathrm{Tr}}
\usepackage{leftidx}

\newcommand{\be}{\mathbf{e}}
\newcommand{\uJ}{{\ul{J}}}
\newcommand{\mat}{{\mathrm{Mat}}}
\newcommand{\rmb}{\mathrm{B}}
\newcommand{\rmm}{\mathrm{M}}
\usepackage{tikz-cd}

\newcommand{\ratto}{\;\tikz{\draw[dashed, ->] (0,0) -- (0.5,0);}}

\newcommand\sca{\mathcal A}

\newcommand\sce{\mathcal E}
\newcommand\scf{\mathcal F}
\newcommand\scg{\mathcal G}
\newcommand\sch{\mathcal H}

\newcommand\sck{\mathcal K}
\newcommand\scl{\mathcal L}
\newcommand\scm{\mathcal M}
\newcommand\scn{\mathcal N}
\newcommand\sco{\mathcal O}

\newcommand\scq{\mathcal Q}
\newcommand\scs{\mathcal S}

%% file: refs.bib
@article{AHK,
  author     = {Adiprasito, Karim and Huh, June and Katz, Eric},
  title      = {Hodge theory for combinatorial geometries},
  journal    = {Ann. of Math. (2)},
  fjournal   = {Annals of Mathematics. Second Series},
  volume     = {188},
  year       = {2018},
  number     = {2},
  pages      = {381--452},
  issn       = {0003-486X,1939-8980},
  mrclass    = {05B35 (05E99 14C25 14T05)},
  mrnumber   = {3862944},
  mrreviewer = {Zvi\ Rosen},
  doi        = {10.4007/annals.2018.188.2.1},
  url        = {https://doi.org/10.4007/annals.2018.188.2.1}
}

@article{Arapura,
  author     = {Arapura, Donu},
  title      = {Frobenius amplitude and strong vanishing theorems for vector
                bundles},
  note       = {With an appendix by Dennis S. Keeler},
  journal    = {Duke Math. J.},
  fjournal   = {Duke Mathematical Journal},
  volume     = {121},
  year       = {2004},
  number     = {2},
  pages      = {231--267},
  issn       = {0012-7094,1547-7398},
  mrclass    = {14F17},
  mrnumber   = {2034642},
  mrreviewer = {Karen\ E.\ Smith},
  doi        = {10.1215/S0012-7094-04-12122-0},
  url        = {https://doi.org/10.1215/S0012-7094-04-12122-0}
}

@article{Artin66,
  author     = {Artin, Michael},
  title      = {On isolated rational singularities of surfaces},
  journal    = {Amer. J. Math.},
  fjournal   = {American Journal of Mathematics},
  volume     = {88},
  year       = {1966},
  pages      = {129--136},
  issn       = {0002-9327},
  mrclass    = {14.18 (14.20)},
  mrnumber   = {199191},
  mrreviewer = {P. Du Val},
  doi        = {10.2307/2373050},
  url        = {https://doi.org/10.2307/2373050}
}

@misc{bauer2025equivarianttuttepolynomial,
  title         = {Equivariant Tutte Polynomial},
  author        = {Mario Bauer and Matěj Doležálek and Magdaléna Mišinová and Semen Słobodianiuk and Julian Weigert},
  year          = {2025},
  eprint        = {2312.00913},
  archiveprefix = {arXiv},
  primaryclass  = {math.AG},
  url           = {https://arxiv.org/abs/2312.00913}
}

@article{BergetFinkChow,
  title      = {Equivariant {C}how classes of matrix orbit closures},
  author     = {Berget, Andrew and Fink, Alex},
  year       = 2017,
  journal    = {Transform. Groups},
  volume     = 22,
  number     = 3,
  pages      = {631--643},
  doi        = {10.1007/s00031-016-9406-5},
  issn       = {1083-4362,1531-586X},
  url        = {https://doi.org/10.1007/s00031-016-9406-5},
  fjournal   = {Transformation Groups},
  mrclass    = {14C15 (05B35 14M15 14N15 52B40)},
  mrnumber   = 3682831,
  mrreviewer = {Nikita\ Semenov}
}

@misc{BergetFinkExternalActivity,
  title         = {The external activity complex of a pair of matroids},
  author        = {Andrew Berget and Alex Fink},
  year          = {2025},
  eprint        = {2412.11759},
  archiveprefix = {arXiv},
  primaryclass  = {math.CO},
  url           = {https://arxiv.org/abs/2412.11759}
}

@article{BergetFinkKtheory,
  title      = {Equivariant {$K$}-theory classes of matrix orbit closures},
  author     = {Berget, Andrew and Fink, Alex},
  year       = 2022,
  journal    = {Int. Math. Res. Not. IMRN},
  number     = 18,
  pages      = {14105--14133},
  doi        = {10.1093/imrn/rnab135},
  issn       = {1073-7928,1687-0247},
  url        = {https://doi.org/10.1093/imrn/rnab135},
  fjournal   = {International Mathematics Research Notices. IMRN},
  mrclass    = {19L47 (14C35)},
  mrnumber   = 4485953,
  mrreviewer = {Mee\ Seong\ Im}
}

@article{BergetFinkMatrixOrbit,
  title      = {Matrix orbit closures},
  author     = {Berget, Andrew and Fink, Alex},
  year       = 2018,
  journal    = {Beitr. Algebra Geom.},
  volume     = 59,
  number     = 3,
  pages      = {397--430},
  doi        = {10.1007/s13366-018-0402-x},
  issn       = {0138-4821,2191-0383},
  url        = {https://doi.org/10.1007/s13366-018-0402-x},
  fjournal   = {Beitr\"age zur Algebra und Geometrie. Contributions to Algebra and Geometry},
  mrclass    = {14M15 (14M12 52B40)},
  mrnumber   = 3844635,
  mrreviewer = {Zach\ Teitler}
}

@article{best,
  title     = {Tautological classes of matroids},
  author    = {Berget, Andrew and Eur, Christopher and Spink, Hunter and Tseng, Dennis},
  year      = 2023,
  journal   = {Invent. Math.},
  volume    = 233,
  number    = 2,
  pages     = {951--1039},
  doi       = {10.1007/s00222-023-01194-5},
  issn      = {0020-9910},
  url       = {https://doi.org/10.1007/s00222-023-01194-5},
  fjournal  = {Inventiones Mathematicae},
  mrclass   = {52B40 (14M25)},
  mrnumber  = 4607725,
  shorthand = {BEST23}
}

@article{bestCohomology,
  title    = {Cohomologies of tautological bundles of matroids},
  author   = {Eur, Christopher},
  year     = 2024,
  journal  = {Selecta Math. (N.S.)},
  volume   = 30,
  number   = 5,
  pages    = {Paper No. 85, 19},
  doi      = {10.1007/s00029-024-00979-7},
  issn     = {1022-1824,1420-9020},
  url      = {https://doi.org/10.1007/s00029-024-00979-7},
  fjournal = {Selecta Mathematica. New Series},
  mrclass  = {14F17 (05B35 05E14 14H10 14M27)},
  mrnumber = 4805085
}

@article{BinglinLi2018,
  title      = {Images of rational maps of projective spaces},
  author     = {Li, Binglin},
  year       = 2018,
  journal    = {Int. Math. Res. Not. IMRN},
  number     = 13,
  pages      = {4190--4228},
  doi        = {10.1093/imrn/rnx003},
  issn       = {1073-7928,1687-0247},
  url        = {https://doi.org/10.1093/imrn/rnx003},
  fjournal   = {International Mathematics Research Notices. IMRN},
  mrclass    = {14E05 (14D06)},
  mrnumber   = 3829180,
  mrreviewer = {Lei\ Zhang}
}

@incollection{Brieskorn,
  title     = {Sur les groupes de tresses [d'apr{\`e}s {V}. {I}. {A}rnol\textquotesingle d]},
  author    = {Brieskorn, Egbert},
  year      = 1973,
  booktitle = {S{\'e}minaire {B}ourbaki, 24{\`e}me ann{\'e}e (1971/1972)},
  publisher = {Springer, Berlin-New York},
  series    = {Lecture Notes in Math.},
  volume    = {Vol. 317},
  pages     = {Exp. No. 401, pp. 21--44}
}

@incollection{Brion2003,
  title      = {Multiplicity-free subvarieties of flag varieties},
  author     = {Brion, Michel},
  year       = 2003,
  booktitle  = {Commutative algebra ({G}renoble/{L}yon, 2001)},
  publisher  = {Amer. Math. Soc., Providence, RI},
  series     = {Contemp. Math.},
  volume     = 331,
  pages      = {13--23},
  doi        = {10.1090/conm/331/05900},
  isbn       = {0-8218-3233-6},
  url        = {https://doi.org/10.1090/conm/331/05900},
  mrclass    = {14M15 (14L30 14N15)},
  mrnumber   = 2011763,
  mrreviewer = {Christian\ Ohn}
}

@article{BrionTohoku,
  author     = {Brion, Michel},
  title      = {Vanishing theorems for {D}olbeault cohomology of log
                homogeneous varieties},
  journal    = {Tohoku Math. J. (2)},
  fjournal   = {The Tohoku Mathematical Journal. Second Series},
  volume     = {61},
  year       = {2009},
  number     = {3},
  pages      = {365--392},
  issn       = {0040-8735,2186-585X},
  mrclass    = {14M17 (14F17 14L30)},
  mrnumber   = {2568260},
  mrreviewer = {Nicolas\ Perrin},
  doi        = {10.2748/tmj/1255700200},
  url        = {https://doi.org/10.2748/tmj/1255700200}
}

@article{Broer,
  author   = {Broer, Abraham},
  title    = {A vanishing theorem for {Dolbeault} cohomology of homogeneous vector bundles},
  fjournal = {Journal f{\"u}r die Reine und Angewandte Mathematik},
  journal  = {J. Reine Angew. Math.},
  issn     = {0075-4102},
  volume   = {493},
  pages    = {153--169},
  year     = {1997},
  language = {English},
  doi      = {10.1515/crll.1997.493.153},
  url      = {https://eudml.org/doc/153961},
  zbmath   = {1091332},
  zbl      = {0911.14003}
}

@article{CarterLusztig,
  author     = {Carter, Roger W. and Lusztig, George},
  title      = {On the modular representations of the general linear and
                symmetric groups},
  journal    = {Math. Z.},
  fjournal   = {Mathematische Zeitschrift},
  volume     = {136},
  year       = {1974},
  pages      = {193--242},
  issn       = {0025-5874,1432-1823},
  mrclass    = {20G05 (20C20)},
  mrnumber   = {354887},
  mrreviewer = {James\ E.\ Humphreys},
  doi        = {10.1007/BF01214125},
  url        = {https://doi.org/10.1007/BF01214125}
}

@book{CLS,
  author    = {Cox, David A. and Little, John B. and Schenck, Henry K.},
  title     = {Toric varieties},
  fseries   = {Graduate Studies in Mathematics},
  series    = {Grad. Stud. Math.},
  issn      = {1065-7339},
  volume    = {124},
  isbn      = {978-0-8218-4819-7},
  year      = {2011},
  publisher = {Providence, RI: American Mathematical Society (AMS)},
  language  = {English},
  zbmath    = {5934474},
  zbl       = {1223.14001}
}

@article{CrapoTuttePolynomial,
  author   = {Crapo, H. H.},
  title    = {The {Tutte} polynomial},
  fjournal = {Aequationes Mathematicae},
  journal  = {Aequationes Math.},
  issn     = {0001-9054},
  volume   = {3},
  pages    = {211--229},
  year     = {1969},
  language = {English},
  doi      = {10.1007/BF01817442},
  url      = {https://eudml.org/doc/136026},
  zbmath   = {3315015},
  zbl      = {0197.50202}
}

@article{dcp,
  title      = {Wonderful models of subspace arrangements},
  author     = {De Concini, C. and Procesi, C.},
  year       = 1995,
  journal    = {Selecta Math. (N.S.)},
  volume     = 1,
  number     = 3,
  pages      = {459--494},
  doi        = {10.1007/BF01589496},
  issn       = {1022-1824,1420-9020},
  url        = {https://doi.org/10.1007/BF01589496},
  fjournal   = {Selecta Mathematica. New Series},
  mrclass    = {14D99 (32G13 52B30)},
  mrnumber   = 1366622,
  mrreviewer = {V.\ Leksin}
}

@article{deligneHodgeII,
  title      = {Th\'eorie de {H}odge. {II}},
  author     = {Deligne, Pierre},
  year       = 1971,
  journal    = {Inst. Hautes \'Etudes Sci. Publ. Math.},
  number     = 40,
  pages      = {5--57},
  issn       = {0073-8301,1618-1913},
  url        = {http://www.numdam.org/item?id=PMIHES_1971__40__5_0},
  fjournal   = {Institut des Hautes \'Etudes Scientifiques. Publications Math\'ematiques},
  mrclass    = {14C30 (14F15)},
  mrnumber   = 498551,
  mrreviewer = {J.\ H. M. Steenbrink}
}

@article{DeligneIllusie,
  title      = {Rel\`evements modulo {$p^2$} et d\'ecomposition du complexe de de {R}ham},
  author     = {Deligne, Pierre and Illusie, Luc},
  year       = 1987,
  journal    = {Invent. Math.},
  volume     = 89,
  number     = 2,
  pages      = {247--270},
  doi        = {10.1007/BF01389078},
  issn       = {0020-9910,1432-1297},
  url        = {https://doi.org/10.1007/BF01389078},
  fjournal   = {Inventiones Mathematicae},
  mrclass    = {14F40 (14C30)},
  mrnumber   = 894379,
  mrreviewer = {Thomas\ Zink}
}

@article{EGAIII1,
  title     = {
               \'El\'ements de g\'eom\'etrie alg\'ebrique. {III}. \'Etude cohomologique des faisceaux
               coh\'erents. {I}.
               },
  author    = {Grothendieck, A.},
  year      = 1961,
  journal   = {Inst. Hautes \'Etudes Sci. Publ. Math.},
  number    = 11,
  pages     = 167,
  issn      = {0073-8301,1618-1913},
  url       = {http://www.numdam.org/item?id=PMIHES_1961__11__167_0},
  fjournal  = {Institut des Hautes \'Etudes Scientifiques. Publications Math\'ematiques},
  mrclass   = {14.55},
  mrnumber  = 217085,
  shorthand = {EGA III\textsubscript{1}}
}

@article{EinLazarsfeldKoszulCohomologyArbitraryDimension,
  author     = {Ein, Lawrence and Lazarsfeld, Robert},
  title      = {Syzygies and {K}oszul cohomology of smooth projective
                varieties of arbitrary dimension},
  journal    = {Invent. Math.},
  fjournal   = {Inventiones Mathematicae},
  volume     = {111},
  year       = {1993},
  number     = {1},
  pages      = {51--67},
  issn       = {0020-9910,1432-1297},
  mrclass    = {13D02 (14F17 14J60)},
  mrnumber   = {1193597},
  mrreviewer = {Gary\ P.\ Kennedy},
  doi        = {10.1007/BF01231279},
  url        = {https://doi.org/10.1007/BF01231279}
}

@article{ESV92,
  title      = {Cohomology of local systems on the complement of hyperplanes},
  author     = {Esnault, H\'el\`ene and Schechtman, Vadim and Viehweg, Eckart},
  year       = 1992,
  journal    = {Invent. Math.},
  volume     = 109,
  number     = 3,
  pages      = {557--561},
  doi        = {10.1007/BF01232040},
  issn       = {0020-9910,1432-1297},
  url        = {https://doi.org/10.1007/BF01232040},
  fjournal   = {Inventiones Mathematicae},
  mrclass    = {32S40 (52B30)},
  mrnumber   = 1176205,
  mrreviewer = {Hiroaki\ Terao}
}

@misc{eur2024ktheoreticpositivitymatroids,
  title         = {K-theoretic positivity for matroids},
  author        = {Christopher Eur and Matt Larson},
  year          = {2024},
  eprint        = {2311.11996},
  archiveprefix = {arXiv},
  primaryclass  = {math.AG},
  url           = {https://arxiv.org/abs/2311.11996}
}

@misc{eur2025vanishingtheoremsmatroids,
      title={Vanishing theorems for combinatorial geometries}, 
      author={Christopher Eur and Alex Fink and Matt Larson},
      year={2025},
      eprint={2510.05207},
      archivePrefix={arXiv},
      primaryClass={math.AG},
      url={https://arxiv.org/abs/2510.05207}, 
}

@misc{fink2024omegainvariantmatroid,
  title         = {The omega invariant of a matroid},
  author        = {Alex Fink and Kris Shaw and David E Speyer},
  year          = {2024},
  eprint        = {2411.19521},
  archiveprefix = {arXiv},
  primaryclass  = {math.CO},
  url           = {https://arxiv.org/abs/2411.19521}
}

@article{FinkSpeyerKClasses,
  author     = {Fink, Alex and Speyer, David E.},
  title      = {{$K$}-classes for matroids and equivariant localization},
  journal    = {Duke Math. J.},
  fjournal   = {Duke Mathematical Journal},
  volume     = {161},
  year       = {2012},
  number     = {14},
  pages      = {2699--2723},
  issn       = {0012-7094,1547-7398},
  mrclass    = {52B40 (05B35 14C35 14M15)},
  mrnumber   = {2993138},
  mrreviewer = {Harry\ Tamvakis},
  doi        = {10.1215/00127094-1813296},
  url        = {https://doi.org/10.1215/00127094-1813296}
}

@article{GGMS,
  author    = {Gel'fand, I. M. and Goresky, R. Mark and MacPherson, Robert D. and Serganova, V. V.},
  title     = {Combinatorial geometries, convex polyhedra, and {Schubert} cells},
  fjournal  = {Advances in Mathematics},
  journal   = {Adv. Math.},
  issn      = {0001-8708},
  volume    = {63},
  pages     = {301--316},
  year      = {1987},
  language  = {English},
  doi       = {10.1016/0001-8708(87)90059-4},
  zbmath    = {4009341},
  zbl       = {0622.57014},
  shorthand = {GGMS87}
}

@article{GrosseKloenne05,
  title    = {Integral structures in the {{\(p\)}}-adic holomorphic discrete series},
  author   = {Gro{\ss}e-Kl{\"o}nne, Elmar},
  year     = 2005,
  journal  = {Represent. Theory},
  volume   = 9,
  pages    = {354--384},
  doi      = {10.1090/S1088-4165-05-00259-1},
  issn     = {1088-4165},
  fjournal = {Representation Theory},
  language = {English},
  zbmath   = 2168573,
  zbl      = {1068.14025}
}

@article{HackingKeelTevelev,
  author     = {Hacking, Paul and Keel, Sean and Tevelev, Jenia},
  title      = {Compactification of the moduli space of hyperplane
                arrangements},
  journal    = {J. Algebraic Geom.},
  fjournal   = {Journal of Algebraic Geometry},
  volume     = {15},
  year       = {2006},
  number     = {4},
  pages      = {657--680},
  issn       = {1056-3911,1534-7486},
  mrclass    = {14D20 (14D06)},
  mrnumber   = {2237265},
  mrreviewer = {Michael\ A.\ van Opstall},
  doi        = {10.1090/S1056-3911-06-00445-0},
  url        = {https://doi.org/10.1090/S1056-3911-06-00445-0}
}

@book{Har,
  title      = {Algebraic geometry},
  author     = {Hartshorne, Robin},
  year       = 1977,
  publisher  = {Springer-Verlag, New York-Heidelberg},
  series     = {Graduate Texts in Mathematics},
  volume     = {No. 52},
  pages      = {xvi+496},
  isbn       = {0-387-90244-9},
  mrclass    = {14-01},
  mrnumber   = 463157,
  mrreviewer = {Robert\ Speiser}
}

@incollection{KapranovChowQuotient,
  author     = {Kapranov, M. M.},
  title      = {Chow quotients of {G}rassmannians. {I}},
  booktitle  = {I. {M}. {G}el'fand {S}eminar},
  series     = {Adv. Soviet Math.},
  volume     = {16, Part 2},
  pages      = {29--110},
  publisher  = {Amer. Math. Soc., Providence, RI},
  year       = {1993},
  isbn       = {0-8218-4119-X},
  mrclass    = {14L30 (14M15 52B20 52B40)},
  mrnumber   = {1237834},
  mrreviewer = {Michel\ Brion}
}

@article{Kempf,
  author     = {Kempf, George R.},
  title      = {On the collapsing of homogeneous bundles},
  journal    = {Invent. Math.},
  fjournal   = {Inventiones Mathematicae},
  volume     = {37},
  year       = {1976},
  number     = {3},
  pages      = {229--239},
  issn       = {0020-9910,1432-1297},
  mrclass    = {14M15 (14F05 14M10)},
  mrnumber   = {424841},
  mrreviewer = {S.\ I.\ Gel'fand},
  doi        = {10.1007/BF01390321},
  url        = {https://doi.org/10.1007/BF01390321}
}

@book{KKMS,
  title      = {Toroidal embeddings. {I}},
  author     = {Kempf, G. and Knudsen, Finn Faye and Mumford, D. and Saint-Donat, B.},
  year       = 1973,
  publisher  = {Springer-Verlag, Berlin-New York},
  series     = {Lecture Notes in Mathematics},
  volume     = {Vol. 339},
  pages      = {viii+209},
  mrclass    = {14E15 (14D20 14E05 14M20 20G15)},
  mrnumber   = 335518,
  mrreviewer = {G.\ Harder},
  shorthand  = {KKMS73}
}

@article{LangerDeligneLusztig,
  author     = {Langer, Adrian},
  title      = {Birational geometry of compactifications of {D}rinfeld
                half-spaces over a finite field},
  journal    = {Adv. Math.},
  fjournal   = {Advances in Mathematics},
  volume     = {345},
  year       = {2019},
  pages      = {861--908},
  issn       = {0001-8708},
  mrclass    = {14G15 (14E05 14J26)},
  mrnumber   = {3902334},
  mrreviewer = {Christian Liedtke},
  doi        = {10.1016/j.aim.2019.01.031},
  url        = {https://doi.org/10.1016/j.aim.2019.01.031}
}

@book{LazarsfeldPositivityI,
  title      = {Positivity in algebraic geometry. {I}},
  author     = {Lazarsfeld, Robert},
  year       = 2004,
  publisher  = {Springer-Verlag, Berlin},
  series     = {
                Ergebnisse der Mathematik und ihrer Grenzgebiete. 3. Folge. A Series of Modern Surveys in
                Mathematics [Results in Mathematics and Related Areas. 3rd Series. A Series of Modern Surveys in
                Mathematics]
                },
  volume     = 48,
  pages      = {xviii+387},
  doi        = {10.1007/978-3-642-18808-4},
  isbn       = {3-540-22533-1},
  url        = {https://doi.org/10.1007/978-3-642-18808-4},
  note       = {Classical setting: line bundles and linear series},
  mrclass    = {14-02 (14C20)},
  mrnumber   = 2095471,
  mrreviewer = {Mihnea\ Popa}
}

@book{LazarsfeldPositivityII,
  author     = {Lazarsfeld, Robert},
  title      = {Positivity in algebraic geometry. {II}},
  series     = {Ergebnisse der Mathematik und ihrer Grenzgebiete. 3. Folge. A
                Series of Modern Surveys in Mathematics [Results in
                Mathematics and Related Areas. 3rd Series. A Series of Modern
                Surveys in Mathematics]},
  volume     = {49},
  note       = {Positivity for vector bundles, and multiplier ideals},
  publisher  = {Springer-Verlag, Berlin},
  year       = {2004},
  pages      = {xviii+385},
  isbn       = {3-540-22534-X},
  mrclass    = {14-02 (14C20 14F05 14F17)},
  mrnumber   = {2095472},
  mrreviewer = {Mihnea\ Popa},
  doi        = {10.1007/978-3-642-18808-4},
  url        = {https://doi.org/10.1007/978-3-642-18808-4}
}

@misc{levineOSalgebra,
  author = {Levine, Lionel},
  title  = {Orlik--Solomon algebras of hyperplane arrangements},
  note   = {\url{https://pi.math.cornell.edu/~levine/orliksolomonfinal.pdf}},
  year   = {2004}
}

@article{LosevManin,
  title      = {New moduli spaces of pointed curves and pencils of flat connections},
  author     = {Losev, A. and Manin, Y.},
  year       = 2000,
  journal    = {Michigan Math. J.},
  volume     = 48,
  pages      = {443--472},
  doi        = {10.1307/mmj/1030132728},
  issn       = {0026-2285},
  url        = {https://doi.org/10.1307/mmj/1030132728},
  note       = {Dedicated to William Fulton on the occasion of his 60th birthday},
  fjournal   = {Michigan Mathematical Journal},
  mrclass    = {14N35 (14H10 53D45)},
  mrnumber   = 1786500,
  mrreviewer = {Andrew Kresch}
}

@article{ManivelVanishing,
  title      = {Vanishing theorems for ample vector bundles},
  author     = {Manivel, Laurent},
  year       = 1997,
  journal    = {Invent. Math.},
  volume     = 127,
  number     = 2,
  pages      = {401--416},
  doi        = {10.1007/s002220050126},
  issn       = {0020-9910,1432-1297},
  url        = {https://doi.org/10.1007/s002220050126},
  fjournal   = {Inventiones Mathematicae},
  mrclass    = {14F17 (14J60)},
  mrnumber   = 1427625,
  mrreviewer = {G.\ Horrocks}
}

@book{OrlikTerao,
  author     = {Orlik, Peter and Terao, Hiroaki},
  title      = {Arrangements of hyperplanes},
  series     = {Grundlehren der mathematischen Wissenschaften [Fundamental
                Principles of Mathematical Sciences]},
  volume     = {300},
  publisher  = {Springer-Verlag, Berlin},
  year       = {1992},
  pages      = {xviii+325},
  isbn       = {3-540-55259-6},
  mrclass    = {52B30 (14F35 20F36 20F55 32S25 57N65)},
  mrnumber   = {1217488},
  mrreviewer = {Michel\ Yves\ Jambu},
  doi        = {10.1007/978-3-662-02772-1},
  url        = {https://doi.org/10.1007/978-3-662-02772-1}
}

@book{Oxley,
  author    = {Oxley, James G.},
  title     = {Matroid theory},
  edition   = {2nd ed.},
  fseries   = {Oxford Graduate Texts in Mathematics},
  series    = {Oxf. Grad. Texts Math.},
  volume    = {21},
  isbn      = {978-0-19-856694-6; 978-0-19-960339-8},
  year      = {2011},
  publisher = {Oxford: Oxford University Press},
  language  = {English},
  zbmath    = {5873618},
  zbl       = {1254.05002}
}

@incollection{Raynaud,
  author     = {Raynaud, M.},
  title      = {Contre-exemple au ``vanishing theorem''\ en caract\'eristique
                {$p>0$}},
  booktitle  = {C. {P}. {R}amanujam---a tribute},
  series     = {Tata Inst. Fundam. Res. Stud. Math.},
  volume     = {8},
  pages      = {273--278},
  publisher  = {Springer, Berlin-New York},
  year       = {1978},
  isbn       = {3-540-08770-2},
  mrclass    = {14F12 (14J10 14J25)},
  mrnumber   = {541027},
  mrreviewer = {Daniel\ Comenetz}
}

@article {EHLstella,
    AUTHOR = {Eur, Christopher and Huh, June and Larson, Matt},
     TITLE = {Stellahedral geometry of matroids},
   JOURNAL = {Forum Math. Pi},
  FJOURNAL = {Forum of Mathematics. Pi},
    VOLUME = {11},
      YEAR = {2023},
     PAGES = {Paper No. e24, 48},
   MRCLASS = {14M25 (05B35 14N20 52B40 52C35)},
  MRNUMBER = {4653766},
MRREVIEWER = {Paolo Aluffi},
       DOI = {10.1017/fmp.2023.24},
       URL = {https://doi.org/10.1017/fmp.2023.24},
}

@book {Welsh,
    AUTHOR = {Welsh, D. J. A.},
     TITLE = {Matroid theory},
    SERIES = {L. M. S. Monographs, No. 8},
 PUBLISHER = {Academic Press [Harcourt Brace Jovanovich, Publishers],
              London-New York},
      YEAR = {1976},
     PAGES = {xi+433},
   MRCLASS = {05B35},
  MRNUMBER = {427112},
MRREVIEWER = {W. T. Tutte},
}

@misc{stacks-project,
  author       = {The {Stacks project authors}},
  title        = {The Stacks project},
  howpublished = {\url{https://stacks.math.columbia.edu}},
  year         = {2025},
  shorthand    = {Stacks}
}

@book{Weyman03,
  title      = {Cohomology of vector bundles and syzygies},
  author     = {Weyman, Jerzy},
  year       = 2003,
  publisher  = {Cambridge University Press, Cambridge},
  series     = {Cambridge Tracts in Mathematics},
  volume     = 149,
  pages      = {xiv+371},
  doi        = {10.1017/CBO9780511546556},
  isbn       = {0-521-62197-6},
  url        = {https://doi.org/10.1017/CBO9780511546556},
  mrclass    = {13D02 (13C40 14L30 14M12 14M15 14M17 20G15)},
  mrnumber   = 1988690,
  mrreviewer = {Laurent\ Manivel}
}

@article{WhiteBasisMonomial,
  author     = {White, Neil L.},
  title      = {The basis monomial ring of a matroid},
  journal    = {Advances in Math.},
  fjournal   = {Advances in Mathematics},
  volume     = {24},
  year       = {1977},
  number     = {3},
  pages      = {292--297},
  issn       = {0001-8708},
  mrclass    = {05B35},
  mrnumber   = {437366},
  mrreviewer = {Richard\ Stanley},
  doi        = {10.1016/0001-8708(77)90060-3},
  url        = {https://doi.org/10.1016/0001-8708(77)90060-3}
}

@article{YuzvinskiiBook,
  author     = {Yuzvinski\u i, S.},
  title      = {Orlik-{S}olomon algebras in algebra and topology},
  journal    = {Uspekhi Mat. Nauk},
  fjournal   = {Uspekhi Matematicheskikh Nauk},
  volume     = {56},
  year       = {2001},
  number     = {2(338)},
  pages      = {87--166},
  issn       = {0042-1316,2305-2872},
  mrclass    = {14N20 (32S22 52C35)},
  mrnumber   = {1859708},
  mrreviewer = {Nguyen\ Viet\ Dung},
  doi        = {10.1070/RM2001v056n02ABEH000383},
  url        = {https://doi.org/10.1070/RM2001v056n02ABEH000383}
}
